\documentclass[11pt,showkeys]{amsart}

\usepackage{comment}
\usepackage{bm}
\usepackage{tikz-cd}
\usepackage{amsmath}
\usepackage{fourier}
\usepackage{amssymb}
\usepackage{amscd}
\usepackage{amsthm}
\usepackage[centertags]{amsmath}
\usepackage{amsfonts}
\usepackage{newlfont}
\usepackage{graphicx}
\usepackage{xcolor}
\usepackage{amsfonts, amssymb}
\usepackage{mathrsfs}
\usepackage{latexsym}
\usepackage{tikz}
\usepackage{verbatim}
\usepackage[normalem]{ulem}
\usepackage[all]{xy}
\usepackage{enumitem}

\usepackage{graphicx}
\usepackage{amsmath, amssymb, amsthm}
\usepackage{braket}
\usepackage{geometry}
\usepackage{mathtools}
\usepackage{enumitem}
\usepackage{bbm}
\usepackage{tikz}
\usepackage{tikz-cd}
\usetikzlibrary{positioning, arrows.meta}
\usepackage{comment}
\title[Projection Constants of Polynomial Spaces via Representation Theory]{Projection Constants of Polynomial Spaces via Representation Theory}
\author{}
\date{}


\newcommand{\sboxtimes}
{\mathbin{\vcenter{\hbox{\scalebox{0.65}{$\boxtimes$}}}}}

\newtheorem{theorem}{Theorem}[section]
\newtheorem{lemma}[theorem]{Lemma}

\newtheorem{remark}[theorem]{Remark}
\newtheorem{corollary}[theorem]{Corollary}
\newtheorem{proposition}[theorem]{Proposition}

\newcommand{\hbullet}{\mathbin{\vcenter{\hbox{\scalebox{0.7}{$\bullet$}}}}}

\usepackage[hidelinks]{hyperref}
\hypersetup{
  pdftitle={Projection Constants of Polynomial Spaces via Representation Theory},
  pdfauthor={A. Defant, D. Galicer, M. Mansilla, M. Mastylo, S. Muro, P. Zadunaisky}
}

 \author[Defant]{A.~Defant}
 \address{%
 Institut f\"{u}r Mathematik,
 Carl von Ossietzky Universit\"at,
 26111 Oldenburg,
 Germany}
 \email{defant@mathematik.uni-oldenburg.de}

 \author[Galicer]{D.~Galicer}
 \address{Universidad Torcuato Di Tella. Departamento de Matem\'atica y Estad\'istica.  IMAS-CONICET.
Av. Figueroa Alcorta 7350 (1428), Buenos Aires, Argentina.}
 \email{daniel.galicer$@$utdt.edu}

 \author[Mansilla]{M.~Mansilla}
 \address{Departamento de Matem\'{a}tica,
 Facultad de Cs. Exactas y Naturales, Universidad de Buenos Aires and IAM-CONICET. Saavedra 15 (C1083ACA) C.A.B.A., Argentina}
 \email{mmansilla$@$dm.uba.ar}

 \author[Masty{\l}o]{M.~Masty{\l}o}
 \address{Faculty of Mathematics and Computer Science, Adam Mickiewicz University, Pozna{\'n}, 61-614 Pozna{\'n}, Poland}
 \email{mieczyslaw.mastylo$@$amu.edu.pl}

 \author[Muro]{S.~Muro}
\address{FCEIA, Universidad Nacional de Rosario and CIFASIS, CONICET, Ocampo $\&$ Esmeralda, S2000 Rosario, Argentina}
 \email{muro$@$cifasis-conicet.gov.ar}

 \author[Zadunaisky]{P.~Zadunaisky}
 \address{Departamento de Matem\'{a}tica,
 Facultad de Cs. Exactas y Naturales, Universidad de Buenos Aires and IMAS-CONICET. C.A.B.A., Argentina}
 \email{pzadub$@$dm.uba.ar}

\date{}

\begin{document}

\begin{abstract}
We develop a representation-theoretic framework for computing projection
constants of finite-dimensional subspaces of \(C(K)\), where \(K\) is a
compact homogeneous \(G\)-space. Within this framework, we characterize
the finite-dimensional \(G\)-invariant subspaces for which the
\(G\)-equivariant projection is unique. These are precisely the finite
orthogonal sums of full isotypic components in the Peter--Weyl
decomposition of \(L_2(K)\). For such subspaces, averaging shows that
this unique \(G\)-equivariant projection is minimal; it is the restriction
to \(C(K)\) of the \(L_2(K)\)-orthogonal projection. Consequently, their
projection constants are given by the \(L_1\)-norm of an explicit
reproducing-kernel slice.
We apply this method to spaces of \(d\)-homogeneous polynomials in high
dimension. The examples range from Fourier analysis on the torus,
through spherical harmonic analysis on real and complex Euclidean
spheres, to genuinely noncommutative harmonic analysis on the unitary
group, where representation theory becomes indispensable. For all
these families, we obtain precise high-dimensional asymptotics at the
square-root-of-dimension scale.
\end{abstract}

\maketitle

\tableofcontents

\newpage


\section*{Introduction}

Projection constants measure how well a finite-dimensional Banach space
can be complemented in ambient spaces. If \(X\) is a subspace of a Banach
space \(Y\), its relative projection constant in \(Y\) is
\[
        \boldsymbol\lambda(X,Y)
        :=
        \inf\{ \,\|P\|:P\in\mathcal L(Y),\ P^2=P,\ P(Y)=X \, \}.
\]
The absolute projection constant of \(X\) is
\[
        \boldsymbol\lambda(X)
        :=
        \sup \boldsymbol\lambda(X,Y),
\]
where the supremum is taken over all Banach spaces \(Y\) containing \(X\)
isometrically. Thus \(\boldsymbol\lambda(X)\) is an isometric invariant of \(X\), measuring its
worst-case complementability.
For many concrete families of finite-dimensional Banach spaces, either exact
formulas for their projection constants or precise asymptotic estimates in
terms of the dimension are known.

A standard theorem asserts that whenever a finite-dimensional Banach space \(X\) is realized isometrically as
a subspace of \(C(K)\), for a compact Hausdorff space $K$, its absolute projection constant coincides with
its relative projection constant in \(C(K)\); see \cite[Theorem~III.B.5]{wojtaszczyk1996banach}:
\[
        \boldsymbol\lambda(X)
        =
        \boldsymbol\lambda(X,C(K)).
\]
Moreover, the infimum in the definition of \(\boldsymbol\lambda(X,C(K))\) is
attained. Thus there exists a projection
\[
        P:C(K)\longrightarrow X
\]
such that
\[
        \|P\|
        =
        \boldsymbol\lambda(X,C(K))
        =
        \boldsymbol\lambda(X).
\]
Such a projection is called a minimal projection onto \(X\).
Thus the computation of absolute projection constants may be carried out
inside suitable spaces of continuous functions.

The classical general bounds
are
$$
1
\leq
\boldsymbol\lambda(X)
\leq
\sqrt{n},$$
for every \(n\)-dimensional Banach space \(X\).
The lower bound is attained by \(\ell_\infty^n\), since the Hahn--Banach
theorem gives
$
\boldsymbol\lambda(\ell_\infty^n)=1,
$
whereas the upper bound is the celebrated theorem of Kadets and Snobar.

In this paper, we study natural finite-dimensional spaces whose geometry is
governed by a rich symmetry structure. More specifically, we focus on spaces
of polynomials on compact homogeneous spaces, where the underlying group
action can be exploited to identify the canonical projection.

Our goals are twofold. First, we develop a representation-theoretic framework
for determining when symmetry uniquely identifies the minimal projection.
Second, we apply this framework to several families of polynomial spaces, obtaining explicit integral formulas for their projection constants. These formulas then serve as the starting point for deriving precise high-dimensional asymptotics by probabilistic methods. Although
the applications involve different forms of harmonic analysis, they all arise
from the same underlying mechanism.

\medskip

\textbf{The representation-theoretic framework.}
The first part of the paper develops a representation-theo\-retic symmetry
principle for computing projection constants in spaces of continuous
functions.

Let \(K\) be a compact homogeneous \(G\)-space with invariant probability
measure \(\mu_K\), and let \(S\subset C(K)\) be a finite-dimensional
\(G\)-invariant subspace. We say that \(S\) is accessible when the
\(G\)-equivariant projection onto \(S\) is unique. In this case, averaging
shows that the restriction to \(C(K)\) of the orthogonal projection
\(\mathbf\pi_S:L_2(K,\mu_K)\to S\) is minimal. Consequently, if \(\mathbf k_S\) denotes the
reproducing kernel of \(S\), then
$$
\boldsymbol\lambda(S)
=
\int_K |\mathbf k_S(x_0,y)|\,d\mu_K(y),
\qquad x_0\in K.
$$
Thus the minimization problem is replaced by a structural and analytic
program: identify the relevant symmetry, determine accessibility, compute
the reproducing kernel, and evaluate the \(L_1\)-norm of one of its slices.

The Peter--Weyl decomposition of \(L_2(K,\mu_K)\), together with its
uniform-density consequence in \(C(K)\), provides the natural framework
for carrying out this program. It describes invariant subspaces in terms of isotypic
components and reveals that multiplicities are precisely what distinguish
invariance from accessibility. The abstract theory shows that accessible subspaces are precisely the finite sums of full isotypic components.
In particular, for multiplicity-free actions, every finite-dimensional
invariant subspace is accessible.

This characterization leads to a general integral formula for projection
constants in terms of the kernel data. For compact
abelian groups, it reduces to the \(L_1\)-norm of a Fourier sum. For compact
groups acting on themselves, it becomes the \(L_1\)-norm of a weighted sum
of irreducible characters. The contrast between the left-regular and
biregular actions further shows why additional symmetry is needed in the
noncommutative setting.

\medskip

\noindent

\textbf{Normalized projection constants of polynomial spaces.}
The second part of the paper applies the abstract machinery to polynomial
spaces in high dimension. Although all the examples follow the same general
scheme, each of them involves a different form of harmonic analysis.

Let $n\in\mathbb N$ and let
$X_n=(\mathbb K^n,\|\cdot\|)$, where
$\mathbb K\in\{\mathbb R,\mathbb C\}$. For $d\in\mathbb N_0$, we
denote by
$\mathcal P_d(X_n)=\mathcal P_d(X_n;\mathbb K)$
the Banach space of all scalar-valued $d$-homogeneous polynomials on
$X_n$. Every $P\in\mathcal P_d(X_n)$ has a unique representation
\[
P(z)=\sum_{|\alpha|=d}c_\alpha z^\alpha, \qquad c_\alpha\in\mathbb K,
\]
and $\mathcal P_d(X_n)$ is equipped with the norm
\[
\|P\|=\sup_{\|z\|\leq 1}|P(z)|=\sup_{\|z\|=1}|P(z)|,
\]
where the last equality follows from homogeneity.

Our concrete applications fall into three principal settings. For
$\mathcal P_d(\ell_\infty^n(\mathbb C))$, the relevant compact space
is the torus $\mathbb T^n$, and the associated reproducing kernels are
Fourier sums. For
$\mathcal P_d(\ell_2^n(\mathbb K))$, where
$\mathbb K\in\{\mathbb R,\mathbb C\}$, one passes to the corresponding
Euclidean sphere, and the kernels are described in terms of spherical
harmonics. Finally, in the complex setting, the relevant symmetry for
$\mathcal P_d\bigl(\mathcal L(\ell_2^n(\mathbb C))\bigr)$ is genuinely
noncommutative: the compact model is the unitary group $\mathcal U_n$,
and the corresponding decomposition is expressed in terms of matrix
coefficients of irreducible unitary representations.

Throughout, the degree $d$ is fixed and the ambient dimension tends to
infinity. The natural normalization is the Kadets--Snobar scale, namely the
square root of the dimension of the polynomial space. More precisely, for each of the families of Banach spaces $X_n$ described above, with $n$ denoting the ambient dimension, we prove that
$$
\frac{\boldsymbol\lambda\bigl(\mathcal P_d(X_n)\bigr)}
{\sqrt{\dim \mathcal P_d(X_n)}}
\longrightarrow c_d,
\qquad n\to\infty,
$$
where the constant $c_d>0$ is determined explicitly.
 Thus, for these spaces
of $d$-homogeneous polynomials, the Kadets--Snobar estimate is
asymptotically sharp, and its precise leading constant is identified.

On the torus, Fourier sums yield new precise high-dimensional asymptotics for
homogeneous and tetrahedral polynomials on $\ell_\infty^n(\mathbb C)$. On real and complex Euclidean spheres, the representation-theoretic
framework provides unified characterizations of invariant and
accessible subspaces. In the real
case, it also leads to a new asymptotic formula involving Hermite polynomials,
whereas in the complex case it naturally recovers the projection-constant
formula and asymptotic limit obtainable from the Ryll--Wojtaszczyk formula.
In the genuinely noncommutative setting of the unitary group
\(\mathcal U_n\), where representation theory is essential, we obtain
a complete Peter--Weyl decomposition of the polynomial spaces, an
explicit reproducing kernel and integral formula for their projection
constants, and precise asymptotics for every fixed degree.
The degree-one case recovers and conceptually explains the
previously known result for the trace class.

Despite the differences among these harmonic models, a common Gaussian
structure emerges. In the complex settings of the torus, the complex
Euclidean sphere, and the unitary group, the normalized projection constants
converge to
$$
\frac{\mathbb E |W|^d}{\sqrt{d!}}
=
\frac{\Gamma(1+d/2)}{\sqrt{d!}},
$$
where $W$ is a standard complex Gaussian normalized by
$\mathbb E|W|^2=1$. In the real Hilbertian setting, the corresponding limit is
$$
\frac{\mathbb E |\operatorname{He}_d(Z)|}{\sqrt{d!}},
$$
where $Z$ is a standard real Gaussian and $\operatorname{He}_d$ is the
$d$-th probabilists' Hermite polynomial.

These examples reveal a common mechanism across abelian, spherical, and
genuinely noncommutative settings: representation theory identifies the
canonical projection and its reproducing kernel, while high-dimensional
probability determines the asymptotic behavior of its $L_1$-norm.

\medskip
\textbf{A note on the exposition.} Although the problems considered in this paper arise naturally in the study of projection constants in Banach space theory, their treatment relies on tools from the representation theory of compact groups that may be less familiar to specialists in functional analysis. We have therefore included a self-contained account of the representation-theoretic framework needed for the applications. Our aim is not only to provide the necessary background, but also to make clear how the structure of the group action enters the identification and computation of minimal projections.

We follow the standard conventions of Banach space theory \cite{wojtaszczyk1996banach}
and representation theory \cite{folland2016course}. We have aimed to make the exposition as self-contained as possible, and many
of the relevant definitions are introduced locally when they are first needed.
\medskip

\section{Projection constants in homogeneous $G$-spaces}
\label{sec:Projection constants}

Throughout Subsections~\ref{sec: homogeneity} and~\ref{The core of the strategy}, we allow both real- and
complex-valued function spaces. When appropriate, we use the notation
$C(K;\mathbb R)$ and $C(K;\mathbb C)$ to distinguish between the two
settings. From Subsection~\ref{sec:The Peter--Weyl decomposition}
through the end of Section~1, all function spaces and representations
are considered over $\mathbb C$. In particular, $C(K)$ denotes
$C(K;\mathbb C)$. Real-valued function spaces arising in the
applications are treated by first passing to their natural
complexifications and then restricting to the corresponding real
forms.

\subsection{Homogeneity, invariance and accessibility}\label{sec: homogeneity}
We begin with the abstract framework. The basic objects are compact
homogeneous \(G\)-spaces and finite-dimensional invariant subspaces of
\(C(K)\). The key notion is accessibility: the condition that symmetry
already determines the minimal projection onto the subspace. Throughout, projections between Banach spaces are understood to be bounded and linear.

\subsubsection{Homogeneous $G$-spaces}

Let $G$ be a compact topological group. By a compact \emph{$G$-space} we mean a
compact  Hausdorff space $K$ endowed with a continuous left action of $G$,
that is, a continuous map
$$
G\times K\longrightarrow K,\qquad (g,x)\longmapsto g\hbullet x,
$$
satisfying $e\hbullet x=x$ and $(gh)\hbullet x=g\hbullet(h\hbullet x)$ for all
$g,h\in G$ and $x\in K$, where $e$ is the identity element of $G$.

A compact $G$-space $K$ is said to be \emph{homogeneous} if the action of $G$
on $K$ is transitive, that is, if for every $x,y\in K$ there exists $g\in G$
such that
$g\hbullet x=y.$
Equivalently, for every $x\in K$, its orbit
$$
G\hbullet x:=\{g\hbullet x:g\in G\}
$$
coincides with the whole space $K$.

  For example, the real sphere
$\mathbb S_{\mathbb R}^{n-1}\subset\mathbb R^n$ is homogeneous under the orthogonal group
$\mathcal O_n$, whereas the complex sphere
$\mathbb S_{\mathbb C}^{n-1}\subset\mathbb C^n$ is homogeneous under the unitary group
$\mathcal U_n$.

When $K=G$ is itself a compact group, two important examples arise naturally. The first is the \emph{left-regular action} of $G$ on itself, given by
$$
g\hbullet x=gx,\qquad g, x\in G.
$$
The second is the \emph{biregular action} of $G\times G$ on $G$, defined by
$$
(g_1,g_2)\hbullet x=g_1xg_2^{-1},\qquad (g_1,g_2) \in G \times G ,\; x\in G.
$$
Both actions are transitive and therefore turn $K$ into a homogeneous space.

\subsubsection{The orbit map and Haar measure}
Let \(K\) be a homogeneous \(G\)-space and fix \(x_0\in K\). The associated
\emph{orbit map} is defined by
\[
        \Phi_{x_0}\colon G\longrightarrow K,
        \qquad
        g\longmapsto g\hbullet x_0 .
\]
Since the action is transitive, \(\Phi_{x_0}\) is surjective. Moreover, the
orbit map is continuous and open. Let \(\mu_G\) denote the normalized Haar
measure on \(G\); since \(G\) is compact, \(\mu_G\) is both left- and
right-invariant.

The invariant probability measure on \(K\) is defined as the pushforward
of \(\mu_G\) under the orbit map. In other words, for every Borel set
\(A\subset K\),
\[
        \mu_K(A):=\mu_G\bigl(\Phi_{x_0}^{-1}(A)\bigr).
\]
This definition does not depend on the choice of the base point $x_0$. Indeed,
if $x_1=h\hbullet x_0$ for some $h\in G$, then
\[
        \Phi_{x_1}(g)=g\hbullet x_1=g\hbullet(h\hbullet x_0)=(gh)\hbullet x_0
        =\Phi_{x_0}(gh).
\]
Thus $\Phi_{x_1}=\Phi_{x_0}\circ R_h$, where $R_h(g)=gh$ is the right translation
by $h$. Since $\mu_G$ is right-invariant, the pushforwards of $\mu_G$ under
$\Phi_{x_0}$ and $\Phi_{x_1}$ coincide.

The measure $\mu_K$ is the unique regular Borel probability measure on $K$ invariant under
the action of $G$.
It has full support: every nonempty open subset of $K$ has a nonempty open preimage under the orbit map, and every nonempty open subset of $G$ has positive Haar measure. In particular, the natural map from $C(K)$ into $L_2(K,\mu_K)$ is injective.
For further information  see \cite{diestel2014joys}.

In particular, when \(K=G\) is itself a compact group, the above construction
gives the usual normalized Haar measure on \(G\). For the left-regular action
this is immediate from the orbit map
\[
        g\longmapsto gx_0 .
\]
For the biregular action of \(G\times G\) on \(G\),
\[
        (a,b)\hbullet x:=axb^{-1},
\]
the orbit map based at \(e\) is
\[
        G\times G\longrightarrow G,
        \qquad
        (a,b)\longmapsto ab^{-1}.
\]
The pushforward of \(\mu_G\otimes\mu_G\) under this map is again the
normalized Haar measure on \(G\): indeed, for every \(f\in C(G)\),
\[
        \int_{G\times G} f(ab^{-1})\,d\mu_G(a)\,d\mu_G(b)
        =
        \int_G f(x)\,d\mu_G(x).
\]
Thus the left-regular and the biregular actions induce the same Haar
probability measure on \(G\).

\subsubsection{Invariance and accessibility}
Let \(K\) be a homogeneous \(G\)-space. For each \(g\in G\), the associated
\emph{translation operator}
\[
L_g\colon C(K)\longrightarrow C(K)
\]
is defined by
\[
(L_g f)(x):=f(g^{-1}\hbullet x),
\qquad f\in C(K),\quad x\in K.
\]
Each \(L_g\) is an isometry on \(C(K)\), and \(L_gL_h=L_{gh}\). 
A subspace \(S\subset C(K)\) is said to be \emph{\(G\)-invariant} if
\[
L_g(S)\subset S,\qquad g\in G.
\]
Let \(S\subset C(K)\) be a \(G\)-invariant subspace. A projection
\(
Q\colon C(K)\longrightarrow S
\)
is said to be \emph{\(G\)-equivariant} if
\[
Q\circ L_g=L_g\circ Q,\qquad g\in G.
\]
If \(S\subset C(K)\) is finite-dimensional and \(G\)-invariant, then the
restriction \(\mathbf\pi_S|_{C(K)}\) to \(C(K)\) of the orthogonal projection
\[
\mathbf\pi_S\colon L_2(K,\mu_K)\longrightarrow S
\]
is a \(G\)-equivariant projection. Throughout the paper, Hermitian inner
products are taken to be linear in the first variable.

A finite-dimensional \(G\)-invariant subspace \(S\subset C(K)\) is called
\emph{\(G\)-accessible} if every \(G\)-equivariant projection
\(
Q\colon C(K)\longrightarrow S
\)
coincides with \(\mathbf\pi_S|_{C(K)}\).

For later use in the real Euclidean setting, we record the following
elementary observation, which allows us to pass from the natural
complexification of a real invariant subspace to its real form.

\begin{lemma}
\label{lem: acR-acC}
Let \(K\) be a homogeneous compact \(G\)-space, let
\(S_{\mathbb R}\subset C(K;\mathbb R)\) be a finite-dimensional
\(G\)-invariant subspace, and let
\[
    S_{\mathbb C}
    :=
    S_{\mathbb R}\oplus iS_{\mathbb R}
    \subset C(K;\mathbb C)
\]
be its natural complexification. If \(S_{\mathbb C}\) is
\(G\)-accessible in \(C(K;\mathbb C)\), then \(S_{\mathbb R}\) is
\(G\)-accessible in \(C(K;\mathbb R)\).
\end{lemma}

\begin{proof}
Let
\(
    Q_{\mathbb R}\colon C(K;\mathbb R)\longrightarrow S_{\mathbb R}
\)
be a \(G\)-equivariant real-linear projection. Its complexification,
defined by
\[
    Q_{\mathbb C}(f+ig)
    :=
    Q_{\mathbb R}f+iQ_{\mathbb R}g,
    \qquad f,g\in C(K;\mathbb R),
\]
is a \(G\)-equivariant complex-linear projection from
\(C(K;\mathbb C)\) onto \(S_{\mathbb C}\). By the
\(G\)-accessibility of \(S_{\mathbb C}\), we have
\(
    Q_{\mathbb C}
    =
    \mathbf\pi_{S_{\mathbb C}}\big|_{C(K;\mathbb C)}.
\)
Restricting this identity to \(C(K;\mathbb R)\) gives
\(
    Q_{\mathbb R}
    =
    \mathbf\pi_{S_{\mathbb R}}\big|_{C(K;\mathbb R)}.
\)
Thus \(S_{\mathbb R}\) is \(G\)-accessible.
\end{proof}

\subsection{The core of the strategy}
\label{The core of the strategy}

We now isolate the mechanism behind the paper.
Symmetry turns the restriction to \(C(K)\) of the \(L_2(K,\mu_K)\)-orthogonal
projection onto \(S\) into the canonical projection; its
reproducing kernel then turns the projection constant into the
\(L_1\)-norm of a single kernel slice.

\subsubsection{Reproducing kernels}
\label{Reproducing kernels}

We first recall the elementary kernel facts behind the method. Let \(K\)
be a homogeneous \(G\)-space and let \(S\subset C(K)\) be
finite-dimensional. We regard \(S\) as a subspace of \(L_2(K,\mu_K)\) and
write again
\[
        \mathbf \pi_S:L_2(K,\mu_K)\longrightarrow S
\]
for the orthogonal projection. If
\((f_j)_{j=1}^{\dim S}\) is an orthonormal basis of \(S\), then
\[
        (\mathbf \pi_S f)(x)
        =
        \sum_{j=1}^{\dim S}
        \left(
        \int_K f(y)\overline{f_j(y)}\,d\mu_K(y)
        \right)f_j(x).
\]
Equivalently,
\[
        (\mathbf \pi_S f)(x)
        =
        \int_K f(y)\,
        \overline{\mathbf k_S(x,y)}\,d\mu_K(y),
\]
where
\[
        \mathbf k_S(x,y)
        :=
        \sum_{j=1}^{\dim S}
        \overline{f_j(x)}\,f_j(y).
\]
This function is independent of the chosen orthonormal basis and is the
reproducing kernel of \(S\). We shall use the following standard
properties repeatedly:
\[
        \mathbf k_S(x,\cdot)\in S,
        \qquad
        \mathbf k_S(x,y)
        =
        \overline{\mathbf k_S(y,x)},
        \qquad x,y\in K.
\]
In particular, for every \(f\in L_2(K,\mu_K)\) and \(x\in K\),
\[
        (\mathbf \pi_S f)(x)
        =
        \big\langle
        f,\mathbf k_S(x,\cdot)
        \big\rangle_{L_2(K,\mu_K)}.
\]
Moreover, if \(S\) is \(G\)-invariant, then
\[
         \mathbf k_S(g\hbullet x,g\hbullet y)
        =
         \mathbf k_S(x,y),
        \qquad
        g\in G,\ x,y\in K.
\]
Consequently,
\[
        \mathbf k_S(x,x)=\dim S,
        \qquad x\in K,
\]
and
\[
        \left(
        \int_K |\mathbf k_S(x,y)|^2\,d\mu_K(y)
        \right)^{1/2}
        =
        \sqrt{\dim S},
        \qquad x\in K.
\]
All these identities follow directly from the above orthonormal expansion
and from the \(G\)-invariance of the orthogonal projection.

Finally, reproducing kernels are additive under finite orthogonal
sums. More precisely, if
\(
    S=\bigoplus_{k=1}^m S_k
\)
orthogonally in \(L_2(K,\mu_K)\), where each \(S_k\subset C(K)\) is
finite-dimensional, then
\[
    \mathbf k_S(x,y)
    =
    \sum_{k=1}^m \mathbf k_{S_k}(x,y),
    \qquad x,y\in K.
\]

\subsubsection{Averaging projections}
The following result lies at the heart of our approach, as it reduces the computation of projection constants to an explicit integral involving the reproducing kernel. In the case of the circle group, this idea can already be traced back to Faber's work \cite{faber1914interpolatorische}. For the sake of completeness, we include a proof inspired by \cite{rudin1962projections}; see also \cite[Theorem~III.B.13]{wojtaszczyk1996banach}.

\begin{theorem}\label{thm:accessible-projection-constant}
Let \(K\) be a homogeneous \(G\)-space and let \(S\subset C(K)\) be a
finite-dimensional \(G\)-accessible subspace. Then
\[
        \boldsymbol\lambda(S)
        =
        \|\mathbf \pi_S\colon C(K)\to S\|
        =
        \int_K
        \bigl|\mathbf{k}_S(x_0,y)\bigr|
        \,d\mu_K(y),
\]
where \(\mathbf{k}_S\) is the reproducing kernel of \(S\) in
\(L_2(K,\mu_K)\), and \(x_0\in K\) is arbitrary.
\end{theorem}

\begin{proof}
Let \(Q\colon C(K)\to S\) be any projection. For \(f\in C(K)\), define
\[
        \widetilde Qf
        :=
        \int_G
        L_g^{-1}\bigl(Q(L_gf)\bigr)
        \,d\mu_G(g).
\]
For each fixed \(f\in C(K)\), the integrand is a continuous
\(S\)-valued function and hence Bochner integrable. Moreover,
\[
        \bigl\|L_g^{-1}\bigl(Q(L_gf)\bigr)\bigr\|
        \le
        \|Q\|\,\|f\|,
        \qquad g\in G,
\]
because each \(L_g\) is an isometry of \(C(K)\). Thus
\(\widetilde Q\colon C(K)\to S\) is a well-defined bounded linear
operator with
\[
        \|\widetilde Q\|\le \|Q\|.
\]
The usual averaging argument shows that \(\widetilde Q\) is a
\(G\)-equivariant projection onto \(S\). Since \(S\) is
\(G\)-accessible, it follows that
\[
        \mathbf\pi_S|_{C(K)}
        =
        \widetilde Q.
\]
Consequently, for every \(f\in C(K)\),
\begin{equation*}
        \|\mathbf\pi_Sf\|
        =
        \|\widetilde Qf\| 
        \le
        \int_G
        \bigl\|L_g^{-1}\bigl(Q(L_gf)\bigr)\bigr\|
        \,d\mu_G(g) 
        \le
        \|Q\|\,\|f\|.
\end{equation*}
Hence
\[
        \|\mathbf\pi_S|_{C(K)}\|
        \le
        \|Q\|.
\]
Taking the infimum over all projections \(Q\colon C(K)\to S\) gives
\[
        \boldsymbol\lambda(S)
        =
        \|\mathbf\pi_S\colon C(K)\to S\|.
\]
By the discussion in Subsection~\ref{Reproducing kernels}, for every
\(f\in C(K)\) and \(x\in K\),
\[
        (\mathbf\pi_Sf)(x)
        =
        \int_K
        f(y)\overline{\mathbf k_S(x,y)}
        \,d\mu_K(y).
\]
Therefore
\[
        \|\mathbf\pi_S\colon C(K)\to S\|
        =
        \sup_{x\in K}
        \int_K
        \bigl|\mathbf k_S(x,y)\bigr|
        \,d\mu_K(y).
\]
Indeed, for fixed \(x\in K\), the functional
\[
        C(K) \ni f\longmapsto
        \int_K
        f(y)\overline{\mathbf k_S(x,y)}
        \,d\mu_K(y)
\]
is represented by the scalar measure
\(
        \overline{\mathbf k_S(x,y)}\,d\mu_K(y),
\)
whose total variation is
\[
        \int_K
        \bigl|\mathbf k_S(x,y)\bigr|
        \,d\mu_K(y).
\]
Finally, the \(G\)-invariance of \(\mathbf k_S\) and the homogeneity of
\(K\) imply that this integral is independent of \(x\).
\end{proof}

\subsubsection{Strategy}
The previous theorem suggests the strategy that we shall use repeatedly
throughout this paper, and which has already proved useful in several
settings. We begin by looking at the symmetries of the space \(S\) and,
when possible, identifying an underlying group action that realizes
\(S\) as an invariant subspace of \(C(K)\) for a suitable homogeneous
\(G\)-space \(K\). We then show that these symmetries determine the
projection uniquely, that is, that \(S\) is \(G\)-accessible. The
problem is thereby reduced to making the corresponding reproducing
kernel explicit and, finally, to computing or estimating the
\(L_1\)-norm of one of its slices. In short,
\[
\text{symmetries of \(S\)}
\ \longrightarrow\
\text{underlying group action}
\ \longrightarrow\
\text{accessibility}
\ \longrightarrow\
\text{reproducing kernel}
\ \longrightarrow\
\boldsymbol\lambda(S).
\]

\subsection{Peter--Weyl theory on homogeneous spaces}
\label{sec:The Peter--Weyl decomposition}

We now recall the representation-theoretic decomposition behind the whole
approach. It breaks \(L_2(K,\mu_K)\) into isotypic components, and these
components will later determine both invariant and accessible subspaces.

Recall that throughout the remainder of Section~\ref{sec:Projection constants}, all spaces and operators are considered over the complex scalar field.

\subsubsection{The abelian prototype}

We begin with the abelian case. It is the simplest instance of the
general mechanism developed below and serves as a useful prototype for
the whole paper.  It will also reappear later as a special case of
the general theory.

Let \(G\) be a compact abelian group and let \(\widehat G\) denote its
dual group, that is, the group of all continuous characters
$     \gamma:G\longrightarrow\mathbb T .
$

The classical Peter--Weyl theorem is then just Fourier analysis: the
characters form an orthonormal basis of \(L_2(G)\), and their linear
span is dense in \(C(G)\).
For \(\gamma\in\widehat G\), put
\(
        \mathcal E_\gamma(G):=\operatorname{span}\{\gamma\}.
\)
Thus the Fourier decomposition may be written in block form as
\[
        L_2(G)
        =
        \widehat{\bigoplus}_{\gamma\in\widehat G}
        \mathcal E_\gamma(G).
\]
This notation is deliberately simple here: each block consists of just
one character. With the matrix-coefficient convention adopted below,
the abstract block indexed by the one-dimensional representation
\(\gamma\) is \(\operatorname{span}\{\overline\gamma\}\). Since
\(\gamma\mapsto\overline\gamma\) is a bijection of \(\widehat G\), this
amounts only to a harmless reindexing of the Fourier decomposition.

\begin{theorem}\label{thm:abelian-accessibility}
Let $G$ be a compact abelian group and let $S\subset C(G)$ be a
finite-dimensional subspace. Then the following assertions are equivalent:
\begin{itemize}
    \item[(1)] $S$ is left-invariant;
    \item[(2)] $S$ is left-accessible;
    \item[(3)] there exists a finite set $\Lambda\subset\widehat{G}$ such that
    \(
        S=\operatorname{span}\{\gamma:\gamma\in\Lambda\}.
    \)
\end{itemize}
Moreover, in this case,
\[
        \boldsymbol\lambda(S)
        =
        \int_G
        \Big|
             \sum_{\gamma\in\Lambda}\gamma(g)
        \Big|
        \,d\mu_G(g).
\]
\end{theorem}

\begin{proof}
We use the Fourier convention
\[
        \widehat f(\gamma)
        :=
        \int_G f(x)\overline{\gamma(x)}\,d\mu_G(x),
        \qquad
        \gamma\in\widehat G.
\]
We first prove \((1)\Rightarrow(3)\). Assume that \(S\) is
\(G\)-invariant. Let \(P:C(G)\to S\) be any bounded projection and
average it over the action:
\[
        Qf
        :=
        \int_G L_a^{-1}P(L_a f)\,d\mu_G(a),
        \qquad f\in C(G).
\]
Then \(Q\) is again a projection from \(C(G)\) onto \(S\), and it is
\(G\)-equivariant for the left action.
We claim that \(Q\) is diagonal on the characters. Indeed, for
\(\gamma\in\widehat G\), one has
 $ L_a\gamma=\overline{\gamma(a)}\gamma.$
Since \(QL_a= L_aQ\) , it follows that
\[
\overline{\gamma(a)}\,Q\gamma=L_a(Q\gamma), \quad a\in G.
\]
Taking the $\gamma'$-Fourier coefficient gives
\[
\overline{\gamma(a)}\,\widehat{Q\gamma}(\gamma')
=\overline{\gamma'(a)}\,\widehat{Q\gamma}(\gamma'), \quad a\in G.
\]
Hence \(\widehat{Q\gamma}(\gamma')=0\) whenever
\(\gamma'\neq\gamma\), and therefore
\[
        Q\gamma=c_\gamma\gamma
\]
for some scalar \(c_\gamma\). Since \(Q\) is a projection,
\(c_\gamma^2=c_\gamma\), so \(c_\gamma\in\{0,1\}\).
Set
\[
        \Lambda
        :=
        \{\gamma\in\widehat G:Q\gamma=\gamma\}.
\]
Then \(\Lambda\) is finite, because
\(\Lambda\subset S\cap\widehat G\) and \(S\) is finite-dimensional.

We now show that
$
        S=\operatorname{span}\{\gamma:\gamma\in\Lambda\}.
$
For every trigonometric polynomial
\[
        p=\sum_{\gamma\in F}\widehat p(\gamma)\gamma,
\]
where \(F\subset\widehat G\) is finite, we have
\[
        Qp
        =
        \sum_{\gamma\in F\cap\Lambda}
        \widehat p(\gamma)\gamma.
\]
Let \(s\in S\), and choose trigonometric polynomials \(p_n\) converging
uniformly to \(s\). Since \(Q\) is bounded and \(Qs=s\), we have
\(Qp_n\to s\). Each \(Qp_n\) belongs to
\(\operatorname{span}\{\gamma:\gamma\in\Lambda\}\), which is
finite-dimensional and hence closed. Therefore
\[
        s\in\operatorname{span}\{\gamma:\gamma\in\Lambda\}.
\]
The reverse inclusion follows directly from the definition of
\(\Lambda\). This proves \((1)\Rightarrow(3)\).
We next prove \((3)\Rightarrow(2)\). Assume that
$
        S=\operatorname{span}\{\gamma:\gamma\in\Lambda\}
$
for some finite set \(\Lambda\subset\widehat G\), and let
\(Q:C(G)\to S\) be any \(G\)-equivariant projection. By the same
diagonalization argument as above,
\[
        Q\gamma=c_\gamma\gamma,
        \qquad \gamma\in\widehat G.
\]
Since \(Q\) is the identity on \(S\) and has range contained in \(S\),
\[
        c_\gamma
        =
        \begin{cases}
        1, & \gamma\in\Lambda,\\
        0, & \gamma\notin\Lambda.
        \end{cases}
\]
Thus \(Q\) agrees on trigonometric polynomials with the orthogonal
Fourier projection onto \(S\). By density,
\[
        Q=\mathbf \pi_S|_{C(G)}.
\]
Hence \(S\) is \(G\)-accessible, proving \((3)\Rightarrow(2)\).
Finally, we prove \((2)\Rightarrow(1)\). Every \(G\)-accessible subspace
is \(G\)-invariant by definition. Thus the three assertions are
equivalent.

It remains to compute the projection constant. By \((3)\), we may write
\(S=\operatorname{span}\{\gamma:\gamma\in\Lambda\}\) for some finite
set \(\Lambda\subset\widehat G\). The characters in \(\Lambda\) form
an orthonormal basis of \(S\), and hence
\[
        \mathbf k_S(e,g)
        =
        \sum_{\gamma\in\Lambda}\gamma(g).
\]
The asserted integral formula now follows from
Theorem~\ref{thm:accessible-projection-constant}.
\end{proof}
This result has found applications in several settings, including
Dirichlet polynomials \cite{defant2024projection}, polynomials on Boolean cubes
\cite{defant2024asymptotic}, and support-sensitive  spaces of polynomials in Hamming schemes
\cite{defant2026support}.

\subsubsection{Basic representation theory}

\label{Basic representation theory}

The abelian picture has to be replaced, in the non-abelian case, by
 irreducible unitary representations and their matrix
coefficients. We recall only the basic facts needed below. For further details and background, see, for example, \cite{folland2016course}.

Let \(G\) be a compact group. A unitary representation of \(G\) is a
continuous homomorphism
\[
        \pi:G\longrightarrow \mathcal U(V_\pi),
\]
where \(V_\pi\) is a complex Hilbert space and
\(\mathcal U(V_\pi)\) denotes the group of unitary operators on \(V_\pi\).
We equip \(\mathcal U(V_\pi)\) with the strong operator topology, so that
continuity of \(\pi\) means that, for every \(v\in V_\pi\), the map
\(g\mapsto\pi(g)v\) is continuous.

Given two unitary representations \(\pi:G\to\mathcal U(V_\pi)\) and
\(\sigma:G\to\mathcal U(V_\sigma)\), a bounded linear operator
\(T:V_\pi\to V_\sigma\) is called an \emph{intertwining operator} if
\[
        T\pi(g)=\sigma(g)T,
        \qquad g\in G.
\]
The representations \(\pi\) and \(\sigma\) are said to be \emph{unitarily
equivalent}, and we write \(\pi\cong\sigma\), if there is a unitary
intertwining isomorphism between \(V_\pi\) and \(V_\sigma\).

A unitary representation \(\pi:G\to\mathcal U(V_\pi)\) is \emph{irreducible} if
the only closed subspaces of \(V_\pi\) which are invariant under all
operators \(\pi(g)\), \(g\in G\), are \(\{0\}\) and \(V_\pi\).

The \emph{unitary
dual} of \(G\) is the set
\[
        \widehat G
        :=
        \big\{[\pi]:
        \pi \text{ is an irreducible unitary representation of }G
        \big\},
\]
where representations are identified up to unitary equivalence. As usual,
we shall write simply \(\pi\in\widehat G\) and choose one representative
from each equivalence class.

Since \(G\) is compact, every irreducible unitary representation of \(G\)
is finite-dimensional, that is, whenever \(\pi\in\widehat G\), the space
\(V_\pi\) is finite-dimensional. We write
\[
        d_\pi:=\dim V_\pi
\]
and call \(d_\pi\) the degree of \(\pi\).
We shall use the following form of Schur's lemma.

\begin{lemma}
\label{thm:schur-lemma}

Let \(G\) be a compact group and let \(\pi\) and \(\sigma\) be irreducible
unitary representations of \(G\). Then the set of intertwining operators
\(T:V_\pi\to V_\sigma\) is 
\begin{enumerate}
    \item trivial if \(\pi\not\cong\sigma\), i.e. in this case \(T=0\) is the only intertwining operator;
    \item one dimensional if \(\pi\cong\sigma\). In particular, every intertwining operator is a multiple of the identity if  \(\pi=\sigma\).
\end{enumerate}

\end{lemma}

We shall also use the Schur orthogonality relations for matrix
coefficients. If \(\pi,\sigma\in\widehat G\), \(u,v\in V_\pi\), and
\(u',v'\in V_\sigma\), then
\[
    \int_G
    \langle v,\pi(g)u\rangle\,
    \overline{\langle v',\sigma(g)u'\rangle}
    \,d\mu_G(g)
    =
    \begin{cases}
        \dfrac{1}{d_\pi}
        \langle v,v'\rangle
        \langle u',u\rangle,
        & \pi=\sigma,\\[6pt]
        0,
        & \pi\not\cong\sigma.
    \end{cases}
\]
For \(\pi\in\widehat G\), its character is defined by
\[
    \chi_\pi(g):=\operatorname{tr}\bigl(\pi(g)\bigr),
    \qquad g\in G.
\]
It is a class function, that is,
\[
    \chi_\pi(hgh^{-1})=\chi_\pi(g),
    \qquad g,h\in G.
\]
As an immediate consequence of the orthogonality relations for matrix
coefficients, the irreducible characters satisfy
\[
    \int_G
    \chi_\pi(g)\overline{\chi_\sigma(g)}
    \,d\mu_G(g)
    =
    \begin{cases}
        1, & \pi\cong\sigma,\\
        0, & \pi\not\cong\sigma.
    \end{cases}
\]
For a finite-dimensional unitary representation
\(\pi\colon G\to\mathcal U(V_\pi)\), its contragredient representation
is the unitary representation
\[
    \pi^*\colon G\longrightarrow\mathcal U(V_\pi^*)
\]
defined by
\[
    (\pi^*(g)\varphi)(v)
    :=
    \varphi\bigl(\pi(g^{-1})v\bigr),
    \qquad
    g\in G,\quad
    \varphi\in V_\pi^*,\quad
    v\in V_\pi.
\]
Its character satisfies
\[
    \chi_{\pi^*}(g)
    =
    \overline{\chi_\pi(g)},
    \qquad g\in G.
\]

\subsubsection{Matrix coefficient spaces}

Let \(K\) be a homogeneous \(G\)-space and fix \(x_0\in K\). We write
\[
        H_{x_0}:=\{h\in G:h\hbullet x_0=x_0\}
\]
for the stabilizer of \(x_0\). For \(\pi\in\widehat G\), set
\[
        V_\pi^{H_{x_0}}
        :=
        \{u\in V_\pi:\pi(h)u=u\text{ for all }h\in H_{x_0}\}.
\]
We call the vectors in \(V_\pi^{H_{x_0}}\) the \(x_0\)-admissible vectors
for \(\pi\). The corresponding matrix coefficient space is defined by
\[
        \mathcal E_\pi(K)
        :=
        \operatorname{span}
        \Big\{
        \pi^{x_0}_{vu}:g\hbullet x_0\longmapsto
        \langle v,\pi(g)u\rangle
        :
        u\in V_\pi^{H_{x_0}},\ v\in V_\pi
        \Big\}.
\]
Indeed, admissibility is precisely what makes the formula well-defined on
\(K\): if \(g\hbullet x_0=h\hbullet x_0\), then \(h^{-1}g\in H_{x_0}\), and
hence
\[
        \langle v,\pi(g)u\rangle
        =
        \langle v,\pi(h)u\rangle .
\]
We put
\[
        m_\pi(K):=\dim V_\pi^{H_{x_0}}.
\]
By the homogeneous Peter--Weyl theorem, stated below as
Theorem~\ref{thm:homogeneous-peter-weyl}, \(m_\pi(K)\) is the multiplicity with which
\(V_\pi\) occurs in the unitary representation of \(G\) on
\(L_2(K,\mu_K)\).

The number \(m_\pi(K)\) is independent of the choice of \(x_0\). Indeed, if
\(x_1=a\hbullet x_0\), then
\[
        H_{x_1}=aH_{x_0}a^{-1},
\]
and \(\pi(a)\) maps \(V_\pi^{H_{x_0}}\) unitarily onto
\(V_\pi^{H_{x_1}}\). Moreover, for every
\(u\in V_\pi^{H_{x_0}}\) and \(v\in V_\pi\),
\[
        \pi^{x_1}_{v,\pi(a)u}
        =
        \pi^{x_0}_{v,u}
\]
as functions on \(K\). Hence the space \(\mathcal E_\pi(K)\) is also
independent of the choice of base point.

If \(m_\pi(K)=0\), then
\(\mathcal E_\pi(K)=\{0\}\).
Assume now that \(m_\pi(K)>0\). Choose an orthonormal basis
\(u_1,\ldots,u_{m_\pi(K)}\) of \(V_\pi^{H_{x_0}}\) and an orthonormal
basis \(v_1,\ldots,v_{d_\pi}\) of \(V_\pi\). Then the functions
\[
        e^\pi_{\ell j}(g\hbullet x_0)
        :=
        \sqrt{d_\pi}\,
        \langle v_j,\pi(g)u_\ell\rangle,
        \qquad
        1\leq \ell\leq m_\pi(K),\quad 1\leq j\leq d_\pi,
\]
form an orthonormal basis of \(\mathcal E_\pi(K)\). In particular,
\[
        \dim \mathcal E_\pi(K)=d_\pi\,m_\pi(K).
\]
For \(1\leq \ell\leq m_\pi(K)\), define
\[
        \mathcal E_\pi^{(\ell)}(K)
        :=
        \operatorname{span}
        \{e^\pi_{\ell j}:1\leq j\leq d_\pi\}.
\]
Then
\[
        \mathcal E_\pi(K)
        =
        \bigoplus_{\ell=1}^{m_\pi(K)}
        \mathcal E_\pi^{(\ell)}(K)
        \quad \text{orthogonally in \(L_2(K,\mu_K),\)}
\]
 and each \(\mathcal E_\pi^{(\ell)}(K)\) is a \(G\)-equivariant copy of \(V_\pi\).
More explicitly, for fixed \(\ell\), the map
\[
        V_\pi\longrightarrow \mathcal E_\pi^{(\ell)}(K),
        \quad
        v\longmapsto
        \big(g\hbullet x_0\mapsto \langle v,\pi(g)u_\ell\rangle\big),
\]
realizes this equivalence. Indeed, if
\(
    f_v(g\hbullet x_0):=\langle v,\pi(g)u_\ell\rangle,
\)
then
\(
    L_af_v=f_{\pi(a)v}
\)
for every \(a\in G\).
Equivalently, after choosing the above orthonormal basis of
\(V_\pi^{H_{x_0}}\), one may write
\begin{equation} \label{eq:tensor description}
     \mathcal E_\pi(K)
        \cong
        V_\pi\otimes \mathbb C^{m_\pi(K)},
\end{equation}
where \(G\) acts only on the first factor, and the summands
\(\mathcal E_\pi^{(\ell)}(K)\) correspond to the coordinate copies
\(V_\pi\otimes\mathbb C e_\ell\).

\subsubsection{Non-abelian case}

We now turn to the general form of the Peter--Weyl decomposition on a homogeneous space. This classical result can be found in many standard references; for instance, see \cite{deitmar2014,folland2016course,fulton2013representation,takeuchi1994modern}.

\begin{theorem}
\label{thm:homogeneous-peter-weyl}
Let \(K\) be a homogeneous \(G\)-space. Then
\[
        L_2(K,\mu_K)
        =
        \widehat{\bigoplus}_{\pi\in\widehat G}\,
        \mathcal E_\pi(K).
\]
Moreover, for each \(\pi\in\widehat G\),
\[
        \mathcal E_\pi(K)
        =
        \bigoplus_{\ell=1}^{m_\pi(K)}
        \mathcal E_\pi^{(\ell)}(K),
\]
where every \(\mathcal E_\pi^{(\ell)}(K)\) is \(G\)-equivalent to
\(V_\pi\). The space \(\mathcal E_\pi(K)\), being the sum of all
irreducible subrepresentations equivalent to \(V_\pi\), is called the
\(\pi\)-isotypic component of \(L_2(K,\mu_K)\).
Moreover, the linear span of the spaces
\(\mathcal E_\pi(K)\), \(\pi\in\widehat G\), is dense in \(C(K)\).
\end{theorem}

Thus the non-abelian Peter--Weyl theorem decomposes \(L_2(K,\mu_K)\)
first into isotypic components and then, inside each isotypic component,
into \(m_\pi(K)\) irreducible copies of \(V_\pi\). The multiplicity
\(m_\pi(K)\) measures exactly how often \(V_\pi\) occurs in
\(\mathcal E_\pi(K)\). The isotypic components and their multiplicities are uniquely determined by the representation, whereas their decompositions into irreducible summands are not. In particular, the abelian case is the special
case in which all irreducible unitary representations are one-dimensional,
and the theorem reduces to the usual Fourier decomposition on compact
abelian groups.

\subsection{Invariance vs. accessibility}
Given a homogeneous \(G\)-space \(K\), we now pass from the Peter--Weyl
decomposition of \(L_2(K,\mu_K)\) to the finite-dimensional subspaces of $C(K)$
compatible with the \(G\)-action. Our main concern is to determine when
the symmetry of such a subspace singles out the minimal projection from $C(K)$ onto
it. More precisely, we ask when \(G\)-invariance already implies
\(G\)-accessibility, and how this implication may fail when the subspace
contains only a proper part of an isotypic component.

\subsubsection{Invariant subspaces}

Let us first  explain how the Peter--Weyl decomposition controls
\(G\)-invariant subspaces. Throughout this subsection \(K\) is a
homogeneous \(G\)-space and \(L_g\) denotes the unitary representation of
\(G\) on \(L_2(K,\mu_K)\) given by
\[
        (L_g f)(x):=f(g^{-1}\hbullet x),
        \qquad g\in G,\ x\in K .
\]
Of course, for a subspace \(S\subset C(K)\), invariance under the
operators \(L_g\) is equivalent to invariance under the translations used
above.

For \(\pi\in\widehat G\), let \(\chi_\pi\) be its character and let
\(d_\pi=\dim V_\pi\). We define an operator
\[
        P_\pi:L_2(K,\mu_K)\longrightarrow L_2(K,\mu_K)
\]
by
\[
        P_\pi f
        :=
        d_\pi
        \int_G
        \overline{\chi_\pi(g)}\, L_g f\, d\mu_G(g),
        \qquad f\in L_2(K,\mu_K).
\]
The integral is understood as a Bochner integral. Since \(g\mapsto L_g f\)
is strongly continuous and \(\chi_\pi\) is continuous on the compact group
\(G\), the operator \(P_\pi\) is well-defined and bounded.

\begin{lemma}
\label{lem:character-projection-commutes}
For every \(\pi\in\widehat G\), every \(a\in G\), and every
\(f\in L_2(K,\mu_K)\), one has
\[
        P_\pi L_a f=L_a P_\pi f .
\]
\end{lemma}

\begin{proof}
By definition,
\[
        P_\pi L_a f
        =
        d_\pi
        \int_G
        \overline{\chi_\pi(g)}\, L_gL_a f\,d\mu_G(g)
        =
        d_\pi
        \int_G
        \overline{\chi_\pi(g)}\, L_{ga}f\,d\mu_G(g).
\]
We make the change of variables \(g=aha^{-1}\). Since the Haar measure on a
compact group is invariant under inner automorphisms and since
\(\chi_\pi\) is a class function (see Subsection~\ref{Basic representation theory} for the definition), this gives
\[
        P_\pi L_a f
        =
        d_\pi
        \int_G
        \overline{\chi_\pi(h)}\, L_{ah}f\,d\mu_G(h)
        =
        L_a
        \left(
        d_\pi
        \int_G
        \overline{\chi_\pi(h)}\, L_h f\,d\mu_G(h)
        \right)
        =
        L_aP_\pi f,
\]
as claimed.
\end{proof}

The operators $P_\pi$ are precisely the orthogonal projections onto the corresponding isotypic components. This classical fact can be found, for instance, in \cite[Proposition~7.3.3]{deitmar2014}; nevertheless, we include a short proof for the sake of completeness.

\begin{proposition}
\label{prop:character-projection}
For every \(\pi\in\widehat G\), the operator \(P_\pi\) is the orthogonal
projection of \(L_2(K,\mu_K)\) onto \(\mathcal E_\pi(K)\).
\end{proposition}

\begin{proof}
Fix \(\rho\in\widehat G\). By the homogeneous Peter--Weyl Theorem~\ref{thm:homogeneous-peter-weyl} we may
write
\[
        \mathcal E_\rho(K)
        =
        \bigoplus_{\ell=1}^{m_\rho(K)}
        \mathcal E_\rho^{(\ell)}(K),
\]
where each \(\mathcal E_{\rho}^{(\ell)}(K)\) is a \(G\)-invariant
subspace \(G\)-equivalent to \(V_\rho\). We first observe that
\(P_\pi\) leaves every \(\mathcal E_{\rho}^{(\ell)}(K)\) invariant.
Indeed, if \(f\in\mathcal E_{\rho}^{(\ell)}(K)\), then
\(
    g\longmapsto
    \overline{\chi_\pi(g)}\,L_gf
\)
takes its values in the finite-dimensional, hence closed, subspace
\(\mathcal E_{\rho}^{(\ell)}(K)\). Its Bochner integral therefore also
belongs to \(\mathcal E_{\rho}^{(\ell)}(K)\). Thus
\[
    P_\pi\big|_{\mathcal E_{\rho}^{(\ell)}(K)}
    :
    \mathcal E_{\rho}^{(\ell)}(K)
    \longrightarrow
    \mathcal E_{\rho}^{(\ell)}(K)
\]
is well defined. By
Lemma~\ref{lem:character-projection-commutes}, this restriction
commutes with the \(G\)-action and is therefore an intertwining
operator.
Hence, by Schur's lemma~\ref{thm:schur-lemma}, it is a scalar
multiple of the identity:
\[
        P_\pi|_{\mathcal E_\rho^{(\ell)}(K)}
        =
        \lambda_{\rho,\ell}\,
        \operatorname{id}_{\mathcal E_\rho^{(\ell)}(K)} .
\]
It remains to compute the scalar. The character of the representation
\(L\) on the irreducible summand
\(\mathcal E_\rho^{(\ell)}(K)\) is \(\chi_\rho\). Therefore
\[
\begin{aligned}
        \operatorname{tr}
        \bigl(
        P_\pi|_{\mathcal E_\rho^{(\ell)}(K)}
        \bigr)
        &=
        d_\pi
        \int_G
        \overline{\chi_\pi(g)}
        \operatorname{tr}
        \bigl(
        L_g|_{\mathcal E_\rho^{(\ell)}(K)}
        \bigr)
        \,d\mu_G(g)                                              =
        d_\pi
        \int_G
        \overline{\chi_\pi(g)}\,\chi_\rho(g)\,d\mu_G(g).
\end{aligned}
\]
By the orthogonality relations for irreducible characters (see again Section~\ref{Basic representation theory}), this trace is
\(d_\pi\) if \(\rho\cong\pi\), and \(0\) otherwise. On the other hand,
\[
        \operatorname{tr}
        \bigl(
        P_\pi|_{\mathcal E_\rho^{(\ell)}(K)}
        \bigr)
        =
        \lambda_{\rho,\ell}\,\dim V_\rho
        =
        \lambda_{\rho,\ell}\,d_\rho .
\]
Consequently,
\[
        \lambda_{\rho,\ell}
        =
        \begin{cases}
        1, & \rho\cong\pi,\\
        0, & \rho\not\cong\pi.
        \end{cases}
\]
Thus \(P_\pi\) is the identity on \(\mathcal E_\pi(K)\) and vanishes on
\(\mathcal E_\rho(K)\) for \(\rho\not\cong\pi\).  Since the algebraic sum of these components is dense in \(L_2(K,\mu_K)\) and \(P_\pi\) is bounded, it follows that
\(P_\pi\) is the orthogonal projection onto \(\mathcal E_\pi(K)\).
\end{proof}

We can now describe invariant subspaces. We show that  every finite-dimensional \(G\)-invariant subspace of
\(C(K)\) is obtained by taking finitely many irreducible copies inside
the isotypic components of the Peter--Weyl decomposition.

\begin{theorem}
\label{thm:invariant-subspaces}
Let \(K\) be a homogeneous \(G\)-space and let
\(S\subset C(K)\) be a finite-dimensional \(G\)-invariant subspace. Then
there is a finite set \(\Lambda\subset\widehat G\) such that
\[
        S
        =
        \bigoplus_{\pi\in\Lambda}
        S_\pi,
        \qquad
        S_\pi:=S\cap\mathcal E_\pi(K),
\]
where the sum is orthogonal in \(L_2(K,\mu_K)\).
Moreover, for every \(\pi\in\Lambda\), after a suitable choice of the
decomposition
\(
        \mathcal E_\pi(K)
        =
        \bigoplus_{\ell=1}^{m_\pi(K)}
        \mathcal E_\pi^{(\ell)}(K)
\)
into \(G\)-equivariant copies of \(V_\pi\), there is an index set
\(
        I_\pi\subset\{1,\ldots,m_\pi(K)\}
\)
such that
\[
        S_\pi
        =
        \bigoplus_{\ell\in I_\pi}
        \mathcal E_\pi^{(\ell)}(K).
\]
\end{theorem}

\begin{proof}
Since \(S\) is finite-dimensional, it is a closed subspace of
\(L_2(K,\mu_K)\). Its \(G\)-invariance implies that
\(L_gf\in S\) for every \(f\in S\) and \(g\in G\). Hence the
character-integral formula gives
\[
        P_\pi f
        =
        d_\pi
        \int_G
        \overline{\chi_\pi(g)}\,L_g f\,d\mu_G(g)
        \in S .
\]
By Proposition~\ref{prop:character-projection}, \(P_\pi f\in
\mathcal E_\pi(K)\). Thus
\[
        P_\pi(S)\subset S\cap\mathcal E_\pi(K)=S_\pi .
\]
Conversely, since \(P_\pi\) is the identity on \(\mathcal E_\pi(K)\), we
also have \(S_\pi\subset P_\pi(S)\). Hence
\[
        P_\pi(S)=S_\pi .
\]
The Peter--Weyl decomposition  from Theorem~\ref{thm:homogeneous-peter-weyl} gives, for every \(f\in S\),
\[
    f=\sum_{\pi\in\widehat G}P_\pi f
    \qquad\text{in }L_2(K,\mu_K),
\]
with \(P_\pi f\in S_\pi\) for every \(\pi\in\widehat G\). Since \(S\)
is finite-dimensional, only finitely many of the mutually orthogonal
subspaces \(S_\pi\) can be non-zero.
Thus there is a finite set
\(\Lambda\subset\widehat G\) such that
\[
        S=
        \bigoplus_{\pi\in\Lambda}S_\pi .
\]
It remains to describe \(S_\pi\) inside a fixed isotypic component.
By Theorem~\ref{thm:homogeneous-peter-weyl},
\[
    \mathcal E_\pi(K)
    =
    \bigoplus_{\ell=1}^{m_\pi(K)}
    \mathcal E_\pi^{(\ell)}(K),
\]
where each summand is \(G\)-equivalent to \(V_\pi\). Equivalently,
after choosing such a decomposition, there is a \(G\)-equivariant
identification
\[
    \mathcal E_\pi(K)
    \cong
    V_\pi\otimes\mathbb C^{m_\pi(K)},
\]
under which \(G\) acts on the tensor product  by
\[
    \pi(g)\otimes \operatorname{id}_{\mathbb C^{m_\pi(K)}}.
\]
Since \(S_\pi\subset \mathcal E_\pi(K)\) is \(G\)-invariant and the
representation is unitary, its orthogonal complement is also
\(G\)-invariant. Hence the orthogonal projection
\[
    Q_\pi:
    V_\pi\otimes\mathbb C^{m_\pi(K)}
    \longrightarrow
    S_\pi
\]
is \(G\)-equivariant. Fixing an orthonormal basis
\(e_1,\ldots,e_{m_\pi(K)}\) of the multiplicity space, write \(Q_\pi\)
as a block operator matrix
\[
    Q_\pi=(Q_{rs})_{r,s=1}^{m_\pi(K)},
    \quad
    Q_{rs}\in\mathcal L(V_\pi).
\]
The equivariance of \(Q_\pi\) implies that
$ Q_{rs}\pi(g)=\pi(g)Q_{rs}$
    for every $g\in G$ and \(r,s\). Since \(\pi\) is irreducible,
Schur's Lemma~\ref{thm:schur-lemma} yields
\(
    Q_{rs}=a_{rs}\operatorname{id}_{V_\pi}
\)
for suitable scalars \(a_{rs}\). Therefore
\[
    Q_\pi
    =
    \operatorname{id}_{V_\pi}\otimes A_\pi
\]
for some operator
\(
    A_\pi\in
    \mathcal L\bigl(\mathbb C^{m_\pi(K)}\bigr).
\)
Because \(Q_\pi\) is an orthogonal projection, so is \(A_\pi\). Thus,
defining
\(
    M_\pi:=\operatorname{ran}A_\pi,
\)
we obtain
\[
    S_\pi
    =
    \operatorname{ran}Q_\pi
    \cong
    V_\pi\otimes M_\pi.
\]
Finally, choose an orthonormal basis of
\(\mathbb C^{m_\pi(K)}\) whose first vectors form an orthonormal basis
of \(M_\pi\). With respect to the corresponding decomposition of
\(\mathcal E_\pi(K)\) into \(G\)-equivariant copies of \(V_\pi\), there
is an index set
\(
    I_\pi\subset\{1,\ldots,m_\pi(K)\}
\)
such that
\(
    S_\pi
    =
    \bigoplus_{\ell\in I_\pi}
    \mathcal E_\pi^{(\ell)}(K).
\)
This completes the proof.
\end{proof}

\subsubsection{Accessible subspaces}

We now characterize accessibility in terms of the Peter--Weyl
decomposition. The point is that accessibility does not merely ask for
\(G\)-invariance. It asks that the equivariant projection onto the space
be uniquely determined by the Hilbert space structure of \(L_2(K,\mu_K)\).

\begin{theorem}
\label{thm:accessible-subspaces}
Let \(K\) be a homogeneous \(G\)-space and let
\(S\subset C(K)\) be a finite-dimensional \(G\)-invariant subspace. Then
the following assertions are equivalent:
\begin{enumerate}
\item \(S\) is \(G\)-accessible.
\item \(S\) is a finite orthogonal sum of full isotypic components, that
      is, there exists a finite set \(\Lambda\subset\widehat G\) such
      that
      \(
              S
              =
              \bigoplus_{\pi\in\Lambda}
              \mathcal E_\pi(K).
      \)
\end{enumerate}
\end{theorem}

\begin{proof}
Assume first that $(2)$ holds, so
\(
        S=
        \bigoplus_{\pi\in\Lambda}
        \mathcal E_\pi(K)
\)
for some finite set \(\Lambda\subset\widehat G\). Let
\(Q:C(K)\to S\) be a \(G\)-equivariant projection. We prove that
\(Q=\mathbf \pi_S|_{C(K)}\).
Fix \(\tau\in\widehat G\). If \(\tau\notin\Lambda\), then
\[
        Q|_{\mathcal E_\tau(K)}
        :
        \mathcal E_\tau(K)
        \longrightarrow
        \bigoplus_{\pi\in\Lambda}
        \mathcal E_\pi(K)
\]
is \(G\)-equivariant. Decomposing both sides into irreducible copies,
Schur's lemma~\ref{thm:schur-lemma} implies that every component of this
map is zero, since no copy of \(V_\tau\) occurs in the range. Hence
\[
        Q|_{\mathcal E_\tau(K)}=0
        \qquad(\tau\notin\Lambda).
\]
If \(\tau\in\Lambda\), then \(\mathcal E_\tau(K)\subset S\), and since
\(Q\) is a projection onto \(S\), one has
\[
        Q|_{\mathcal E_\tau(K)}
        =
        \operatorname{id}_{\mathcal E_\tau(K)}.
\]
Thus \(Q\) and \(\mathbf \pi_S|_{C(K)}\) agree on every isotypic component.
The restriction \(\mathbf \pi_S|_{C(K)}\) is bounded, as follows directly
from its finite-rank integral representation. Since the linear span of the
isotypic components is dense in \(C(K)\), and both operators are bounded
on \(C(K)\), it follows that
\[
        Q=\mathbf \pi_S|_{C(K)}.
\]
Thus \(S\) is \(G\)-accessible, and hence $(1)$ holds.

Conversely, assume that \(S\) is \(G\)-accessible. By
Theorem~\ref{thm:invariant-subspaces}, there is a finite set
\(\Lambda\subset\widehat G\) such that
\[
        S=
        \bigoplus_{\pi\in\Lambda} S_\pi,
        \qquad
        S_\pi:=S\cap\mathcal E_\pi(K),
\]
and, for each \(\pi\in\Lambda\), one may choose an orthogonal decomposition
\[
        \mathcal E_\pi(K)
        =
        \bigoplus_{\ell=1}^{m_\pi(K)}
        \mathcal E_\pi^{(\ell)}(K)
\]
such that there is an index set
\(I_\pi\subset\{1,\ldots,m_\pi(K)\}\) with
\[
        S_\pi
        =
        \bigoplus_{\ell\in I_\pi}
        \mathcal E_\pi^{(\ell)}(K).
\]
Suppose that \(S\) is not a finite sum of full isotypic components. Then,
for some \(\tau\in\Lambda\), the set \(I_\tau\) is a nonempty proper
subset of \(\{1,\ldots,m_\tau(K)\}\). Using the tensor description in \eqref{eq:tensor description}
\[
        \mathcal E_\tau(K)
        \cong
        V_\tau\otimes \mathbb C^{m_\tau(K)},
\]
where \(G\) acts on the tensor product only on the first factor, the
subspace \(S_\tau\) has the form
\[
        S_\tau
        \cong
        V_\tau\otimes M_\tau
\]
for a proper nonzero subspace
\(M_\tau\subset\mathbb C^{m_\tau(K)}\).
Choose a projection
\(
        R:\mathbb C^{m_\tau(K)}\longrightarrow M_\tau
\)
which is not the orthogonal projection. Then
\[
        \operatorname{id}_{V_\tau}\otimes R
        :
        V_\tau\otimes \mathbb C^{m_\tau(K)}
        \longrightarrow
        V_\tau\otimes M_\tau
\]
is a \(G\)-equivariant projection onto \(S_\tau\), but it is not the
orthogonal projection onto \(S_\tau\).
Let \(P_\pi\) denote the orthogonal projection from \(L_2(K,\mu_K)\) onto
\(\mathcal E_\pi(K)\). By the character-integral formula in Proposition~\ref{prop:character-projection},
\(P_\pi|_{C(K)}\) is bounded in the supremum norm.  Define an operator
\(
        Q_S:C(K)\longrightarrow S
\)
by
\[
        Q_S
        :=
        \sum_{\pi\in\Lambda,\ \pi\neq\tau}
        \mathbf \pi_{S_\pi} P_\pi
        +
        (\operatorname{id}_{V_\tau}\otimes R) P_\tau ,
\]
where \(\mathbf \pi_{S_\pi}\) denotes the orthogonal projection of
\(\mathcal E_\pi(K)\) onto \(S_\pi\). This operator is well-defined and
bounded on \(C(K)\), has range \(S\), is a projection, and is
\(G\)-equivariant. However, on \(\mathcal E_\tau(K)\) it differs from the
orthogonal projection onto \(S_\tau\), because \(R\) was chosen
non-orthogonal. Hence
\[
        Q_S\neq \mathbf \pi_S|_{C(K)}.
\]
This contradicts the \(G\)-accessibility of \(S\), showing that \(S\) is
a finite orthogonal sum of full isotypic components.
\end{proof}

\subsubsection{Multiplicity-free actions}
\label{sss:multiplicity-free}

The preceding two characterizations show precisely when invariance
already implies accessibility. Let \(K\) be a homogeneous \(G\)-space.
We say that the action of \(G\) on \(K\) is \emph{multiplicity-free} if
\[
        m_\pi(K)\leq 1
        \qquad
        \text{for every }\pi\in\widehat G .
\]
Equivalently, every non-zero isotypic component \(\mathcal E_\pi(K)\)
is already \(G\)-equivalent to the irreducible representation space
\(V_\pi\).

\begin{remark}
Let \(H\) be a closed subgroup of \(G\). The pair \((G,H)\) is called a
\emph{Gelfand pair} if, for every irreducible unitary representation
\(\pi\) of \(G\), the space of \(H\)-invariant vectors \(V_\pi^H\) has
dimension at most~\(1\). Since
\(
        m_\pi(K)=\dim V_\pi^{H_{x_0}},
\)
the action of \(G\) on \(K\) is multiplicity-free if and only if
\((G,H_{x_0})\) is a Gelfand pair.
\end{remark}

\begin{corollary}
\label{cor:gelfand-pairs}
Let \(K\) be a homogeneous \(G\)-space. Then the following assertions are
equivalent:
\begin{enumerate}
\item \((G,K)\) is multiplicity-free.
\item Every finite-dimensional \(G\)-invariant subspace
      \(S\subset C(K)\) is \(G\)-accessible.
\end{enumerate}
\end{corollary}

\begin{proof}
Assume first that \((G,K)\) is multiplicity-free. Let
\(S\subset C(K)\) be finite-dimensional and \(G\)-invariant. By
Theorem~\ref{thm:invariant-subspaces}, there is a finite set
\(\Lambda\subset\widehat G\) such that
\[
        S
        =
        \bigoplus_{\pi\in\Lambda}
        S_\pi,
        \qquad
        S_\pi:=S\cap\mathcal E_\pi(K).
\]
Moreover, inside each isotypic component,
\[
        \mathcal E_\pi(K)
        =
        \bigoplus_{\ell=1}^{m_\pi(K)}
        \mathcal E_\pi^{(\ell)}(K),
\]
the space \(S_\pi\) is a sum of some of the irreducible copies
\(\mathcal E_\pi^{(\ell)}(K)\). Since \(m_\pi(K)\leq 1\), there is no
proper non-zero choice inside \(\mathcal E_\pi(K)\). Hence, for every
\(\pi\),
\[
        S_\pi=0
        \qquad\text{or}\qquad
        S_\pi=\mathcal E_\pi(K).
\]
Therefore \(S\) is a finite orthogonal sum of full isotypic components.
By Theorem~\ref{thm:accessible-subspaces}, \(S\) is \(G\)-accessible.

Conversely, assume that \((G,K)\) is not multiplicity-free. Then there is some \(\pi\in\widehat G\) with
\(m_\pi(K)>1\). Choose one irreducible copy
\[
        \mathcal E_\pi^{(1)}(K)
        \subset
        \mathcal E_\pi(K).
\]
This is a finite-dimensional \(G\)-invariant subspace of \(C(K)\). But
since \(\mathcal E_\pi^{(1)}(K)\) is not a full isotypic component, it is not accessible by
Theorem~\ref{thm:accessible-subspaces}.

\end{proof}

\subsubsection{Strong accessibility}
\label{sec:strong-accessibility}
The following notion provides a convenient criterion for recognizing multiplicity-one isotypic components.
Let \(K\) be a homogeneous \(G\)-space and fix a base point \(x_0\in K\).
We write
\[
        H:=H_{x_0}:=\{h\in G:h\hbullet x_0=x_0\}
\]
for its stabilizer. If \(S\subset C(K)\) is finite-dimensional and
\(G\)-invariant, we define
\[
        S^H
        :=
        \{f\in S:f(h\hbullet x)=f(x)
        \text{ for all }h\in H,\ x\in K\}.
\]
Thus \(S^H\) is the subspace of functions in \(S\) which are invariant
under the stabilizer of the chosen base point.

We say that  \(S\) is \emph{strongly accessible with respect to \(x_0\)}
if, for every non-zero \(f\in S^H\), there exists
\(\lambda\in\mathbb C\) such that
\[
          \mathbf k_S(x_0,\cdot)=\lambda f .
\]
Equivalently, the kernel section \( \mathbf k_S(x_0,\cdot)\) spans the whole
space of stabilizer-invariant functions in \(S\).

The following simple characterization shows, in
Remark~\ref{nox0}, that strong accessibility is independent of the
choice of base point.

\begin{lemma}
\label{lem:strong-accessibility-stabilizer}
Let \(K\) be a homogeneous \(G\)-space,
 \(S\subset C(K)\) a non-zero finite-dimensional \(G\)-invariant
subspace and $x_0 \in K$. Then \(S\) is strongly accessible with respect to \(x_0\) if
and only if
\(
       \, \dim S^{H_{x_0}}=1 .
\)
\end{lemma}

\begin{proof}
The kernel section \(\mathbf k_S(x_0,\cdot)\) belongs to \(S\). It is also
\(H\)-invariant. Indeed, if \(h\in H\), then \(h\hbullet x_0=x_0\), and
the \(G\)-invariance of the reproducing kernel gives
\[
        \mathbf k_S(x_0,h\hbullet x)
        =
        \mathbf k_S(h^{-1}\hbullet x_0,x)
        =
        \mathbf k_S(x_0,x),
        \qquad x\in K .
\]
Hence \(\mathbf k_S(x_0,\cdot)\in S^H\). Moreover, this function is non-zero:
otherwise every \(f\in S\) would vanish at \(x_0\), and by the
\(G\)-invariance of \(S\) and the transitivity of the action, every
\(f\in S\) would vanish identically.

If \(S\) is strongly accessible with respect to \(x_0\), then every
non-zero element of \(S^H\) is a scalar multiple of
\(\mathbf k_S(x_0,\cdot)\). Therefore \(\dim S^H=1\). Conversely, if
\(\dim S^H=1\), then the non-zero vector \(\mathbf k_S(x_0,\cdot)\) spans
\(S^H\), and the defining condition of strong accessibility follows.
\end{proof}

\begin{remark}
\label{nox0}
   Although the definition of strong accessibility involves a base point, the property itself does
not. Indeed, let \(x_1=g_0\hbullet x_0\). Then
\(
        H_{x_1}=g_0H_{x_0}g_0^{-1}.
\)
The map
\(
        f\longmapsto \bigl[x\mapsto f(g_0^{-1}\hbullet x)\bigr]
\)
is an isomorphism from \(S^{H_{x_0}}\) onto \(S^{H_{x_1}}\). Hence
\(
        \dim S^{H_{x_0}}=\dim S^{H_{x_1}},
\)
and Lemma~\ref{lem:strong-accessibility-stabilizer} shows the claim.
\end{remark}

We now express this condition in terms of the Peter--Weyl blocks of
\(C(K)\).

\begin{theorem}
\label{thm:strong-accessibility}

Let \(K\) be a homogeneous \(G\)-space and  \(S\subset C(K)\) be a non-zero finite-dimensional \(G\)-invariant
subspace. Then the following assertions are equivalent.
\begin{itemize}
\item[(1)] \(S\) is strongly accessible. \item[(2)] There exists \(\pi\in\widehat G\) with \(m_\pi(K)=1\) such that
\(
        S=\mathcal E_\pi(K).
\)
\end{itemize}
\end{theorem}

\begin{proof}
By Theorem~\ref{thm:invariant-subspaces}, there are a finite set
\(\Lambda\subset\widehat G\) and subsets
\(I_\pi\subset\{1,\ldots,m_\pi(K)\}\) such that
\[
        S
        =
        \bigoplus_{\pi\in\Lambda}
        \bigoplus_{\ell\in I_\pi}
        \mathcal E_\pi^{(\ell)}(K).
\]
Each space \(\mathcal E_\pi^{(\ell)}(K)\) is \(G\)-equivalent to \(V_\pi\).
Under this equivalence, its \(H\)-invariant part corresponds to
\(V_\pi^H\), where \(H=H_{x_0}\) is the stabilizer of an arbitrary
base point \(x_0\in K\).
 Hence
\[
        \dim \bigl(\mathcal E_\pi^{(\ell)}(K)\bigr)^H
        =
        \dim V_\pi^H
        =
        m_\pi(K).
\]
Consequently,
\[
        \dim S^H
        =
        \sum_{\pi\in\Lambda}
        |I_\pi|\,m_\pi(K).
\]
Assume first that \(S\) is strongly accessible with respect to \(x_0\).
By Lemma~\ref{lem:strong-accessibility-stabilizer}, \(\dim S^H=1\).
The preceding formula then forces exactly one contribution: there is a
single representation \(\pi\in\widehat G\) with \(m_\pi(K)=1\) and
\(|I_\pi|=1\). Since \(m_\pi(K)=1\), this single copy is already the
whole Peter--Weyl block \(\mathcal E_\pi(K)\). Thus
\(
        S=\mathcal E_\pi(K).
\)

Conversely, assume that \(S=\mathcal E_\pi(K)\) for some
\(\pi\in\widehat G\) with \(m_\pi(K)=1\). Then \(\mathcal{E}_\pi(K)\) consists of
one copy equivalent to \(V_\pi\), and its \(H\)-invariant part has
dimension
\(
        \dim V_\pi^H=m_\pi(K)=1.
\)
Hence \(\dim S^H=1\), and
Lemma~\ref{lem:strong-accessibility-stabilizer} shows that \(S\) is
strongly accessible with respect to \(x_0\).
\end{proof}

Combining Theorem~\ref{thm:strong-accessibility} with
Theorem~\ref{thm:accessible-subspaces}, we obtain the following
immediate consequence.

\begin{corollary}
Let \(K\) be a homogeneous \(G\)-space. Then every strongly accessible
subspace of \(C(K)\) is accessible.
\end{corollary}

\subsection{The projection constant machine}
\label{sec:The projection constant machine}
We now combine the characterization of accessibility from
Theorem~\ref{thm:accessible-subspaces} with the kernel formula from
Theorem~\ref{thm:accessible-projection-constant}.

Let
\(K\) be a homogeneous \(G\)-space and fix \(x_0\in K\). For
\(\pi\in\widehat G\), choose an orthonormal basis
\[
        u_1^\pi,\ldots,u_{m_\pi(K)}^\pi
\]
of the admissible directions \(V_\pi^{H_{x_0}}\). Define
\[
    \chi_{\pi,x_0}^K(g\hbullet x_0)
    :=
    \sum_{\ell=1}^{m_\pi(K)}
    \langle u_\ell^\pi,\pi(g)u_\ell^\pi\rangle,
    \qquad g\in G.
\]

This function is well-defined on \(K\) and is independent of the
chosen orthonormal basis of \(V_\pi^{H_{x_0}}\). If
\(x_1=a\hbullet x_0\), then \(\{\pi(a)u_l^\pi:\,l=1,\dots,m_\pi(K)\}\) is an orthonormal basis of \(V_\pi^{H_{x_1}}\). Thus, given \(x=g\hbullet x_1\),
\[
    \chi_{\pi,x_1}^K(x)=\sum_{\ell=1}^{m_\pi(K)}
    \langle \pi(a)u_\ell^\pi,\pi(g)\pi(a)u_\ell^\pi\rangle
    =\chi_{\pi,x_0}^K(a^{-1}ga\hbullet x_0)=
    \chi_{\pi,x_0}^K(a^{-1}\hbullet x),
    \qquad x\in K.
\]
For a fixed base point \(x_0\), we shall abbreviate
\(\chi_{\pi,x_0}^K\) to \(\chi_\pi^K\).

\begin{theorem}
\label{thm:homogeneous-machine}
Let \(K\) be a homogeneous \(G\)-space, let \(x_0\in K\), and let
\(
    S
    =
    \bigoplus_{\pi\in\Lambda}
    \mathcal E_\pi(K)
    \subset C(K),
\)
where \(\Lambda\subset\widehat G\) is finite. Then \(S\) is
\(G\)-accessible. Moreover, the reproducing kernel of \(S\) satisfies
\begin{equation*}
    \mathbf k_S(x_0,x)
    =
    \sum_{\pi\in\Lambda}
    d_\pi\,\chi_\pi^K(x),
    \qquad x\in K.
\end{equation*}
Consequently,
\[
    \boldsymbol\lambda(S)
    =
    \int_K
    \Big|
        \sum_{\pi\in\Lambda}
        d_\pi\,\chi_\pi^K(x)
    \Big|
    \,d\mu_K(x).
\]
\end{theorem}

\begin{proof}
By Theorem~\ref{thm:accessible-subspaces}, the space \(S\) is
\(G\)-accessible.
We first prove the kernel formula. Fix \(\pi\in\Lambda\), and let
\(
    u_1^\pi,\ldots,u_{m_\pi(K)}^\pi
\)
be an orthonormal basis of \(V_\pi^{H_{x_0}}\), and
\(
    v_1^\pi,\ldots,v_{d_\pi}^\pi
\)
an orthonormal basis of \(V_\pi\). Then the functions
\[
    e_{\ell j}^\pi(g\hbullet x_0)
    =
    \sqrt{d_\pi}\,
    \bigl\langle v_j^\pi,\pi(g)u_\ell^\pi\bigr\rangle,
    \qquad
    1\leq \ell\leq m_\pi(K),\quad
    1\leq j\leq d_\pi,
\]
form an orthonormal basis of \(\mathcal E_\pi(K)\).
Therefore, for \(g\in G\),
\[
\begin{aligned}
    \mathbf k_{\mathcal E_\pi(K)}
        (x_0,g\hbullet x_0)
    &=
    \sum_{\ell=1}^{m_\pi(K)}
    \sum_{j=1}^{d_\pi}
    \overline{e_{\ell j}^\pi(x_0)}\,
    e_{\ell j}^\pi(g\hbullet x_0)
    \\
    &=
    d_\pi
    \sum_{\ell=1}^{m_\pi(K)}
    \sum_{j=1}^{d_\pi}
    \overline{
        \bigl\langle v_j^\pi,u_\ell^\pi\bigr\rangle
    }\,
    \bigl\langle
        v_j^\pi,\pi(g)u_\ell^\pi
    \bigr\rangle.
\end{aligned}
\]
Since the inner product is linear in the first variable and
\((v_j^\pi)_{j=1}^{d_\pi}\) is an orthonormal basis of \(V_\pi\),
\[
    \sum_{j=1}^{d_\pi}
    \overline{
        \bigl\langle v_j^\pi,u_\ell^\pi\bigr\rangle
    }\,
    \bigl\langle
        v_j^\pi,\pi(g)u_\ell^\pi
    \bigr\rangle
    =
    \bigl\langle
        u_\ell^\pi,\pi(g)u_\ell^\pi
    \bigr\rangle.
\]
It follows that
\[
    \mathbf k_{\mathcal E_\pi(K)}
        (x_0,g\hbullet x_0)
    =
    d_\pi
    \sum_{\ell=1}^{m_\pi(K)}
    \bigl\langle
        u_\ell^\pi,\pi(g)u_\ell^\pi
    \bigr\rangle
        =
    d_\pi\,\chi_\pi^K(g\hbullet x_0).
\]
The spaces \(\mathcal E_\pi(K)\), \(\pi\in\Lambda\), are mutually
orthogonal. Hence the reproducing kernel of their orthogonal direct
sum is the sum of their reproducing kernels (see again Section~\ref{Reproducing kernels}). Consequently,
\[
    \mathbf k_S(x_0,x)
    =
    \sum_{\pi\in\Lambda}
    \mathbf k_{\mathcal E_\pi(K)}(x_0,x)
    =
    \sum_{\pi\in\Lambda}
    d_\pi\,\chi_\pi^K(x).
\]
The integral formula now follows from
Theorem~\ref{thm:accessible-projection-constant}.
\end{proof}

Corollaries~\ref{cor:left-regular-machine}
and~\ref{cor:biregular-machine} below describe how the preceding theorem
specializes to the left-regular and biregular actions of a compact group
on itself.

\subsection{Groups acting on groups}
\label{sec:Groups acting on groups}

We finally record two natural ways in which a compact group acts on
itself. The contrast between them will be important later.

\subsubsection{Left-regular action}

Let \(G\) be a compact group acting on itself by left multiplication,
\[
        a\hbullet x:=ax,
        \qquad a,x\in G.
\]
The induced representation on \(C(G)\) is the left-regular action
\[
        (L_a f)(x):=f(a^{-1}x),
        \qquad f\in C(G).
\]
For \(\pi\in\widehat G\), let
\[
        \mathcal E_\pi(G)
        :=
        \operatorname{span}
        \big\{x\mapsto \langle v,\pi(x)u\rangle:u,v\in V_\pi\big\}
        \subset C(G)
\]
be the usual matrix coefficient space. Since the stabilizer of the
identity is trivial, every vector in \(V_\pi\) is admissible. Hence
\[
        m_\pi(G)=d_\pi .
\]
Consequently, under the left-regular action,
\[
        \mathcal E_\pi(G)
        =
        \bigoplus_{\ell=1}^{d_\pi}
        \mathcal E_\pi^{(\ell)}(G),
\]
where each \(\mathcal E_\pi^{(\ell)}(G)\) is \(G\)-equivalent to
\(V_\pi\).
More explicitly, if \(v_1,\ldots,v_{d_\pi}\) is an orthonormal basis of
\(V_\pi\), then
\[
        \mathcal E_\pi^{(\ell)}(G)
        :=
        \operatorname{span}
        \big\{x\mapsto \langle v,\pi(x)v_\ell\rangle:v\in V_\pi\big\},
        \qquad
        1\leq \ell\leq d_\pi .
\]
For fixed \(\ell\), the map
\[
        V_\pi\longrightarrow \mathcal E_\pi^{(\ell)}(G),
        \qquad
        v\longmapsto
        \bigl(x\mapsto \sqrt{d_\pi}\,\langle v,\pi(x)v_\ell\rangle\bigr),
\]
realizes this $G$-equivalence.

Thus the Peter--Weyl decomposition from Theorem~\ref{thm:homogeneous-peter-weyl} may be read, under the
left-regular action, as follows: \(L_2(G)\) is the orthogonal sum of the
isotypic components \(\mathcal E_\pi(G)\), and inside
\(\mathcal E_\pi(G)\) the irreducible representation \(V_\pi\) appears
with multiplicity \(d_\pi\).
In particular, a left-invariant subspace of \(C(G)\) may contain only
some of the irreducible copies \(\mathcal E_\pi^{(\ell)}(G)\); see
Theorem~\ref{thm:invariant-subspaces}. Hence, by
Theorem~\ref{thm:accessible-subspaces}, left-invariance alone does not
imply accessibility unless the action is multiplicity-free (for more details see Subsection~\ref{Invariance and accessibility in the left- and biregular case}).

\subsubsection{Biregular action}
\label{sec:Biregular action}

We now let \(G\times G\) act on \(G\) by
\[
        (a,b)\hbullet x
        :=
        axb^{-1},
        \qquad
        a,b,x\in G.
\]
The induced representation on \(C(G)\) is given by
\[
        ((a,b)\hbullet f)(x)
        :=
        f(a^{-1}xb),
        \qquad
        f\in C(G).
\]
For \(\pi\in\widehat G\), we again write
\[
\mathcal E_\pi(G)
:=
\operatorname{span}
\bigl\{
x\longmapsto\langle v,\pi(x)u\rangle:
u,v\in V_\pi
\bigr\}.
\]
These spaces are invariant under the biregular action.
We recall that the irreducible unitary representations of \(G\times G\)
are precisely the outer tensor products
\[
        \pi\sboxtimes\sigma,
        \qquad
        \pi,\sigma\in\widehat G,
\]
where \(\pi\sboxtimes\sigma\) acts on \(V_\pi\otimes V_\sigma\) by
\[
        (\pi\sboxtimes\sigma)(a,b)
        :=
        \pi(a)\otimes\sigma(b),
        \qquad
        a,b\in G.
\]
Thus
\[
        \widehat{G\times G}
        \cong
        \widehat G\times\widehat G.
\]
The stabilizer of the identity \(e\in G\) under the biregular action is
the diagonal subgroup
\[
        \Delta G
        :=
        \big\{(g,g):g\in G\big\}.
\]
Hence, if one applies the homogeneous Peter--Weyl Theorem~\ref{thm:homogeneous-peter-weyl} to the
biregular action of \(G\times G\), the decomposition first appears in
the form
\begin{equation} \label{biPW:first}
     L_2(G)
        =
        \widehat{\bigoplus}_{\pi,\sigma\in\widehat G}
        \mathcal E_{\pi\sboxtimes\,\sigma}(G),
\end{equation}
where \(\mathcal E_{\pi\sboxtimes\,\sigma}(G)\) denotes the biregular
isotypic component corresponding to the irreducible representation
\(\pi\sboxtimes\sigma\) of \(G\times G\). Its multiplicity is
\[
        m_{\pi\sboxtimes\,\sigma}(G)
        =
        \dim(V_\pi\otimes V_\sigma)^{\Delta G},
\]
where \((V_\pi\otimes V_\sigma)^{\Delta G}
\subset V_\pi\otimes V_\sigma
\) denotes the subspace of
vectors fixed by the diagonal action
\(
        g\longmapsto \pi(g)\otimes\sigma(g).
\)
The next lemma shows that almost all of these possible summands in~\eqref{biPW:first} vanish:
a component can occur only when \(\sigma\cong\pi^*\), and then it occurs
with multiplicity one.

\begin{lemma}
\label{lem:biregular-multiplicity-one}
Let \(\pi,\sigma\in\widehat G\). Then
\[
        \dim(V_\pi\otimes V_\sigma)^{\Delta G}
        =
        \begin{cases}
        1, & \sigma\cong\pi^*,\\
        0, & \sigma\not\cong\pi^*.
        \end{cases}
\]
In particular, the only irreducible representations that occur in the
biregular Peter--Weyl decomposition of \(L_2(G)\) are
\(
        \pi\sboxtimes\pi^*,\,
                \pi\in\widehat G,
\)
and each of them occurs with multiplicity one.
\end{lemma}

\begin{proof}
Let
\[
        \omega
        =
        \sum_i u_i\otimes w_i
        \in V_\pi\otimes V_\sigma.
\]
To \(\omega\) we associate the linear map
\[
        T_\omega:V_\sigma^*\longrightarrow V_\pi,
        \qquad
        T_\omega(\varphi)
        :=
        \sum_i\varphi(w_i)u_i.
\]
The assignment \(\omega\mapsto T_\omega\) is a linear bijection between
\(V_\pi\otimes V_\sigma\) and the space of linear maps
\(V_\sigma^*\to V_\pi\).
The tensor \(\omega\) is fixed by the diagonal action if and only if
\[
        \sum_i\pi(g)u_i\otimes\sigma(g)w_i
        =
        \sum_i u_i\otimes w_i
        \qquad
        \text{for every }g\in G.
\]
Applying a functional \(\varphi\in V_\sigma^*\) to the second tensor
factor, this condition becomes
\[
        T_\omega\bigl(\sigma^*(g)\varphi\bigr)
        =
        \pi(g)T_\omega(\varphi),
        \qquad
        g\in G,\quad \varphi\in V_\sigma^*.
\]
Thus the diagonal fixed vectors in \(V_\pi\otimes V_\sigma\)
correspond exactly to the intertwining operators
\[
        T_\omega:V_\sigma^*\longrightarrow V_\pi
\]
between \(\sigma^*\) and \(\pi\).
By Schur's lemma~\ref{thm:schur-lemma}, there is no non-zero such operator if
\(\pi\not\cong\sigma^*\). If \(\pi\cong\sigma^*\), the space of such
operators is one-dimensional. Since
\(\pi\cong\sigma^*\) is equivalent to \(\sigma\cong\pi^*\), the stated
formula follows.
\end{proof}

For \(\pi\in\widehat G\), consider the irreducible representation
\[
        \pi\sboxtimes\pi^*
        :
        G\times G
        \longrightarrow
        \mathcal U(V_\pi\otimes V_\pi^*)
\]
given by
\[
        (\pi\sboxtimes\pi^*)(a,b)(u\otimes\varphi)
        :=
        \pi(a)u\otimes\pi^*(b)\varphi,
\]
for \(a,b\in G\), \(u\in V_\pi\), and
\(\varphi\in V_\pi^*\). The next lemma identifies its concrete
realization inside \(C(G)\).

\begin{lemma}
\label{lem:biregular-matrix-coefficients}
For every \(\pi\in\widehat G\), equip
\(V_\pi\otimes V_\pi^*\) with the outer tensor product action
\[
    (a,b)\hbullet(u\otimes\varphi)
    :=
    \pi(a)u\otimes\pi^*(b)\varphi,
    \qquad a,b\in G.
\]
Then the map
\[
    \Theta_\pi:V_\pi\otimes V_\pi^*
    \longrightarrow C(G),
    \qquad
    \Theta_\pi(u\otimes\varphi)(x)
    :=
    \varphi\bigl(\pi(x^{-1})u\bigr),
\]
is a \(G\times G\)-equivariant isomorphism onto the matrix
coefficient space \(\mathcal E_\pi(G)\). Consequently,
\[
    \mathcal E_{\pi\sboxtimes\pi^*}(G)
    =
    \mathcal E_\pi(G)
    \cong
    V_\pi\otimes V_\pi^*
\]
as \(G\times G\)-representations.
\end{lemma}

\begin{proof}
Let \(v_1,\ldots,v_{d_\pi}\) be a basis of \(V_\pi\), and let
\(v_1^*,\ldots,v_{d_\pi}^*\) be the corresponding dual basis. Then
\[
        \Theta_\pi(v_i\otimes v_j^*)(x)
        =
        v_j^*(\pi(x^{-1})v_i),
\]
which span precisely the functions
\(x\mapsto\langle v,\pi(x)u\rangle\), and hence the matrix coefficient
space \(\mathcal E_\pi(G)\) in our convention. Hence the image
of \(\Theta_\pi\) is \(\mathcal E_\pi(G)\).
Moreover, these \(d_\pi^2\) matrix coefficients are linearly
independent. Therefore
\[
        \dim\mathcal E_\pi(G)
        =
        d_\pi^2
        =
        \dim(V_\pi\otimes V_\pi^*),
\]
and \(\Theta_\pi\) is a linear isomorphism onto
\(\mathcal E_\pi(G)\).
It remains to verify equivariance. For \(a,b,x\in G\),
\(u\in V_\pi\), and \(\varphi\in V_\pi^*\), one has
\[
\begin{aligned}
        \Theta_\pi\bigl((a,b)\hbullet (u\otimes\varphi)\bigr)(x)
        &=
        \Theta_\pi
        \bigl(\pi(a)u\otimes\pi^*(b)\varphi\bigr)(x)
        \\
        &=
        (\pi^*(b)\varphi)\bigl(\pi(x^{-1})\pi(a)u\bigr)
        \\
        &=
        \varphi\bigl(\pi(b^{-1}x^{-1}a)u\bigr)
        \\
        &=
        \Theta_\pi(u\otimes\varphi)(a^{-1}xb)
        =
        \bigl((a,b)\hbullet\Theta_\pi(u\otimes\varphi)\bigr)(x).
\end{aligned}
\]
Thus \(\Theta_\pi\) is \(G\times G\)-equivariant.
Since \(\Theta_\pi\) is non-zero and \(G\times G\)-equivariant, and
its domain is irreducible of type \(\pi\sboxtimes\pi^*\), its image
\(\mathcal E_\pi(G)\) is an irreducible \(G\times G\)-subrepresentation
of the corresponding biregular isotypic component:
\[
        \mathcal E_\pi(G) \subset
        \mathcal E_{\pi\sboxtimes\,\pi^*}(G).
\]
By Lemma~\ref{lem:biregular-multiplicity-one}, the representation
\(\pi\sboxtimes\pi^*\) occurs in \(L_2(G)\) with multiplicity one.
Therefore its isotypic component consists of a single irreducible copy,
and the above inclusion must be an equality. Hence
\[
        \mathcal E_{\pi\sboxtimes\,\pi^*}(G)
        = \mathcal E_\pi(G),
\]
and \(\Theta_\pi\) gives the asserted \(G\times G\)-equivariant
identification with \(V_\pi\otimes V_\pi^*\).
\end{proof}

We can now simplify the Peter--Weyl decomposition
\eqref{biPW:first} for the biregular action. By
Lemma~\ref{lem:biregular-multiplicity-one}, the biregular isotypic
component associated with \(\rho\sboxtimes\sigma\) vanishes unless
\(\sigma\simeq\rho^*\), and every non-zero component occurs with
multiplicity one. On the other hand,
Lemma~\ref{lem:biregular-matrix-coefficients} identifies the
non-zero components corresponding to \(\pi\sboxtimes\,\pi^*\) with the
usual matrix coefficient space:
\[
    \mathcal E_{\pi\sboxtimes\,\pi^*}(G)
    =
    \mathcal E_\pi(G)
    \simeq
    V_\pi\otimes V_\pi^*
\]
as \(G\times G\)-representations, where \(V_\pi\otimes V_\pi^*\) is
equipped with the action given by $\pi\sboxtimes\pi^*$. Consequently,
\[
    L_2(G)
    =
    \widehat{\bigoplus}_{\pi\in\widehat G}
    \mathcal E_{\pi\sboxtimes\,\pi^*}(G)
    =
    \widehat{\bigoplus}_{\pi\in\widehat G}
    \mathcal E_\pi(G),
\]
and, equivalently, at the level of \(G\times G\)-representations,
\[
    L_2(G)
    \simeq
    \widehat{\bigoplus}_{\pi\in\widehat G}
    V_\pi\otimes V_\pi^*.
\]
Thus the non-zero isotypic components of the biregular representation
are precisely the matrix coefficient spaces
\(\mathcal E_\pi(G)\), and each occurs with multiplicity one. Hence
the biregular Peter--Weyl decomposition is multiplicity-free.

\begin{remark}
\label{rem:left-vs-biregular}
The same matrix coefficient space \(\mathcal E_\pi(G)\) has two different
representation-theoretic meanings. Under the left-regular action of
\(G\), it decomposes into \(d_\pi\) copies of \(V_\pi\). Under the
biregular action of \(G\times G\), it is irreducible and isomorphic to
\(V_\pi\otimes V_\pi^*\). When \(d_\pi>1\), biregular invariance therefore
imposes a strictly stronger condition on subspaces of
\(\mathcal E_\pi(G)\) than left invariance.
\end{remark}

\subsubsection{Invariance and accessibility in the left- and biregular case}
\label{Invariance and accessibility in the left- and biregular case}

We now apply the preceding results to the two natural actions of a
compact group on itself.

\begin{corollary}
\label{cor:left-regular-accessibility}
Let \(G\) be a compact group and let \(S\subset C(G)\) be
finite-dimensional. Then \(S\) is left-accessible
if and only if there is a finite set \(\Lambda\subset\widehat G\) such
that
\(
        S
        =
        \bigoplus_{\pi\in\Lambda}
        \mathcal E_\pi(G).
\)
\end{corollary}

\begin{proof}
For the left-regular action, \(G\) is a homogeneous \(G\)-space with base
point \(e\) and trivial stabilizer. Hence every vector in \(V_\pi\) is
admissible and
\(
        m_\pi(G)=d_\pi .
\)
Thus each isotypic component decomposes as
\[
        \mathcal E_\pi(G)
        =
        \bigoplus_{\ell=1}^{d_\pi}
        \mathcal E_\pi^{(\ell)}(G)
\]
into \(d_\pi\) copies of \(V_\pi\). The assertion now follows directly
from Theorem~\ref{thm:accessible-subspaces}, which says that left-accessible
subspaces are precisely the finite orthogonal sums of full isotypic
components.
\end{proof}

\begin{corollary}
\label{cor:left-invariance-accessibility-abelian}
Let \(G\) be a compact group. Then every finite-dimensional
left-invariant subspace \(S\subset C(G)\) is left-accessible if and only
if \(G\) is abelian.
\end{corollary}

\begin{proof}
By Corollary~\ref{cor:gelfand-pairs}, invariance implies accessibility
for all finite-dimensional subspaces precisely when the homogeneous
\(G\)-space is multiplicity-free. In the left-regular case this means
\[
        m_\pi(G)=d_\pi=1
        \qquad
        \text{for every }\pi\in\widehat G .
\]
This is equivalent to saying that all irreducible unitary
representations of \(G\) are one-dimensional, which is equivalent to
\(G\) being abelian (see e.g. \cite{folland2016course}).
\end{proof}

\begin{corollary}
\label{cor:biregular-invariance-accessibility}
Let \(G\) be a compact group and let \(S\subset C(G)\) be
finite-dimensional. Then the following assertions are equivalent:
\begin{enumerate}
\item \(S\) is biregular-invariant.
\item \(S\) is biregular-accessible.
\item There is a finite set \(\Lambda\subset\widehat G\) such that
      \(
              S
              =
              \bigoplus_{\pi\in\Lambda}
              \mathcal E_\pi(G).
      \)
\end{enumerate}
\end{corollary}

\begin{proof}
By Remark~\ref{rem:left-vs-biregular}, every matrix coefficient space
\(\mathcal E_\pi(G)\) is irreducible under the biregular action and is
\(G\times G\)-equivalent to \(V_\pi\otimes V_\pi^*\). Thus the
biregular Peter--Weyl decomposition is multiplicity-free.

The claim now follows from Corollary \ref{cor:gelfand-pairs} together with the description of accessible subspaces of Theorem \ref{thm:accessible-subspaces}.


\end{proof}

\begin{corollary}
\label{cor:left-biregular-accessible-subspaces}
Let \(G\) be a compact group and let \(S\subset C(G)\) be
finite-dimensional. Then
\[
    S \text{ is left-accessible}
    \quad\Longleftrightarrow\quad
    S \text{ is biregular-accessible}.
\]
Moreover, for either action, the accessible subspaces are precisely
the finite orthogonal sums
\(
    S
    =
    \bigoplus_{\pi\in\Lambda}
    \mathcal E_\pi(G),
\)
where \(\Lambda\subset\widehat G\) is finite.
\end{corollary}

\begin{proof}
This follows immediately from
Corollary~\ref{cor:left-regular-accessibility} and
Corollary~\ref{cor:biregular-invariance-accessibility}.
\end{proof}

\subsubsection{The projection constant machine in the left- and biregular case}
We now specialize the homogeneous machine
from Section~\ref{sec:The projection constant machine}
to the group case.

\begin{corollary}
\label{cor:left-regular-machine}
Let \(G\) be a compact group and let
\(S  =        \bigoplus_{\pi\in\Lambda}
        \mathcal E_\pi(G)
        \),
        where
        $\Lambda\subset\widehat G$ is finite.
        Then
\[
        \boldsymbol\lambda(S)
        =
        \int_G
        \Big|
        \sum_{\pi\in\Lambda}
        d_\pi\,\chi_\pi(g)
         \Big|
        \,d\mu_G(g).
\]
\end{corollary}

\begin{proof}
By Corollary~\ref{cor:left-regular-accessibility}, the space \(S\) is
accessible for the left-regular action. We apply
Theorem~\ref{thm:homogeneous-machine} with base point \(e\in G\). 
Since the stabilizer of \(e\) is trivial, all vectors in \(V_\pi\) are
admissible. If \(u_1^\pi,\ldots,u_{d_\pi}^\pi\) is an orthonormal basis
of \(V_\pi\), then
\[
        \chi_\pi^G(g)
        =
        \sum_{\ell=1}^{d_\pi}
        \langle u_\ell^\pi,\pi(g)u_\ell^\pi\rangle
        =
                \overline{\chi_\pi(g)}.
\]
Therefore, Theorem~\ref{thm:homogeneous-machine} gives
\[
\boldsymbol\lambda(S)
=
\int_G
\Big|
\sum_{\pi\in\Lambda}
d_\pi\overline{\chi_\pi(g)}
\Big|
\,d\mu_G(g)
=
\int_G
\Big|
\sum_{\pi\in\Lambda}
d_\pi\chi_\pi(g)
\Big|
\,d\mu_G(g),
\]
which proves the formula.
\end{proof}


The biregular version is worth recording separately, since by Corollary~\ref{cor:biregular-invariance-accessibility} biregular invariance alone guarantees the structural hypothesis required in
Corollary~\ref{cor:left-regular-machine}.

\begin{corollary}
\label{cor:biregular-machine}
Let \(G\) be a compact group and let \(S\subset C(G)\) be a
finite-dimensional biregular-invariant subspace. Then there is a finite set \(\Lambda\subset\widehat G\) such
that
\(
        S
        =
        \bigoplus_{\pi\in\Lambda}
        \mathcal E_\pi(G),
\)
and
\[
        \boldsymbol\lambda(S)
        =
        \int_G
         \Big|
        \sum_{\pi\in\Lambda}
        d_\pi\,\chi_\pi(g)
        \Big|
        \,d\mu_G(g).
\]
\end{corollary}


\section{Normalized projection constants in high dimensions}
Given a concrete function space contained in $C(K)$, the abstract part of the paper reduces the computation of its projection constant to a representation-theoretic problem: identifying the underlying symmetries and the relevant isotypic components, establishing accessibility, and deriving an explicit integral formula in terms of a reproducing-kernel slice.

This, however, is only the first step. A further analytic problem remains, namely, to determine the asymptotic behavior of the resulting integral as the ambient dimension tends to infinity. This step can be delicate and, in the examples considered below, requires a combination of probabilistic and combinatorial arguments.

We carry out this program in three natural high-dimensional settings. Although they follow a common scheme, the underlying harmonic analysis is quite different. On the torus, the machinery reduces to ordinary Fourier analysis; on Euclidean spheres, it becomes spherical harmonic analysis; and on the unitary group, it leads to genuinely noncommutative harmonic analysis.

\subsection{Fourier analysis on the torus}
\label{sec:torus-application}

We begin the applications of the abstract theory with polynomial spaces on
\(\ell_\infty^n(\mathbb C)\). In this setting, the relevant compact
homogeneous space is the torus, and the repre\-sentation-theoretic machinery
reduces to ordinary Fourier analysis.

\begin{theorem}
\label{thm:torus-polynomial-projection-asymptotic}
For every fixed \(d\in\mathbb N\),
\[
    \lim_{n\to\infty}
    \frac{
        \boldsymbol\lambda\bigl(\mathcal P_d(\ell_\infty^n(\mathbb C))\bigr)
    }{
        \sqrt{\binom{n+d-1}{d}}
    }
    =
    \frac{\Gamma\bigl(1+\frac d2\bigr)}{\sqrt{d!}}.
\]
Equivalently,
\[
    \lim_{n\to\infty}
    \frac{
        \boldsymbol\lambda\bigl(\mathcal P_d(\ell_\infty^n(\mathbb C))\bigr)
    }{
        n^{d/2}
    }
    =
    \frac{\Gamma\bigl(1+\frac d2\bigr)}{d!}.
\]
\end{theorem}

The asymptotics of \(\boldsymbol\lambda\bigl(\mathcal P_d(\ell_p^n(\mathbb C))\bigr)\), \(1\le p\le\infty\) was established in \cite{defant2011bohr}, up to a multiplicative constant of size 
\(C^d\)
 for some absolute constant 
\(C>0\).

The proof will be completed at the end of this subsection in~\ref{proof::torus-polynomial-projection-asymptotic}, after an analysis
of the corresponding Fourier sums.
The same analysis will also yield the corresponding asymptotic formula
for tetrahedral polynomials; see
Theorem~\ref{thm:torus-tetrahedral-projection-asymptotic} below.

As usual, let
\[
    \mathbb T^n
    :=
    \bigl\{
        z=(z_1,\ldots,z_n)\in\mathbb C^n:
        |z_j|=1 \text{ for }1\le j\le n
    \bigr\}
\]
be the \(n\)-dimensional torus, endowed with its normalized Haar measure
\(\mu_{\mathbb T^n}\). We regard \(\mathbb T^n\) as a compact abelian group acting on
itself by left multiplication. Its characters are the monomials
\[
    z^\alpha
    :=
    z_1^{\alpha_1}\cdots z_n^{\alpha_n},
    \qquad
    \alpha\in\mathbb Z^n.
\]
Throughout this subsection, all \(L^p\)-norms are taken with respect to
the Haar measure \(\mu_{\mathbb T^n}\).
For \(d\in\mathbb N\), define
\[
    \mathcal P_d(\mathbb T^n)
    :=
    \operatorname{span}
    \bigl\{
        z^\alpha:
        \alpha\in\mathbb N_0^n,\ |\alpha|=d
    \bigr\} \subset C(\mathbb T^n),
\]
and let
\[
    \mathcal T_d(\mathbb T^n)
    :=
    \operatorname{span}
    \bigl\{
        z^\alpha:
        \alpha\in\{0,1\}^n,\ |\alpha|=d
    \bigr\}
\]
be its tetrahedral subspace.
Both spaces are finite-dimensional spans of characters. Therefore,
Theorem~\ref{thm:abelian-accessibility} shows that they are
\(\mathbb T^n\)-accessible and identifies their projection constants with
the \(L^1\)-norms of their Fourier kernels. More precisely, put
\[
    S_{n,d}(z)
    :=
    \sum_{\substack{\alpha\in\mathbb N_0^n\\|\alpha|=d}}
    z^\alpha
\]
and
\[
    U_{n,d}(z)
    :=
    \sum_{\substack{A\subset\{1,\ldots,n\}\\|A|=d}}
    z_A,
    \qquad
    z_A:=\prod_{j\in A}z_j.
\]
Then by Theorem~\ref{thm:abelian-accessibility}, 
\begin{equation}
\label{eq:torus-projection-kernels}
    \boldsymbol\lambda\bigl(\mathcal P_d(\mathbb T^n)\bigr)
    =
    \|S_{n,d}\|_1,
    \qquad
    \boldsymbol\lambda\bigl(\mathcal T_d(\mathbb T^n)\bigr)
    =
    \|U_{n,d}\|_1.
\end{equation}
Thus the asymptotic problem is reduced to determining the \(L^1/L^2\)
ratios of these two explicit Fourier sums.

\subsubsection{Reduction to the multilinear part}

Since the characters of \(\mathbb T^n\) form an orthonormal basis of
\(L^2(\mathbb T^n)\), we have
\begin{equation}
\label{eq:torus-kernel-l2-norms}
    \|S_{n,d}\|_2^2
    =
    \binom{n+d-1}{d},
    \qquad
    \|U_{n,d}\|_2^2
    =
    \binom nd.
\end{equation}
For fixed \(d\),
\begin{equation} \label{DIMO1}
    \binom{n+d-1}{d}
    =
    \frac{n^d}{d!}
    \prod_{j=0}^{d-1}
    \left(1+\frac{j}{n}\right)
    =
    \frac{n^d}{d!}
    \left(1+O_d\left(\frac1n\right)\right),
\end{equation}
whereas
\begin{equation} \label{DIMO2}
    \binom nd
    =
    \frac{n^d}{d!}
    \prod_{j=0}^{d-1}
    \left(1-\frac{j}{n}\right)
    =
    \frac{n^d}{d!}
    \left(1+O_d\left(\frac1n\right)\right).
\end{equation}
Consequently,
\begin{equation}
\label{eq:torus-kernel-l2-asymptotics}
    \|S_{n,d}\|_2
    \sim
    \|U_{n,d}\|_2
    \sim
    \frac{n^{d/2}}{\sqrt{d!}}.
\end{equation}
We now separate the multilinear part of \(S_{n,d}\). Write
\[
    S_{n,d}=U_{n,d}+V_{n,d},
\]
where
\[
    V_{n,d}
    :=
    \sum_{\substack{
        \alpha\in\mathbb N_0^n,\ |\alpha|=d\\
        \max_j\alpha_j\geq 2
    }}
    z^\alpha.
\]
Thus \(V_{n,d}\) consists precisely of the monomials in which at least one
variable occurs with exponent greater than one.

\begin{lemma}
\label{lem:torus-remainder}
For every fixed \(d\in\mathbb N\),
\[
    \frac{\|V_{n,d}\|_2}{\|S_{n,d}\|_2}
    \longrightarrow 0
    \qquad\text{and}\qquad
    \frac{\|V_{n,d}\|_1}{\|S_{n,d}\|_2}
    \longrightarrow 0,
    \qquad n\to\infty.
\]
\end{lemma}

\begin{proof}
The Fourier supports of \(U_{n,d}\) and \(V_{n,d}\) are disjoint.
Therefore, by orthogonality and
\eqref{eq:torus-kernel-l2-norms},
\[
    \|V_{n,d}\|_2^2
    =
    \|S_{n,d}\|_2^2-\|U_{n,d}\|_2^2
    =
    \binom{n+d-1}{d}-\binom nd.
\]
For fixed \(d\), both binomial coefficients are polynomials in \(n\) of
degree \(d\) with the same leading coefficient \(1/d!\). Hence their
difference has degree at most \(d-1\), and therefore
\[
    \|V_{n,d}\|_2^2=O_d(n^{d-1}).
\]
Since
\(
    \|S_{n,d}\|_2^2
    \sim
    \frac{n^d}{d!}
\)
by Equation~\eqref{eq:torus-kernel-l2-asymptotics},
it follows that
\[
    \frac{\|V_{n,d}\|_2}{\|S_{n,d}\|_2}
    =
    O_d(n^{-1/2})
    \longrightarrow 0.
\]
Finally, because \(\mu_{\mathbb T^n}\) is a probability measure,
\(\|V_{n,d}\|_1\leq \|V_{n,d}\|_2\), which proves the second assertion.
\end{proof}

\begin{lemma}
\label{lem:torus-reduction}
For every fixed \(d\in\mathbb N\),
\[
    \frac{\|S_{n,d}\|_1}{\|S_{n,d}\|_2}
    -
    \frac{\|U_{n,d}\|_1}{\|U_{n,d}\|_2}
    \longrightarrow 0,
    \qquad n\to\infty.
\]
\end{lemma}

\begin{proof}
By the reverse triangle inequality,
\[
    \bigl|
        \|S_{n,d}\|_1-\|U_{n,d}\|_1
    \bigr|
    \leq
    \|V_{n,d}\|_1.
\]
Hence
\[
\begin{aligned}
    \left|
        \frac{\|S_{n,d}\|_1}{\|S_{n,d}\|_2}
        -
        \frac{\|U_{n,d}\|_1}{\|U_{n,d}\|_2}
    \right|
    \leq
    \frac{\|V_{n,d}\|_1}{\|S_{n,d}\|_2}
+
    \|U_{n,d}\|_1
    \left|
        \frac{1}{\|S_{n,d}\|_2}
        -
        \frac{1}{\|U_{n,d}\|_2}
    \right|.
\end{aligned}
\]
Since \(\|U_{n,d}\|_1\leq\|U_{n,d}\|_2\), the second term is bounded by
\[
    \left|
        \frac{\|U_{n,d}\|_2}{\|S_{n,d}\|_2}-1
    \right|.
\]
The first term tends to zero by
Lemma~\ref{lem:torus-remainder}, while
\[
    \frac{\|U_{n,d}\|_2^2}{\|S_{n,d}\|_2^2}
    =
    \frac{\binom nd}{\binom{n+d-1}{d}}
    \longrightarrow 1
\]
by the asymptotic formulas given in
\eqref{DIMO1} and \eqref{DIMO2}.
 This proves the claim.
\end{proof}

It remains to determine the asymptotic \(L^1/L^2\)-ratio of
\(U_{n,d}\).

\subsubsection{Gaussian approximation of the multilinear sum}

For \(n\geq d\), let
\[
    N_{n,d}:=\binom nd,
    \qquad
    X_{n,d}:=\frac{U_{n,d}}{\sqrt{N_{n,d}}}.
\]
We regard the coordinate functions \(z_1,\ldots,z_n\) as independent
Steinhaus random variables. Since the monomials
\[
    z_A=\prod_{j\in A}z_j,
    \qquad
    A\subset\{1,\ldots,n\},
    \qquad
    |A|=d,
\]
form an orthonormal family in \(L^2(\mathbb T^n)\),
\[
    \mathbb E|X_{n,d}|^2
    =
    \frac{\|U_{n,d}\|_2^2}{N_{n,d}}
    =
    1.
\]
Let \(W\) denote a standard circular complex Gaussian random variable, 
normalized by
\[
    W=\frac{1}{\sqrt2}(X+iY),
\]
where \(X\) and \(Y\) are independent standard real Gaussian random
variables. Then \(\mathbb E|W|^2=1\), and
\begin{equation}
\label{eq:complex-gaussian-moments}
    \mathbb E|W|^r
    =
    \Gamma\left(1+\frac r2\right),
    \qquad r>-2.
\end{equation}

\begin{proposition}
\label{thm:torus-multilinear-limit}
For every fixed \(d\in\mathbb N\), the random variables \(X_{n,d}\)
converge in distribution to \(W^d/\sqrt{d!}\):
\[
    X_{n,d}
    \Longrightarrow
    \frac{W^d}{\sqrt{d!}},
    \qquad n\to\infty.
\]
Moreover,
\begin{equation}
\label{eq:torus-full-l1-l2-limit-first}
    \frac{\|U_{n,d}\|_1}{\|U_{n,d}\|_2}
    =
    \mathbb E|X_{n,d}|
    \longrightarrow
    \frac{\Gamma\bigl(1+\frac d2\bigr)}{\sqrt{d!}}
\end{equation}
and
\begin{equation}
\label{eq:torus-full-l1-l2-limit}
    \frac{\|S_{n,d}\|_1}{\|S_{n,d}\|_2}
    \longrightarrow
    \frac{\Gamma\bigl(1+\frac d2\bigr)}{\sqrt{d!}}.
\end{equation}
\end{proposition}

\begin{proof}
Regarding \(z_1,\ldots,z_n\) as the coordinate functions on
\((\mathbb T^n,\mu_{\mathbb T^n})\), set
\[
    Y_n:=\frac{1}{\sqrt n}\sum_{j=1}^n z_j.
\]
We first prove the \(L^2\)-approximation
\[
    X_{n,d}-\frac{Y_n^d}{\sqrt{d!}}
    \longrightarrow 0
    \qquad\text{in }L^2.
\]
Expanding the \(d\)-th power gives
\begin{equation}
\label{eq:power-multilinear-decomposition}
    \left(\sum_{j=1}^n z_j\right)^d
    =
    d!\,U_{n,d}+E_{n,d},
\end{equation}
where \(E_{n,d}\) is a linear combination of those degree-\(d\)
monomials involving at most \(d-1\) distinct variables. More explicitly,
\[
    E_{n,d}
    =
    \sum_{\substack{\alpha\in\mathbb N_0^n,\ |\alpha|=d\\
                    \alpha\notin\{0,1\}^n}}
    \binom{d}{\alpha}z^\alpha,
    \qquad
    \binom d\alpha
    :=
    \frac{d!}{\alpha_1!\cdots\alpha_n!}.
\]
For fixed \(d\), every monomial occurring in \(E_{n,d}\) involves
\(r\leq d-1\) distinct variables. For each such \(r\), there are
\(\binom nr\) ways to choose the variables and
\(\binom{d-1}{r-1}\) ways to assign them positive exponents summing to
\(d\). Hence the number of distinct monomials in \(E_{n,d}\) is
\[
    \sum_{r=1}^{d-1}
    \binom nr\binom{d-1}{r-1}
    =
    O_d(n^{d-1}).
\] Moreover,  the  coefficient of each monomial is bounded by \(d!\).
Since distinct characters are orthogonal in \(L^2(\mathbb T^n)\), it
follows that
\begin{equation}
\label{eq:power-remainder-bound}
    \|E_{n,d}\|_2^2
    =
    O_d(n^{d-1}).
\end{equation}
Using~\eqref{eq:power-multilinear-decomposition}, we obtain
\[
    \frac{Y_n^d}{\sqrt{d!}}
    =
    \frac{\sqrt{d!}}{n^{d/2}}\,U_{n,d}
    +
    \frac{E_{n,d}}{\sqrt{d!}\,n^{d/2}}.
\]
Therefore,
\begin{align*}
    \left\|
        X_{n,d}-\frac{Y_n^d}{\sqrt{d!}}
    \right\|_2
    \leq
    \left|
        \frac{1}{\sqrt{N_{n,d}}}
        -
        \frac{\sqrt{d!}}{n^{d/2}}
    \right|
    \|U_{n,d}\|_2
    +
    \frac{\|E_{n,d}\|_2}{\sqrt{d!}\,n^{d/2}}
=
    \left|
        1-\sqrt{\frac{d!N_{n,d}}{n^d}}
    \right|
    +
    \frac{\|E_{n,d}\|_2}{\sqrt{d!}\,n^{d/2}}.
\end{align*}
Now, by \eqref{DIMO2},
\(
    \frac{d!N_{n,d}}{n^d}
    \longrightarrow 1,
\)
while~\eqref{eq:power-remainder-bound} yields
\(
    \frac{\|E_{n,d}\|_2}{n^{d/2}}
    =
    O_d\bigl(n^{-1/2}\bigr).
\)
Consequently,
\begin{equation}
\label{eq:multilinear-l2-approximation}
    \left\|
        X_{n,d}-\frac{Y_n^d}{\sqrt{d!}}
    \right\|_2
    \longrightarrow 0.
\end{equation}
In particular, by Chebyshev's inequality,
\[
    X_{n,d}-\frac{Y_n^d}{\sqrt{d!}}
    \longrightarrow 0
    \qquad\text{in probability}.
\]
Viewed as random vectors in \(\mathbb R^2\), the Steinhaus variables
\(z_1,z_2,\ldots\) are independent, have mean zero, and have covariance
matrix \(\frac12 I_2\). Hence the two-dimensional central limit theorem
gives
\[
    Y_n
    =
    \frac{1}{\sqrt n}\sum_{j=1}^n z_j
    \Longrightarrow W.
\]
By the continuous mapping theorem,
\[
    \frac{Y_n^d}{\sqrt{d!}}
    \Longrightarrow
    \frac{W^d}{\sqrt{d!}}.
\]
Together with~\eqref{eq:multilinear-l2-approximation}, Slutsky's theorem
therefore yields
\[
    X_{n,d}
    \Longrightarrow
    \frac{W^d}{\sqrt{d!}}.
\]
It remains to pass to first absolute moments. Since
\(\mathbb E|X_{n,d}|^2=1\) for every \(n\geq d\), the family
\(\bigl(|X_{n,d}|\bigr)_{n\geq d}\) is uniformly integrable.
Consequently,
\[
    \mathbb E|X_{n,d}|
    \longrightarrow
    \mathbb E\left|\frac{W^d}{\sqrt{d!}}\right|
    =
    \frac{\mathbb E|W|^d}{\sqrt{d!}}.
\]
Using~\eqref{eq:complex-gaussian-moments}, we conclude that
\[
    \mathbb E|X_{n,d}|
    \longrightarrow
    \frac{\Gamma\bigl(1+\frac d2\bigr)}{\sqrt{d!}}.
\]
Finally,
\[
    \mathbb E|X_{n,d}|
    =
    \frac{\|U_{n,d}\|_1}{\sqrt{N_{n,d}}}
    =
    \frac{\|U_{n,d}\|_1}{\|U_{n,d}\|_2},
\]
which proves the limit in~\eqref{eq:torus-full-l1-l2-limit-first}. The limit in~\eqref{eq:torus-full-l1-l2-limit} then is an immediate consequence of Lemma~\ref{lem:torus-reduction}.
\end{proof}

\subsubsection{Projection constants}\label{proof::torus-polynomial-projection-asymptotic}

We now return to polynomial spaces on \(\ell_\infty^n(\mathbb C)\).

\begin{proof}[Proof of
Theorem~\ref{thm:torus-polynomial-projection-asymptotic}]
By the maximum modulus principle, the restriction map
\[
    \mathcal P_d(\ell_\infty^n(\mathbb C))
    \longrightarrow
    \mathcal P_d(\mathbb T^n)
\]
is an isometric isomorphism. Hence, by
Theorem~\ref{thm:abelian-accessibility} and
Equation~\eqref{eq:torus-projection-kernels},
\[
    \boldsymbol\lambda\bigl(\mathcal P_d(\ell_\infty^n(\mathbb C))\bigr)
    =
    \boldsymbol\lambda\bigl(\mathcal P_d(\mathbb T^n)\bigr)
    =
    \|S_{n,d}\|_1.
\]
Moreover, by Equation~\eqref{eq:torus-kernel-l2-norms}, one has
\(
    \|S_{n,d}\|_2
    =
    \sqrt{\binom{n+d-1}{d}}.
\)
Therefore the limit in Equation~\eqref{eq:torus-full-l1-l2-limit} gives
\[
    \frac{
        \boldsymbol\lambda\bigl(\mathcal P_d(\ell_\infty^n(\mathbb C))\bigr)
    }{
        \sqrt{\binom{n+d-1}{d}}
    }
    =
    \frac{\|S_{n,d}\|_1}{\|S_{n,d}\|_2}
    \longrightarrow
    \frac{\Gamma\bigl(1+\frac d2\bigr)}{\sqrt{d!}}.
\]
Finally,
\(
    \sqrt{\binom{n+d-1}{d}}
    \sim
    \frac{n^{d/2}}{\sqrt{d!}}
\)
 by Equation~\eqref{DIMO1},
which yields the equivalent formulation
\[
    \frac{
        \boldsymbol\lambda\bigl(\mathcal P_d(\ell_\infty^n(\mathbb C))\bigr)
    }{
        n^{d/2}
    }
    \longrightarrow
    \frac{\Gamma\bigl(1+\frac d2\bigr)}{d!}.
\]
This completes the proof.
\end{proof}

The same argument gives the tetrahedral analogue. Let
\(\mathcal T_d(\ell_\infty^n(\mathbb C))\) denote the space of
tetrahedral \(d\)-homogeneous polynomials
\[
    P(z)
    =
    \sum_{\substack{A\subset\{1,\ldots,n\}\\|A|=d}}
    c_A\prod_{j\in A}z_j,
\]
equipped with the supremum norm on the unit ball of
\(\ell_\infty^n(\mathbb C)\).

\begin{theorem}
\label{thm:torus-tetrahedral-projection-asymptotic}
For every fixed \(d\in\mathbb N\),
\[
    \lim_{n\to\infty}
    \frac{
        \boldsymbol\lambda\bigl(\mathcal T_d(\ell_\infty^n(\mathbb C))\bigr)
    }{
        \sqrt{\binom nd}
    }
    =
    \frac{\Gamma\bigl(1+\frac d2\bigr)}{\sqrt{d!}}.
\]
Equivalently,
\[
    \lim_{n\to\infty}
    \frac{
        \boldsymbol\lambda\bigl(\mathcal T_d(\ell_\infty^n(\mathbb C))\bigr)
    }{
        n^{d/2}
    }
    =
    \frac{\Gamma\bigl(1+\frac d2\bigr)}{d!}.
\]
\end{theorem}

\begin{proof}
By the maximum modulus principle, restriction to \(\mathbb T^n\)
identifies
\[
    \mathcal T_d(\ell_\infty^n(\mathbb C))
    \quad\text{and}\quad
    \mathcal T_d(\mathbb T^n)
\]
isometrically. Thus Theorem~\ref{thm:abelian-accessibility} and
Equation~\eqref{eq:torus-projection-kernels} give
\[
    \boldsymbol\lambda\bigl(\mathcal T_d(\ell_\infty^n(\mathbb C))\bigr)
    =
    \boldsymbol\lambda\bigl(\mathcal T_d(\mathbb T^n)\bigr)
    =
    \|U_{n,d}\|_1.
\]
Since
\(
    \|U_{n,d}\|_2
    =
    \sqrt{\binom nd},
\)
the limit in Equation~\eqref{eq:torus-full-l1-l2-limit-first} yields the first assertion.
The second assertion then follows immediately from Equation~\eqref{DIMO2}.
\end{proof}

\subsection{Spherical harmonic analysis on Euclidean spheres}
\label{sec:hilbertian-spheres}

We now apply the homogeneous-space machinery to real and complex Euclidean
spheres.

In both settings, the classical decompositions into spherical harmonics
coincide with the Peter--Weyl decompositions associated with the natural
transitive actions of the corresponding compact groups:
$\mathcal O_n$ on the real sphere and $\mathcal U_n$ on the complex sphere.
This allows us to identify the invariant and accessible subspaces, compute
their reproducing kernels, and derive explicit formulas for the associated
projection constants.

There is, however, a minor distinction between the two settings. On the
complex sphere, Peter--Weyl theory applies directly to the naturally occurring
complex-valued function spaces. On the real sphere, we first work with the
natural complexifications of the spaces of real spherical harmonics and
homogeneous polynomials, and then pass back to their real forms using
Lemma~\ref{lem: acR-acC}. This will ultimately yield the projection constants of
\(
\mathcal P_d\bigl(\ell_2^n(\mathbb R)\bigr)
=
\mathcal P_d\bigl(\ell_2^n(\mathbb R);\mathbb R\bigr).
\)

Besides the structural description of invariant subspaces, we are interested
in the high-dimensional behavior of the resulting projection constants. For
fixed degree $d$, the real case leads to a limit expressed in terms of the
$d$-th probabilists' Hermite polynomial, whereas the complex case yields the
corresponding complex Gaussian moment.

\subsubsection{The real Euclidean sphere}
\label{sec:real-sphere}

Throughout this subsection, we assume that $n\geq 2$. Let
\[
\mathbb S_{\mathbb R}^{n-1}
:=
\bigl\{x\in\mathbb R^n:\|x\|_2=1\bigr\}
\]
be the real Euclidean sphere, endowed with its normalized rotation-invariant
measure $\sigma_n$. The orthogonal group $\mathcal O_n$ acts transitively on
$\mathbb S_{\mathbb R}^{n-1}$ by
\[
A\hbullet x:=Ax,
\qquad
A\in\mathcal O_n,\quad x\in\mathbb S_{\mathbb R}^{n-1}.
\]
We choose the base point $e_1=(1,0,\ldots,0)$. Its stabilizer is
\[
\mathcal O_n(e_1)
:=
\bigl\{A\in\mathcal O_n:Ae_1=e_1\bigr\},
\]
which is naturally isomorphic to $\mathcal O_{n-1}$, acting orthogonally on
$e_1^\perp$. Thus $\mathbb S_{\mathbb R}^{n-1}$ is a homogeneous
$\mathcal O_n$-space.

Although the underlying sphere is real, we first work with complex-valued
functions, so that the Peter--Weyl theory developed in
Subsections~1.3--1.6 applies directly. We shall then pass to the corresponding
real forms and, finally, to spaces of real homogeneous polynomials on
$\ell_2^n(\mathbb R)$.

For $d\in\mathbb N_0$, let
\[
\mathcal H_d\bigl(\mathbb S_{\mathbb R}^{n-1};\mathbb C\bigr)
\]
denote the space of restrictions to $\mathbb S_{\mathbb R}^{n-1}$ of
complex-valued $d$-homogeneous harmonic polynomials on $\mathbb R^n$, that is,
polynomials with complex coefficients which are homogeneous of degree $d$ and
annihilated by the real Laplacian. Likewise, let
\[
\mathcal P_d\bigl(\mathbb S_{\mathbb R}^{n-1};\mathbb C\bigr)
\]
be the space of restrictions to $\mathbb S_{\mathbb R}^{n-1}$ of
complex-valued $d$-homogeneous polynomials on $\mathbb R^n$. Equivalently,
\[
\mathcal H_d\bigl(\mathbb S_{\mathbb R}^{n-1};\mathbb C\bigr)
=
\mathcal H_d\bigl(\mathbb S_{\mathbb R}^{n-1};\mathbb R\bigr)
\oplus
i\mathcal H_d\bigl(\mathbb S_{\mathbb R}^{n-1};\mathbb R\bigr),
\]
and
\[
\mathcal P_d\bigl(\mathbb S_{\mathbb R}^{n-1};\mathbb C\bigr)
=
\mathcal P_d\bigl(\mathbb S_{\mathbb R}^{n-1};\mathbb R\bigr)
\oplus
i\mathcal P_d\bigl(\mathbb S_{\mathbb R}^{n-1};\mathbb R\bigr).
\]
Each space $\mathcal H_d(\mathbb S_{\mathbb R}^{n-1};\mathbb C)$ is
finite-dimensional and $\mathcal O_n$-invariant, and these spaces are pairwise
orthogonal in $L_2(\mathbb S_{\mathbb R}^{n-1},\sigma_n;\mathbb C)$. The
classical harmonic decomposition gives
\begin{equation}
\label{eq:P-d-harmonic-decomposition}
\mathcal P_d\bigl(\mathbb S_{\mathbb R}^{n-1};\mathbb C\bigr)
=
\bigoplus_{j=0}^{\lfloor d/2\rfloor}
\mathcal H_{d-2j}\bigl(\mathbb S_{\mathbb R}^{n-1};\mathbb C\bigr);
\end{equation}
see, for instance,
\cite[Theorem~2.18]{atkinson2012spherical}  or \cite[Proposition~2.2]{defant2026minimal}.

The restrictions to $\mathbb S_{\mathbb R}^{n-1}$ of real polynomials on
$\mathbb R^n$ form a unital subalgebra of
$C(\mathbb S_{\mathbb R}^{n-1};\mathbb R)$ which separates points. Hence the
real Stone--Weierstrass theorem implies that their linear span is dense in
$C(\mathbb S_{\mathbb R}^{n-1};\mathbb R)$. After complexification, the
restrictions of complex-valued polynomials are dense in
$C(\mathbb S_{\mathbb R}^{n-1};\mathbb C)$. Together with
\eqref{eq:P-d-harmonic-decomposition}, this shows that
\[
\bigoplus_{d=0}^{\infty}
\mathcal H_d\bigl(\mathbb S_{\mathbb R}^{n-1};\mathbb C\bigr)
\]
is dense in $C(\mathbb S_{\mathbb R}^{n-1};\mathbb C)$. Consequently, one
obtains the classical orthogonal decomposition
\begin{equation}
\label{eq:L2-harmonic-decomposition}
L_2\bigl(\mathbb S_{\mathbb R}^{n-1},\sigma_n;\mathbb C\bigr)
=
\widehat{\bigoplus}_{d\geq 0}
\mathcal H_d\bigl(\mathbb S_{\mathbb R}^{n-1};\mathbb C\bigr).
\end{equation}
Our next goal is to show that each block
$\mathcal H_d(\mathbb S_{\mathbb R}^{n-1};\mathbb C)$ is strongly
accessible. To this end, we study its $\mathcal O_n(e_1)$-fixed subspace,
\[
\begin{split}
\mathcal H_d\bigl(\mathbb S_{\mathbb R}^{n-1};\mathbb C\bigr)^{\mathcal O_n(e_1)}
:=
\bigl\{
f\in\mathcal H_d\bigl(\mathbb S_{\mathbb R}^{n-1};\mathbb C\bigr):
f(u\xi)=f(\xi) 
\text{ for all }u\in\mathcal O_n(e_1)
\text{ and }\xi\in\mathbb S_{\mathbb R}^{n-1}
\bigr\}.
\end{split}
\]
For $d\geq 0$, define the normalized zonal harmonic of degree $d$ by
\[
L^\diamond_{n,d}(t)
:=
\sum_{j=0}^{\lfloor d/2\rfloor}
b_j(n,d)t^{d-2j}(1-t^2)^j,
\qquad
-1\leq t\leq 1,
\]
where
\[
b_j(n,d)
=
\frac{(-1)^j d!\,\Gamma\bigl(\frac{n-1}{2}\bigr)}
{4^j j!(d-2j)!\,
\Gamma\bigl(j+\frac{n-1}{2}\bigr)}.
\]
We shall use the following well-known lemma; see
\cite[Section~2.1.2]{atkinson2012spherical} and in our context
\cite[Proposition~3.4 and Corollary~3.5]{defant2026minimal}.

\begin{lemma}
\label{lem:diamond-function}
The function
\[
L^\diamond_{n,d}\bigl(\langle\,\cdot\,,e_1\rangle\bigr)
\in
C\bigl(\mathbb S_{\mathbb R}^{n-1};\mathbb C\bigr)
\]
is the unique element
\[
f\in
\mathcal H_d\bigl(\mathbb S_{\mathbb R}^{n-1};\mathbb C\bigr)^{\mathcal O_n(e_1)}
\]
satisfying $f(e_1)=1$. In particular,
\[
\dim_{\mathbb C}
\mathcal H_d\bigl(\mathbb S_{\mathbb R}^{n-1};\mathbb C\bigr)^{\mathcal O_n(e_1)}
=1.
\]
\end{lemma}

Combining this result  with Lemma~\ref{lem:strong-accessibility-stabilizer}
and Remark~\ref{nox0}, we obtain the following.

\begin{corollary}
\label{cor:real-harmonics-strong-accessibility}
The space
$\mathcal H_d(\mathbb S_{\mathbb R}^{n-1};\mathbb C)$ is strongly accessible.
\end{corollary}

We now identify the Peter--Weyl decomposition of
\(L_2(\mathbb S_{\mathbb R}^{n-1},\sigma_n;\mathbb C)\). For each
\(d\geq 0\), let
\[
    \pi_d:\mathcal O_n
    \longrightarrow
    \mathcal U\bigl(
        \mathcal H_d(\mathbb S_{\mathbb R}^{n-1};\mathbb C)
    \bigr)
\]
be the restriction of the natural \(\mathcal O_n\)-action, that is,
\[
    \bigl(\pi_d(U)f\bigr)(x)
    =
    f(U^{-1}x).
\]
By Corollary~\ref{cor:real-harmonics-strong-accessibility} and
Theorem~\ref{thm:strong-accessibility}, there exists
\(\rho_d\in\widehat{\mathcal O_n}\) such that
\[
    \mathcal H_d(\mathbb S_{\mathbb R}^{n-1};\mathbb C)
    =
    \mathcal E_{\rho_d}(\mathbb S_{\mathbb R}^{n-1})
\]
and
\[
    m_{\rho_d}(\mathbb S_{\mathbb R}^{n-1})=1.
\]
Thus this isotypic component consists of a single irreducible copy of
\(\rho_d\). Since the induced action on it is precisely \(\pi_d\), we
have
\[
    \pi_d\simeq\rho_d.
\]
In particular, \(\pi_d\) is irreducible, and, choosing \(\pi_d\) as the
representative of the corresponding equivalence class, we may write
\[
    \mathcal E_{\pi_d}(\mathbb S_{\mathbb R}^{n-1})
    =
    \mathcal H_d(\mathbb S_{\mathbb R}^{n-1};\mathbb C),
    \qquad
    m_{\pi_d}(\mathbb S_{\mathbb R}^{n-1})=1.
\]
Since the harmonic spaces are non-zero and pairwise orthogonal, their
identification with full isotypic components implies that the
representations \(\pi_d\), \(d\geq0\), are pairwise inequivalent.
Moreover, the decomposition
\eqref{eq:L2-harmonic-decomposition} shows that no other irreducible
representation occurs. Hence the Peter--Weyl decomposition of the
natural \(\mathcal O_n\)-action is
\[
L_2\bigl(\mathbb S_{\mathbb R}^{n-1},\sigma_n;\mathbb C\bigr)
=
\widehat{\bigoplus}_{d\geq0}
\mathcal E_{\pi_d}\bigl(\mathbb S_{\mathbb R}^{n-1}\bigr)
=
\widehat{\bigoplus}_{d\geq0}
\mathcal H_d\bigl(\mathbb S_{\mathbb R}^{n-1};\mathbb C\bigr).
\]
The action is therefore multiplicity-free, and we can characterize all
finite-dimensional invariant and accessible complex-valued subspaces on
the real sphere.

\begin{theorem}
\label{thm:real-sphere-accessibility}
Let
$S\subset C(\mathbb S_{\mathbb R}^{n-1};\mathbb C)$ be finite-dimensional.
Then the following assertions are equivalent.
\begin{enumerate}
\item $S$ is $\mathcal O_n$-invariant.
\item $S$ is $\mathcal O_n$-accessible.
\item There exists a finite set $F\subset\mathbb N_0$ such that
\(
S
=
\bigoplus_{d\in F}
\mathcal H_d\bigl(\mathbb S_{\mathbb R}^{n-1};\mathbb C\bigr).
\)
\end{enumerate}
\end{theorem}

\begin{proof}
The equivalence of the first two assertions follows from
Corollary~\ref{cor:gelfand-pairs}, since the action of
\(\mathcal O_n\) on \(\mathbb S_{\mathbb R}^{n-1}\) is multiplicity-free.
The equivalence with the third assertion follows from
Theorem~\ref{thm:accessible-subspaces} together with the above
identification of the isotypic components.
\end{proof}

We next determine the reproducing kernels and projection constants of these
spaces. Set
\[
N_{n,d}
:=
\dim_{\mathbb C}
\mathcal H_d\bigl(\mathbb S_{\mathbb R}^{n-1};\mathbb C\bigr)
=
\dim_{\mathbb R}
\mathcal H_d\bigl(\mathbb S_{\mathbb R}^{n-1};\mathbb R\bigr).
\]
Explicitly,
\[
N_{n,d}
=
\begin{cases}
1, & d=0,\\[4pt]
\dfrac{(n+2d-2)(n+d-3)!}{d!(n-2)!}, & d\geq1
\end{cases}
\]
(see, for example, \cite[Equation (18)]{defant2026minimal}).
The homogeneous machine, Theorem~\ref{thm:homogeneous-machine}, applied to
\[
\mathcal E_{\pi_d}\bigl(\mathbb S_{\mathbb R}^{n-1}\bigr)
=
\mathcal H_d\bigl(\mathbb S_{\mathbb R}^{n-1};\mathbb C\bigr),
\]
gives
\[
\mathbf k_{\mathcal H_d(\mathbb S_{\mathbb R}^{n-1};\mathbb C)}(e_1,\xi)
=
d_{\pi_d}\,
\chi_{\pi_d}^{\mathbb S_{\mathbb R}^{n-1}}(\xi).
\]
Since $m_{\pi_d}(\mathbb S_{\mathbb R}^{n-1})=1$, the spherical character
$\chi_{\pi_d}^{\mathbb S_{\mathbb R}^{n-1}}$ is
$\mathcal O_n(e_1)$-invariant and normalized at $e_1$. Hence
Lemma~\ref{lem:diamond-function} yields
\[
\chi_{\pi_d}^{\mathbb S_{\mathbb R}^{n-1}}(\xi)
=
L^\diamond_{n,d}\bigl(\langle\xi,e_1\rangle\bigr).
\]
Since
\[
d_{\pi_d}
=
\dim_{\mathbb C}
\mathcal H_d\bigl(\mathbb S_{\mathbb R}^{n-1};\mathbb C\bigr)
=N_{n,d},
\]
we obtain
\begin{equation}
\label{eq:harmonic-kernel}
\mathbf k_{\mathcal H_d(\mathbb S_{\mathbb R}^{n-1};\mathbb C)}(e_1,\xi)
=
N_{n,d}L^\diamond_{n,d}\bigl(\langle\xi,e_1\rangle\bigr),
\qquad
\xi\in\mathbb S_{\mathbb R}^{n-1}.
\end{equation}
Notice that this reproducing-kernel slice is real-valued.

\begin{theorem}
\label{prop:real-invariant-kernel}
Let
$S\subset C(\mathbb S_{\mathbb R}^{n-1};\mathbb C)$ be an
$\mathcal O_n$-invariant finite-dimensional subspace. Then there exists a
finite set $F\subset\mathbb N_0$ such that
\[
S
=
\bigoplus_{d\in F}
\mathcal H_d\bigl(\mathbb S_{\mathbb R}^{n-1};\mathbb C\bigr).
\]
Moreover, for each $\xi\in\mathbb S_{\mathbb R}^{n-1}$,
\[
\mathbf k_S(e_1,\xi)
=
\sum_{d\in F}
N_{n,d}L^\diamond_{n,d}\bigl(\langle\xi,e_1\rangle\bigr).
\]
Consequently,
\[
\boldsymbol\lambda(S)
=
\int_{\mathbb S_{\mathbb R}^{n-1}}
\bigg|
\sum_{d\in F}
N_{n,d}L^\diamond_{n,d}\bigl(\langle\xi,e_1\rangle\bigr)
\bigg|
\,\mathrm{d}\sigma_n(\xi).
\]
\end{theorem}

\begin{proof}
The description of $S$ follows from
Theorem~\ref{thm:real-sphere-accessibility}. Since the summands are
orthogonal, the reproducing kernel of $S$ is the sum of the reproducing
kernels of its harmonic components. The kernel formula therefore follows
from \eqref{eq:harmonic-kernel}. Finally, $S$ is $\mathcal O_n$-accessible,
and Theorem~\ref{thm:homogeneous-machine} gives the asserted formula for its complex projection
constant.
\end{proof}

In particular, by \eqref{eq:P-d-harmonic-decomposition},
\[
\mathbf k_{\mathcal P_d(\mathbb S_{\mathbb R}^{n-1};\mathbb C)}(e_1,\xi)
=
\sum_{j=0}^{\lfloor d/2\rfloor}
N_{n,d-2j}
L^\diamond_{n,d-2j}\bigl(\langle\xi,e_1\rangle\bigr),
\]
and hence
\begin{equation}
\label{eq:complex-real-sphere-projection-constant}
\boldsymbol \lambda\bigl(\mathcal P_d(\mathbb S_{\mathbb R}^{n-1};\mathbb C)\bigr)
=
\int_{\mathbb S_{\mathbb R}^{n-1}}
\bigg|
\sum_{j=0}^{\lfloor d/2\rfloor}
N_{n,d-2j}
L^\diamond_{n,d-2j}\bigl(\langle\xi,e_1\rangle\bigr)
\bigg|
\,\mathrm{d}\sigma_n(\xi).
\end{equation}
We now pass from the complex-valued polynomial space to its real form. Set
\[
S_{\mathbb R}
:=
\mathcal P_d\bigl(\mathbb S_{\mathbb R}^{n-1};\mathbb R\bigr)
\subset
C\bigl(\mathbb S_{\mathbb R}^{n-1};\mathbb R\bigr)
\]
and
\[
S_{\mathbb C}
:=
\mathcal P_d\bigl(\mathbb S_{\mathbb R}^{n-1};\mathbb C\bigr)
=
S_{\mathbb R}\oplus iS_{\mathbb R}.
\]

\begin{proposition}
\label{prop:real-complex-projection-constants}
The space
$\mathcal P_d(\mathbb S_{\mathbb R}^{n-1};\mathbb R)$ is an
$\mathcal O_n$-accessible subspace of
$C(\mathbb S_{\mathbb R}^{n-1};\mathbb R)$, and
\[
\boldsymbol \lambda\bigl(\mathcal P_d(\mathbb S_{\mathbb R}^{n-1};\mathbb R)\bigr)
=
\boldsymbol \lambda\bigl(\mathcal P_d(\mathbb S_{\mathbb R}^{n-1};\mathbb C)\bigr).
\]
Both projection constants are given by the integral in
\eqref{eq:complex-real-sphere-projection-constant}.
\end{proposition}

\begin{proof}
The complexification $S_{\mathbb C}$ is $\mathcal O_n$-accessible by
Theorem~\ref{thm:real-sphere-accessibility}. Hence Lemma~\ref{lem: acR-acC} implies that
$S_{\mathbb R}$ is $\mathcal O_n$-accessible in the real-valued function
space.
The real and complex $L_2$-orthogonal projections onto $S_{\mathbb R}$ and
$S_{\mathbb C}$, respectively, are represented by the same real-valued
reproducing kernel
\[
\mathbf k_{\mathcal P_d(\mathbb S_{\mathbb R}^{n-1};\mathbb R)}(x,y)
=
\mathbf k_{\mathcal P_d(\mathbb S_{\mathbb R}^{n-1};\mathbb C)}(x,y).
\]
Indeed, the complex orthogonal projection is the complexification of the real
orthogonal projection. Therefore both operator norms are equal to
\[
\sup_{x\in\mathbb S_{\mathbb R}^{n-1}}
\int_{\mathbb S_{\mathbb R}^{n-1}}
\bigl|
\mathbf k_{\mathcal P_d(\mathbb S_{\mathbb R}^{n-1};\mathbb R)}(x,y)
\bigr|
\,\mathrm{d}\sigma_n(y)
=
\sup_{x\in\mathbb S_{\mathbb R}^{n-1}}
\int_{\mathbb S_{\mathbb R}^{n-1}}
\bigl|
\mathbf k_{\mathcal P_d(\mathbb S_{\mathbb R}^{n-1};\mathbb C)}(x,y)
\bigr|
\,\mathrm{d}\sigma_n(y).
\]
By Theorem~\ref{thm:accessible-projection-constant}, both orthogonal projections are minimal in their
respective scalar categories. Since they have the same operator norm,
the two projection constants coincide.
Formula~\eqref{eq:complex-real-sphere-projection-constant} then gives their
common value.
\end{proof}

The finite sum in \eqref{eq:complex-real-sphere-projection-constant} can be
rewritten in terms of Jacobi polynomials. Once the kernel formula is known,
this is a purely algebraic, although nontrivial, calculation. The resulting
formula is the following. For $d\geq1$, see
\cite[Proposition~4.5 and Theorem~4.6]{defant2026minimal}; the case $d=0$ is immediate.

\begin{proposition}
\label{prop:jacobi-projection-formula}
For each $d\in\mathbb N_0$,
\[
\boldsymbol \lambda\bigl(\mathcal P_d(\mathbb S_{\mathbb R}^{n-1};\mathbb C)\bigr)
=
\boldsymbol \lambda\bigl(\mathcal P_d(\mathbb S_{\mathbb R}^{n-1};\mathbb R)\bigr),
\]
and their common value is
\[
\frac{
\Gamma\bigl(\frac n2\bigr)\Gamma(d+n)}
{2\sqrt{\pi}\,\Gamma(n-1)
\Gamma\bigl(d+\frac{n+1}{2}\bigr)}
\int_{-1}^{1}
\bigg|
P_d^{\bigl(\frac{n-1}{2},\frac{n-1}{2}\bigr)}(t)
\bigg|
(1-t^2)^{\frac{n-3}{2}}
\,\mathrm{d} t.
\]
Here \(P_d^{(\alpha,\beta)}\) is the Jacobi polynomial normalized by
\(P_d^{(\alpha,\beta)}(1)=\binom{d+\alpha}{d}\), and
\(\alpha=\beta=(n-1)/2\).
\end{proposition}

We finally relate the polynomial space on the sphere to the original space of real-valued
homogeneous polynomials on the real Hilbert space. The restriction map
\[
\mathcal P_d\bigl(\ell_2^n(\mathbb R)\bigr)
=
\mathcal P_d\bigl(\ell_2^n(\mathbb R);\mathbb R\bigr)
\longrightarrow
\mathcal P_d\bigl(\mathbb S_{\mathbb R}^{n-1};\mathbb R\bigr),
\qquad
P\longmapsto P\big|_{\mathbb S_{\mathbb R}^{n-1}},
\]
is a real-linear isometric isomorphism. Indeed, by $d$-homogeneity,
\[
\|P\|
=
\sup_{\|x\|_2\leq1}|P(x)|
=
\sup_{\|x\|_2=1}|P(x)|.
\]
Moreover, if $P$ vanishes on the sphere, then homogeneity implies that it
vanishes on all of $\mathbb R^n$. Consequently, Proposition~\ref{prop:jacobi-projection-formula}
gives
\begin{equation}
\label{eq:real-polynomial-projection-constant}
\boldsymbol \lambda\bigl(\mathcal P_d(\ell_2^n(\mathbb R))\bigr)
=
\boldsymbol \lambda\bigl(\mathcal P_d(\mathbb S_{\mathbb R}^{n-1};\mathbb R)\bigr)
=
\boldsymbol \lambda\bigl(\mathcal P_d(\mathbb S_{\mathbb R}^{n-1};\mathbb C)\bigr).
\end{equation}
In \cite[Proposition~4.5 and Corollary~4.7]{defant2026minimal}, the asymptotic behavior
of these projection constants was studied in the regime where the dimension
$n$ is fixed and the degree $d$ tends to infinity. More precisely, for
$n=2$,
\[
\lim_{d\to\infty}
\frac{
\boldsymbol \lambda\bigl(\mathcal P_d(\ell_2^2(\mathbb R))\bigr)}
{\log d}
=
\frac{4}{\pi^2},
\]
whereas, for $n>2$,
\[
\lim_{d\to\infty}
\frac{
\boldsymbol \lambda\bigl(\mathcal P_d(\ell_2^n(\mathbb R))\bigr)}
{d^{\frac{n-2}{2}}}
=
\frac{2^{n+1}}{\pi^2(n-2)}
\frac{
\Gamma\bigl(\frac n4+\frac12\bigr)^2}
{\Gamma(n-1)}.
\]
By Proposition~\ref{prop:real-complex-projection-constants}, the same
formulas hold for the corresponding complexified spaces.
Here we take the opposite point of view: we fix the degree $d$ and let the
dimension $n$ tend to infinity.

\begin{theorem}
\label{thm:high-dimensional-real-polynomials}
For each $d\in\mathbb N_0$,
\[
\lim_{n\to\infty}
\frac{
\boldsymbol \lambda\bigl(\mathcal P_d(\ell_2^n(\mathbb R))\bigr)}
{\sqrt{\binom{n+d-1}{d}}}
=
\frac{\mathbb E|\mathrm{He}_d(Z)|}{\sqrt{d!}},
\]
where $\mathrm{He}_d$ denotes the $d$-th probabilists' Hermite polynomial and
$Z\sim N(0,1)$. Equivalently,
\[
\lim_{n\to\infty}
\frac{
\boldsymbol \lambda\bigl(\mathcal P_d(\ell_2^n(\mathbb R))\bigr)}
{\sqrt{n^d}}
=
\frac{\mathbb E|\mathrm{He}_d(Z)|}{d!}.
\]
By \eqref{eq:real-polynomial-projection-constant}, the same limits hold with
$\boldsymbol \lambda(\mathcal P_d(\mathbb S_{\mathbb R}^{n-1};\mathbb C))$ or
$\boldsymbol \lambda(\mathcal P_d(\mathbb S_{\mathbb R}^{n-1};\mathbb R))$ in the
numerator.
\end{theorem}

The Gaussian constant in
Theorem~\ref{thm:high-dimensional-real-polynomials} also admits a precise
asymptotic expansion as $d\to\infty$, namely
\[
\frac{\mathbb E|\mathrm{He}_d(Z)|}{\sqrt{d!}}
=
\frac{2^{7/4}}{\pi^{5/4}}\frac{1}{d^{1/4}}
\bigg(
1+O\bigg(\frac{1}{d^2}\bigg)
\bigg).
\]
This follows from the asymptotic formula for the $L_1$-norm of Hermite
polynomials due to Larsson--Cohn
\cite[Remarks~2.6 and~3.2]{larsson2002p}.

For the proof of Theorem~\ref{thm:high-dimensional-real-polynomials}, we use
the following two elementary facts.

\begin{lemma}
\label{lem:ultraspherical-gaussian-limit}
For each $n\geq2$, set
\[
c_n
:=
\int_{-1}^{1}(1-t^2)^{\frac{n-3}{2}}\,\mathrm{d} t
=
\frac{
\sqrt{\pi}\,\Gamma\bigl(\frac{n-1}{2}\bigr)}
{\Gamma\bigl(\frac n2\bigr)}
\]
and
\[
f_n(t)
:=
\begin{cases}
c_n^{-1}(1-t^2)^{\frac{n-3}{2}}, & |t|<1,\\
0, & |t|\geq1.
\end{cases}
\]
Then, for each $u\in\mathbb R$,
\[
g_n(u)
:=
\frac{1}{\sqrt n}
f_n\bigg(\frac{u}{\sqrt n}\bigg)
\longrightarrow
\frac{1}{\sqrt{2\pi}}e^{-u^2/2}.
\]
Moreover, there exist constants $C>0$ and $n_0\in\mathbb N$ such that
\[
g_n(u)
\leq
Ce^{-u^2/4},
\qquad
u\in\mathbb R,\quad n\geq n_0.
\]
\end{lemma}

\begin{proof}
For \(|u|<\sqrt n\), the definition gives
\[
g_n(u)
=
\frac{1}{\sqrt n\,c_n}
\bigg(1-\frac{u^2}{n}\bigg)^{\frac{n-3}{2}},
\]
while \(g_n(u)=0\) outside its support. Every fixed \(u\in\mathbb R\)
satisfies \(|u|<\sqrt n\) for all sufficiently large \(n\), and
\[
\bigg(1-\frac{u^2}{n}\bigg)^{\frac{n-3}{2}}
=
\exp\bigg(
\frac{n-3}{2}
\log\bigg(1-\frac{u^2}{n}\bigg)
\bigg)
\longrightarrow
e^{-u^2/2}.
\]
Furthermore, the standard asymptotic formula for quotients of Gamma
functions gives
\[
\frac{1}{\sqrt n\,c_n}
=
\frac{\Gamma\bigl(\frac n2\bigr)}
{\sqrt{n\pi}\,\Gamma\bigl(\frac{n-1}{2}\bigr)}
\longrightarrow
\frac{1}{\sqrt{2\pi}}.
\]
This proves the pointwise convergence.
For the domination, the sequence
\(\bigl(1/(\sqrt n\,c_n)\bigr)_{n\geq2}\) is bounded. If
\(|u|<\sqrt n\) and \(n\geq6\), the estimate
\(\log(1-x)\leq-x\), valid for \(0\leq x<1\), gives
\[
\bigg(1-\frac{u^2}{n}\bigg)^{\frac{n-3}{2}}
\leq
\exp\bigg(-\frac{n-3}{2n}u^2\bigg)
\leq
e^{-u^2/4}.
\]
Outside the support of \(g_n\), the estimate is trivial. Hence there is
a constant \(C>0\) such that, for every \(n\geq6\) and \(u\in\mathbb R\),
\[
        g_n(u)\leq Ce^{-u^2/4},
\]
which proves the claim.
\end{proof}

\begin{lemma}
\label{lem:jacobi-hermite-limit}
Fix $d\in\mathbb N_0$, and set $\alpha=(n-1)/2$. Then, for each
$u\in\mathbb R$,
\[
n^{-d/2}
P_d^{(\alpha,\alpha)}
\bigg(\frac{u}{\sqrt n}\bigg)
\longrightarrow
\frac{1}{2^d d!}\mathrm{He}_d(u).
\]
The convergence is uniform on compact subsets of $\mathbb R$. Moreover, there
exists a constant $C_d>0$ such that, for each $n\geq2$ and each
$u\in\mathbb R$,
\[
\bigg|
n^{-d/2}
P_d^{(\alpha,\alpha)}
\bigg(\frac{u}{\sqrt n}\bigg)
\bigg|
\leq
C_d(1+|u|)^d.
\]
\end{lemma}

\begin{proof}
We use the relation between Jacobi and Gegenbauer polynomials,
\[
P_d^{(\alpha,\alpha)}(x)
=
\frac{(\alpha+1)_d}{(2\alpha+1)_d}
C_d^{\alpha+\frac12}(x),
\]
where
\[
(a)_d:=a(a+1)\cdots(a+d-1),
\qquad
(a)_0:=1,
\]
denotes the Pochhammer symbol. Since $\alpha=(n-1)/2$, this becomes
\[
P_d^{(\alpha,\alpha)}(x)
=
\frac{\bigl(\frac{n+1}{2}\bigr)_d}{(n)_d}
C_d^{n/2}(x).
\]
For fixed $d$,
\[
\frac{\bigl(\frac{n+1}{2}\bigr)_d}{(n)_d}
=
\prod_{k=0}^{d-1}
\frac{\frac{n+1}{2}+k}{n+k}
\longrightarrow
2^{-d}.
\]
We now use the finite expansion
\[
C_d^\lambda(x)
=
\sum_{m=0}^{\lfloor d/2\rfloor}
(-1)^m
\frac{(\lambda)_{d-m}}{m!(d-2m)!}
(2x)^{d-2m}.
\]
Taking $\lambda=n/2$ and $x=u/\sqrt n$, we obtain
\[
\begin{split}
n^{-d/2}
C_d^{n/2}\bigg(\frac{u}{\sqrt n}\bigg)
=
\sum_{m=0}^{\lfloor d/2\rfloor}
(-1)^m
\frac{\bigl(\frac n2\bigr)_{d-m}}{m!(d-2m)!}\,
2^{d-2m}n^{-d/2}
\bigg(\frac{u}{\sqrt n}\bigg)^{d-2m}.
\end{split}
\]
For each fixed $m$,
\(
\bigl(\tfrac n2\bigr)_{d-m}
\sim
\bigg(\frac n2\bigg)^{d-m}.
\)
Therefore the $m$-th summand converges to
\[
(-1)^m
\frac{u^{d-2m}}{2^m m!(d-2m)!}.
\]
Hence
\[
n^{-d/2}
C_d^{n/2}\bigg(\frac{u}{\sqrt n}\bigg)
\longrightarrow
\sum_{m=0}^{\lfloor d/2\rfloor}
\frac{(-1)^m u^{d-2m}}
{2^m m!(d-2m)!}.
\]
By the standard formula for the probabilists' Hermite polynomial,
\[
\mathrm{He}_d(u)
=
d!
\sum_{m=0}^{\lfloor d/2\rfloor}
\frac{(-1)^m u^{d-2m}}
{2^m m!(d-2m)!}.
\]
Thus
\[
n^{-d/2}
C_d^{n/2}\bigg(\frac{u}{\sqrt n}\bigg)
\longrightarrow
\frac{1}{d!}\mathrm{He}_d(u).
\]
Multiplying by the Pochhammer factor, which converges to $2^{-d}$, gives
\[
n^{-d/2}
P_d^{(\alpha,\alpha)}
\bigg(\frac{u}{\sqrt n}\bigg)
\longrightarrow
\frac{1}{2^d d!}\mathrm{He}_d(u).
\]
The preceding finite expansions show coefficientwise convergence.
Since the degree \(d\) is fixed, this implies uniform convergence on
compact subsets of \(\mathbb R\).

It remains to establish the required growth bound. Using the same
expansion and the estimate
\[
        \biggl(\frac n2\biggr)_r
        \leq
        A_d n^r,
        \qquad 0\leq r\leq d,
\]
each summand of
\[
        n^{-d/2}
        C_d^{n/2}\biggl(\frac{u}{\sqrt n}\biggr)
\]
is bounded in absolute value by a constant depending only on \(d\),
times \(|u|^{d-2m}\). Since only finitely many values of \(m\) occur, we
obtain
\[
        \biggl|
        n^{-d/2}
        C_d^{n/2}\biggl(\frac{u}{\sqrt n}\biggr)
        \biggr|
        \leq
        B_d(1+|u|)^d.
\]
Finally,
\[
        0<
        \frac{\bigl(\frac{n+1}{2}\bigr)_d}{(n)_d}
        \leq 1,
        \qquad n\geq2,
\]
and hence the corresponding Jacobi polynomials satisfy the same type of
polynomial growth bound.
\end{proof}

\smallskip

\begin{proof}[Proof of Theorem~\ref{thm:high-dimensional-real-polynomials}]
The case $d=0$ is immediate. Assume from now on that $d\geq1$. By
\eqref{eq:real-polynomial-projection-constant} and
Proposition~\ref{prop:jacobi-projection-formula}, with
$\alpha=(n-1)/2$, we have
\[
\begin{split}
\boldsymbol\lambda\bigl(\mathcal P_d(\ell_2^n(\mathbb R))\bigr)
=&
\frac{
\Gamma\bigl(\frac n2\bigr)\Gamma(n+d)}
{2\sqrt{\pi}\,\Gamma(n-1)
\Gamma\bigl(d+\frac{n+1}{2}\bigr)}\,
\int_{-1}^{1}
\bigl|P_d^{(\alpha,\alpha)}(t)\bigr|
(1-t^2)^{\frac{n-3}{2}}
\,\mathrm{d} t.
\end{split}
\]
We normalize the ultraspherical weight as in
Lemma~\ref{lem:ultraspherical-gaussian-limit}. Thus
\[
f_n(t)
=
\frac{1}{c_n}
(1-t^2)^{\frac{n-3}{2}}
\mathbf 1_{[-1,1]}(t),
\qquad
c_n
=
\frac{
\sqrt{\pi}\,\Gamma\bigl(\frac{n-1}{2}\bigr)}
{\Gamma\bigl(\frac n2\bigr)}.
\]
Then
\[
\boldsymbol\lambda\bigl(\mathcal P_d(\ell_2^n(\mathbb R))\bigr)
=
A_n
\int_{-1}^{1}
\bigl|P_d^{(\alpha,\alpha)}(t)\bigr|
f_n(t)\,\mathrm{d} t,
\]
where
\[
A_n
:=
\frac12
\frac{
\Gamma\bigl(\frac{n-1}{2}\bigr)\Gamma(n+d)}
{\Gamma(n-1)
\Gamma\bigl(d+\frac{n+1}{2}\bigr)}.
\]
The Gamma quotient asymptotics give
\[
\frac{\Gamma(n+d)}{\Gamma(n-1)}
\sim n^{d+1} \quad \text{and} \quad
\frac{
\Gamma\bigl(\frac{n-1}{2}\bigr)}
{\Gamma\bigl(d+\frac{n+1}{2}\bigr)}
\sim
\bigg(\frac n2\bigg)^{-d-1}.
\]
Consequently,
\(
A_n\longrightarrow2^d.
\)
We now study the integral. With $g_n$ as in
Lemma~\ref{lem:ultraspherical-gaussian-limit}, the change of variables
$t=u/\sqrt n$ gives
\[
\begin{split}
\frac{1}{n^{d/2}}
\int_{-1}^{1}
\bigl|P_d^{(\alpha,\alpha)}(t)\bigr|
f_n(t)\,\mathrm{d} t
=
\int_{\mathbb R}
\bigg|
n^{-d/2}
P_d^{(\alpha,\alpha)}
\bigg(\frac{u}{\sqrt n}\bigg)
\bigg|
g_n(u)\,\mathrm{d} u.
\end{split}
\]
By Lemmas~\ref{lem:jacobi-hermite-limit} and
\ref{lem:ultraspherical-gaussian-limit}, the integrand converges pointwise to
\[
\frac{1}{2^d d!}
|\mathrm{He}_d(u)|
\frac{1}{\sqrt{2\pi}}e^{-u^2/2}.
\]
Moreover, for \(n\geq n_0\), the same two lemmas give
\[
\bigg|
n^{-d/2}
P_d^{(\alpha,\alpha)}
\bigg(\frac{u}{\sqrt n}\bigg)
\bigg|
g_n(u)
\leq
C_d(1+|u|)^d e^{-u^2/4},
\]
and the function on the right-hand side is integrable on $\mathbb R$. Hence
the dominated convergence theorem yields
\[
\begin{split}
\lim_{n\to\infty}
\frac{1}{n^{d/2}}
\int_{-1}^{1}
\bigl|P_d^{(\alpha,\alpha)}(t)\bigr|
f_n(t)\,\mathrm{d} t
=
\frac{1}{2^d d!}
\frac{1}{\sqrt{2\pi}}
\int_{\mathbb R}
|\mathrm{He}_d(u)|e^{-u^2/2}\,\mathrm{d} u.
\end{split}
\]
Equivalently, if $Z\sim N(0,1)$, then
\[
\lim_{n\to\infty}
\frac{1}{n^{d/2}}
\int_{-1}^{1}
\bigl|P_d^{(\alpha,\alpha)}(t)\bigr|
f_n(t)\,\mathrm{d} t
=
\frac{\mathbb E|\mathrm{He}_d(Z)|}{2^d d!}.
\]
Combining this with $A_n\to2^d$, we obtain
\[
\lim_{n\to\infty}
\frac{
\boldsymbol\lambda\bigl(\mathcal P_d(\ell_2^n(\mathbb R))\bigr)}
{n^{d/2}}
=
\frac{\mathbb E|\mathrm{He}_d(Z)|}{d!}.
\]
This proves the second assertion of Theorem~\ref{thm:high-dimensional-real-polynomials}.
Finally, since \(d\) is fixed and \( \binom{n+d-1}{d} \sim \frac{n^d}{d!}  \), the first assertion follows as well.
\end{proof}

\subsubsection{The complex Euclidean sphere}
\label{sec:complex-sphere}
Unlike the preceding subsection on the real Euclidean sphere, the
present discussion takes place entirely over the complex scalar field.
Throughout this subsection, we assume that \(n\geq 2\). Let
\[
    \mathbb S_{\mathbb C}^{n-1}
    :=
    \{z\in\mathbb C^n:\|z\|_2=1\}
\]
be the complex Euclidean sphere, endowed with its normalized
rotation-invariant measure \(\sigma_n\). The unitary group
\(\mathcal U_n\) acts transitively on
\(\mathbb S_{\mathbb C}^{n-1}\) by
\[
    U\hbullet z:=Uz,
    \qquad
    U\in\mathcal U_n,\quad
    z\in\mathbb S_{\mathbb C}^{n-1}.
\]
We choose the base point \(e_1=(1,0,\ldots,0)\). Its stabilizer is
\[
    \mathcal U_n(e_1)
    :=
    \{U\in\mathcal U_n:Ue_1=e_1\},
\]
which is naturally isomorphic to \(\mathcal U_{n-1}\), acting
unitarily on \(e_1^\perp\). Thus
\(\mathbb S_{\mathbb C}^{n-1}\) is a homogeneous
\(\mathcal U_n\)-space.
For \(p,q\in\mathbb N_0\), let
\[
    \mathcal H_{p,q}(\mathbb S_{\mathbb C}^{n-1})
\]
denote the space of restrictions to
\(\mathbb S_{\mathbb C}^{n-1}\) of harmonic polynomials on
\(\mathbb C^n\) that are bihomogeneous of degree \(p\) in \(z\) and
degree \(q\) in \(\overline z\). Equivalently, these are the restrictions of polynomials \(P\) satisfying
\[
    P(\lambda z)
    =
    \lambda^p\overline{\lambda}^{\,q}P(z),
    \qquad
    \lambda\in\mathbb C,
\]
and annihilated by
\(
    \Delta_{\mathbb C}
    :=
    \sum_{j=1}^n
    \partial_{z_j}\partial_{\overline z_j},
\)
which is one quarter of the real Laplacian on \(\mathbb R^{2n}\).
Each space
\(\mathcal H_{p,q}(\mathbb S_{\mathbb C}^{n-1})\) is
finite-dimensional and \(\mathcal U_n\)-invariant, and these spaces
are pairwise orthogonal. Their Hilbertian direct sum gives the
classical complex spherical harmonic decomposition
\begin{equation}
\label{eq:L2-complex-harmonic-decomposition}
    L_2(\mathbb S_{\mathbb C}^{n-1},\sigma_n)
    =
    \widehat{\bigoplus}_{p,q\geq 0}
    \mathcal H_{p,q}(\mathbb S_{\mathbb C}^{n-1});
\end{equation}
see, for example, \cite{defant2026ryll}.

Our next goal is to identify the
\(\mathcal U_n(e_1)\)-fixed vectors in each harmonic block. For
\(p,q\in\mathbb N_0\), define
\[
    L^\diamond_{n,p,q}(w)
    :=
    \sum_{j=0}^{\min\{p,q\}}
    \frac{(-1)^j p!q!(n-2)!}
         {j!(p-j)!(q-j)!(j+n-2)!}\,
    w^{p-j}\overline w^{\,q-j}(1-|w|^2)^j,
    \qquad |w|\leq 1.
\]
In particular,
\[
    L^\diamond_{n,p,q}(1)=1.
\]
The following standard counterpart of
Lemma~\ref{lem:diamond-function} identifies the corresponding
normalized zonal harmonic. In our context, see, for example, \cite[Corollary~4.2]{defant2026ryll}.

\begin{lemma}
\label{lem:complex-diamond-function}
The function
\[
L^\diamond_{n,p,q}\bigl(\langle\,\cdot\,,e_1\rangle\bigr)
\in
C\bigl(\mathbb S_{\mathbb C}^{n-1}\bigr)
\]
is the unique element
\[
    f\in
    \mathcal H_{p,q}(\mathbb S_{\mathbb C}^{n-1})
    ^{\mathcal U_n(e_1)}
\]
satisfying \(f(e_1)=1\). In particular,
\[
    \dim
    \mathcal H_{p,q}(\mathbb S_{\mathbb C}^{n-1})
    ^{\mathcal U_n(e_1)}
    =
    1.
\]
\end{lemma}

Combining this lemma with
Lemma~\ref{lem:strong-accessibility-stabilizer}, we obtain the complex
counterpart of
Corollary~\ref{cor:real-harmonics-strong-accessibility}.

\begin{corollary}
\label{cor:complex-harmonics-strong-accessibility}
For every \(p,q\in\mathbb N_0\), the space
\(
    \mathcal H_{p,q}(\mathbb S_{\mathbb C}^{n-1})
\)
is strongly accessible.
\end{corollary}

We now identify the Peter--Weyl decomposition associated with the
natural action of \(\mathcal U_n\). For \(p,q\in\mathbb N_0\), let
\[
    \pi_{p,q}:\mathcal U_n
    \longrightarrow
    \mathcal U\bigl(
        \mathcal H_{p,q}(\mathbb S_{\mathbb C}^{n-1})
    \bigr)
\]
be the representation defined by
\[
    \bigl(\pi_{p,q}(U)f\bigr)(z)
    :=
    f(U^{-1}z).
\]
By
Corollary~\ref{cor:complex-harmonics-strong-accessibility} and
Theorem~\ref{thm:strong-accessibility}, there exists an irreducible
representation \(\rho_{p,q}\in\widehat{\mathcal U_n}\) such that
\[
    \mathcal H_{p,q}(\mathbb S_{\mathbb C}^{n-1})
    =
    \mathcal E_{\rho_{p,q}}
        (\mathbb S_{\mathbb C}^{n-1})
\]
and
\[
    m_{\rho_{p,q}}
        (\mathbb S_{\mathbb C}^{n-1})
    =
    1.
\]
Thus this isotypic component consists of a single irreducible copy of
\(\rho_{p,q}\). Since the induced action on it is precisely
\(\pi_{p,q}\), we have
\(
    \pi_{p,q}\simeq\rho_{p,q}.
\)
In particular, \(\pi_{p,q}\) is irreducible, and, choosing it as the
representative of the corresponding equivalence class, we may write
\[
    \mathcal E_{\pi_{p,q}}
        (\mathbb S_{\mathbb C}^{n-1})
    =
    \mathcal H_{p,q}(\mathbb S_{\mathbb C}^{n-1}),
    \qquad
    m_{\pi_{p,q}}
        (\mathbb S_{\mathbb C}^{n-1})
    =
    1.
\]
Since distinct harmonic blocks are non-zero and orthogonal, their
identification with full isotypic components implies that the
representations \(\pi_{p,q}\), \(p,q\geq0\), are pairwise inequivalent.
Equation~\eqref{eq:L2-complex-harmonic-decomposition} then shows that
no other irreducible representation occurs. Hence
\[
    L_2(\mathbb S_{\mathbb C}^{n-1},\sigma_n)
    =
    \widehat{\bigoplus}_{p,q\geq 0}
    \mathcal E_{\pi_{p,q}}
        (\mathbb S_{\mathbb C}^{n-1})
    =
    \widehat{\bigoplus}_{p,q\geq 0}
    \mathcal H_{p,q}(\mathbb S_{\mathbb C}^{n-1})
\]
is the Peter--Weyl decomposition of the natural
\(\mathcal U_n\)-action. This action is therefore multiplicity-free.
Corollary~\ref{cor:gelfand-pairs} and
Theorem~\ref{thm:accessible-subspaces} now give the following
characterization of finite-dimensional invariant and accessible
subspaces, the complex counterpart of
Theorem~\ref{thm:real-sphere-accessibility}.

\begin{theorem}
\label{thm:complex-sphere-accessibility}
Let
\(
    S\subset C(\mathbb S_{\mathbb C}^{n-1})
\)
be finite-dimensional. Then the following assertions are equivalent.
\begin{itemize}
\item[(1)] \(S\) is \(\mathcal U_n\)-invariant.
\item[(2)] \(S\) is \(\mathcal U_n\)-accessible.
\item[(3)] There exists a finite set
      \(F\subset\mathbb N_0^2\) such that
\(
    S
    =
    \bigoplus_{(p,q)\in F}
    \mathcal H_{p,q}(\mathbb S_{\mathbb C}^{n-1}).
\)
\end{itemize}
\end{theorem}

We next determine the reproducing kernels of these spaces. Set
\[
    N_{n,p,q}
    :=
    \dim\mathcal H_{p,q}(\mathbb S_{\mathbb C}^{n-1})
    =
    \frac{(n+p+q-1)(n+p-2)!(n+q-2)!}
         {p!q!(n-1)!(n-2)!}
\]
(see, for example, \cite[Equation~(5)]{defant2026ryll}).
Applying Theorem~\ref{thm:homogeneous-machine} to the full isotypic
component
\(
    \mathcal E_{\pi_{p,q}}
        (\mathbb S_{\mathbb C}^{n-1})
    =
    \mathcal H_{p,q}(\mathbb S_{\mathbb C}^{n-1}),
\)
we obtain
\[
    \mathbf k_{\mathcal H_{p,q}
        (\mathbb S_{\mathbb C}^{n-1})}(e_1,\xi)
    =
    d_{\pi_{p,q}}\,
    \chi_{\pi_{p,q}}^{\mathbb S_{\mathbb C}^{n-1}}(\xi).
\]
Since
\(
    m_{\pi_{p,q}}(\mathbb S_{\mathbb C}^{n-1})=1,
\)
the spherical character
\(\chi_{\pi_{p,q}}^{\mathbb S_{\mathbb C}^{n-1}}\) is a normalized
\(\mathcal U_n(e_1)\)-invariant element of
\(\mathcal H_{p,q}(\mathbb S_{\mathbb C}^{n-1})\). Hence
Lemma~\ref{lem:complex-diamond-function} gives
\[
    \chi_{\pi_{p,q}}^{\mathbb S_{\mathbb C}^{n-1}}(\xi)
    =
    L^\diamond_{n,p,q}(\xi_1).
\]
Moreover,
\[
    d_{\pi_{p,q}}
    =
    \dim\mathcal H_{p,q}(\mathbb S_{\mathbb C}^{n-1})
    =
    N_{n,p,q}.
\]
Consequently,
\begin{equation}
\label{eq:complex-harmonic-kernel}
    \mathbf k_{\mathcal H_{p,q}
        (\mathbb S_{\mathbb C}^{n-1})}(e_1,\xi)
    =
    N_{n,p,q}\,
    L^\diamond_{n,p,q}(\xi_1).
\end{equation}
Together with Theorem~\ref{thm:complex-sphere-accessibility}, this yields
the following complex counterpart of Theorem~\ref{prop:real-invariant-kernel}
and of the corresponding kernel formula for the real sphere.

\begin{theorem}
\label{prop:complex-sphere-kernel}
Let
\(
    S\subset C(\mathbb S_{\mathbb C}^{n-1})
\)
be a finite-dimensional \(\mathcal U_n\)-invariant subspace. Then
there exists a finite set \(F\subset\mathbb N_0^2\) such that
\[
    S
    =
    \bigoplus_{(p,q)\in F}
    \mathcal H_{p,q}(\mathbb S_{\mathbb C}^{n-1}).
\]
Moreover, for every
\(\xi\in\mathbb S_{\mathbb C}^{n-1}\),
\[
    \mathbf k_S(e_1,\xi)
    =
    \sum_{(p,q)\in F}
    N_{n,p,q}\,
    L^\diamond_{n,p,q}(\xi_1).
\]
Consequently,
\[
    \boldsymbol\lambda(S)
    =
    \int_{\mathbb S_{\mathbb C}^{n-1}}
    \left|
        \sum_{(p,q)\in F}
        N_{n,p,q}\,
        L^\diamond_{n,p,q}(\xi_1)
    \right|
    \,d\sigma_n(\xi).
\]
\end{theorem}

We now specialize this description to analytic homogeneous
polynomials. For \(d\in\mathbb N_0\), restriction to the sphere gives
an isometric identification
\[
    \mathcal P_d(\ell_2^n(\mathbb C))
    \cong
    \mathcal P_d(\mathbb S_{\mathbb C}^{n-1}),
\]
where
\[
    \mathcal P_d(\mathbb S_{\mathbb C}^{n-1})
    =
    \mathcal H_{d,0}(\mathbb S_{\mathbb C}^{n-1}).
\]
In this case,
\[
    L^\diamond_{n,d,0}(w)=w^d
\]
and
\[
    N_{n,d,0}
    =
    \dim\mathcal P_d(\ell_2^n(\mathbb C))
    =
    \binom{n+d-1}{d}.
\]
Therefore, by \eqref{eq:complex-harmonic-kernel},
\[
    \mathbf k_{\mathcal P_d
        (\mathbb S_{\mathbb C}^{n-1})}(e_1,\xi)
    =
    \binom{n+d-1}{d}\,\xi_1^d.
\]
It follows that
\begin{equation}
\label{eq:complex-polynomial-kernel-integral}
    \boldsymbol\lambda
        \bigl(\mathcal P_d(\ell_2^n(\mathbb C))\bigr)
    =
    \binom{n+d-1}{d}
    \int_{\mathbb S_{\mathbb C}^{n-1}}
    |\xi_1|^d\,d\sigma_n(\xi).
\end{equation}
For \(\xi\) distributed according to \(\sigma_n\), the random variable
\(|\xi_1|^2\) has the Beta\((1,n-1)\) density
\((n-1)(1-t)^{n-2}\) on \([0,1]\). Hence
\begin{align*}
    \int_{\mathbb S_{\mathbb C}^{n-1}}
    |\xi_1|^d\,d\sigma_n(\xi)
    =
    (n-1)\int_0^1
    t^{d/2}(1-t)^{n-2}\,dt   =
    (n-1)
    B\left(1+\frac d2,n-1\right) =
    \frac{\Gamma(n)\Gamma(1+d/2)}
         {\Gamma(n+d/2)}.
\end{align*}
Substituting this identity into
\eqref{eq:complex-polynomial-kernel-integral}, we recover the classical
Ryll--Wojtaszczyk formula
\cite{ryll1983homogeneous}:
\begin{equation}
\label{Ryll--Wojtaszczyk}
    \boldsymbol\lambda
        \bigl(\mathcal P_d(\ell_2^n(\mathbb C))\bigr)
    =
    \frac{\Gamma(n+d)\Gamma(1+d/2)}
         {\Gamma(1+d)\Gamma(n+d/2)}.
\end{equation}
We finally record the two natural asymptotic regimes. First, for fixed
\(d\in\mathbb N_0\), the standard asymptotics for quotients of Gamma
functions give the complex counterpart of
Theorem~\ref{thm:high-dimensional-real-polynomials}:
\[
    \lim_{n\to\infty}
    \frac{
        \boldsymbol\lambda
        \bigl(\mathcal P_d(\ell_2^n(\mathbb C))\bigr)
    }{
        \sqrt{\binom{n+d-1}{d}}
    }
    =
    \frac{\Gamma\left(1+\frac d2\right)}{\sqrt{d!}}.
\]
Equivalently, let \(W=(X+iY)/\sqrt2\), where \(X\) and \(Y\) are
independent \(N(0,1)\) random variables. Then
\(\mathbb E|W|^2=1\) and
\(\mathbb E|W|^d=\Gamma(1+d/2)\), so
\[
    \lim_{n\to\infty}
    \frac{
        \boldsymbol\lambda
        \bigl(\mathcal P_d(\ell_2^n(\mathbb C))\bigr)
    }{
        \sqrt{\binom{n+d-1}{d}}
    }
    =
    \frac{\mathbb E|W|^d}{\sqrt{d!}}.
\]
On the other hand, for fixed \(n\geq2\), another application of the
gamma-ratio asymptotics to \eqref{Ryll--Wojtaszczyk} gives
\[
    \lim_{d\to\infty}
    \boldsymbol\lambda
        \bigl(\mathcal P_d(\ell_2^n(\mathbb C))\bigr)
    =
    2^{n-1}.
\]

\begin{remark}
The preceding formulas concern analytic homogeneous polynomial spaces,
that is, the blocks corresponding to \((p,q)=(d,0)\). In a separate
project, the authors study the projection constants of the more
general bihomogeneous spaces
\(
    \mathcal P_{p,q}(\mathbb S_{\mathbb C}^{n-1}),
\)
including their high-dimensional asymptotic behavior. In that setting,
the limiting formulas are naturally expressed in terms of integrals
involving Laguerre polynomials.
\end{remark}

\subsection{Noncommutative harmonic analysis on the unitary group}
\label{sec:unitary-group-application}

From this point on, we write \(\ell_2^n\) for the complex Hilbert space
\(\ell_2^n(\mathbb C)\). We denote by
\(\mathcal P_d\bigl(\mathcal L(\ell_2^n)\bigr)\) the space of all complex
\(d\)-homogeneous polynomials on \(\mathcal L(\ell_2^n)\), equipped with
the supremum norm on the closed operator-norm unit ball.

We now turn to a genuinely noncommutative setting. Our main result is
Theorem~\ref{thm:unitary-projection-asymptotic}, which gives, for every
\(d\in\mathbb N\),
\begin{equation}
\label{non-com-limit}
    \lim_{n\to\infty}
    \frac{
        \boldsymbol{\lambda}
        \bigl(\mathcal P_d(\mathcal L(\ell_2^n))\bigr)
    }{
        \sqrt{\binom{n^2+d-1}{d}}
    }
    =
    \frac{\Gamma\left(1+\frac d2\right)}{\sqrt{d!}}.
\end{equation}
In this setting, the strategy from Subsection~\ref{The core of the strategy} takes a particularly
transparent form. The relevant symmetries arise from the biregular
action of \(\mathcal U_n\times\mathcal U_n\) on \(\mathcal U_n\).
Peter--Weyl theory identifies the corresponding invariant subspace in
Theorem~\ref{thm:unitary-pw-polynomials}, while
Theorem~\ref{thm:unitary-reproducing-kernel} reduces the computation of
the projection constant to the \(L_1\)-norm of an explicit kernel
slice. The remaining asymptotic analysis is then governed by the
Diaconis--Shahshahani limit theorem
\cite{diaconis1994eigenvalues} for traces of Haar-distributed unitary
matrices.

The proof of \eqref{non-com-limit}, and more generally of
Theorem~\ref{thm:unitary-projection-asymptotic}, is given in
Subsection~\ref{sec:unitary-high-dimensional-asymptotics}. Before turning to it, we
develop the necessary representation-theoretic preparation.

\subsubsection{From operators to the unitary group}
\label{subsub:operators-to-unitary-group}

The first step in our general strategy is to realize the polynomial
space isometrically on a compact homogeneous space, where the available
symmetries and the Peter--Weyl decomposition can be brought into play.
In the present setting, this is achieved by restricting polynomials
from the operator space to the unitary group.
Set
\[
    \mathcal P_d(\mathcal U_n)
    :=
    \bigl\{
        p|_{\mathcal U_n}:
        p\in\mathcal P_d(\mathcal L(\ell_2^n))
    \bigr\}.
\]
Nelson's theorem \cite{nelson1961distinguished} states that the
unitary group is a norming set for these holomorphic polynomials.
Consequently, the restriction map
\[
    \mathcal P_d\bigl(\mathcal L(\ell_2^n)\bigr)
    \longrightarrow
    \mathcal P_d(\mathcal U_n),
    \qquad
    p\longmapsto p|_{\mathcal U_n},
\]
is injective and isometric with respect to the supremum norm; see also
\cite[Theorem~1]{harris1997holomorphic}. We shall therefore identify
\[
    \mathcal P_d\bigl(\mathcal L(\ell_2^n)\bigr)
    \quad\text{isometrically with}\quad
    \mathcal P_d(\mathcal U_n).
\]
We equip \(\mathcal U_n\) with its normalized Haar measure
\(\mu_{\mathcal U_n}\) and consider the biregular action of
\(\mathcal U_n\times\mathcal U_n\) on~\(\mathcal U_n\),
\[
    (g,h)\hbullet x
    :=
    gxh^{-1},
    \qquad
    g,h,x\in\mathcal U_n.
\]
The induced action on \(C(\mathcal U_n)\) is given by
\[
    ((g,h)\hbullet f)(x)
    :=
    f(g^{-1}xh),
    \qquad
    f\in C(\mathcal U_n).
\]
Thus \(\mathcal U_n\) is a homogeneous
\(\mathcal U_n\times\mathcal U_n\)-space. Once the Peter--Weyl blocks
contained in \(\mathcal P_d(\mathcal U_n)\) have been identified,
Corollary~\ref{cor:biregular-machine} applies and yields the desired
projection-constant formula.

\subsubsection{Back to the projection constant of the trace class}
\label{sec:unitary-trace-class}
Before turning to the proof of Equation~\eqref{non-com-limit}, we first
consider the special case \(d=1\).

In \cite[Theorem~2.1]{defant2025projection}, the authors derived an
integral formula for the projection constant of the trace class
\(\mathcal S_1(n)\), which is isometrically identified with
\(\mathcal P_1\bigl(\mathcal L(\ell_2^n)\bigr)\) by the trace pairing.
More precisely, they proved that
\begin{equation}\label{trace-one}
        \boldsymbol{\lambda}\bigl(\mathcal S_1(n)\bigr)
        =
        n
        \int_{\mathcal U_n}
        \bigl|\operatorname{tr}(U)\bigr|
        \,d\mu_{\mathcal U_n}(U).
\end{equation}
Using the Weingarten calculus, they further obtained the asymptotic formula
\begin{equation}\label{trace-two}
        \lim_{n\to\infty}
        \frac{\boldsymbol\lambda\bigl(\mathcal S_1(n)\bigr)}{n}
        =
        \frac{\sqrt{\pi}}{2}.
\end{equation}
We first show how the results of the preceding sections recover the
integral representation~\eqref{trace-one}, and later also the limit formula~\eqref{trace-two}.
Let
\[
        \pi_{\mathrm{std}}:\mathcal U_n\longrightarrow
        \mathcal U(\ell_2^n),
        \qquad
        \pi_{\mathrm{std}}(U)=U,
\]
be the standard representation of \(\mathcal U_n\). It is irreducible,
and its matrix coefficients are the coordinate functions
\[
        u_{ij}(U)
        :=
        \langle Ue_j,e_i\rangle,
        \qquad 1\le i,j\le n.
\]
With the matrix-coefficient convention adopted above, these holomorphic
coordinate functions belong to the block indexed by the contragredient
representation. They form an orthogonal basis of the matrix coefficient space
\[
        \mathcal E_{\pi_{\mathrm{std}}^*}(\mathcal U_n)
        \subset L_2(\mathcal U_n,\mu_{\mathcal U_n}).
\]
They also form a basis of
\(
        \mathcal P_1(\mathcal U_n)
        :=
        \{p|_{\mathcal U_n}:p: \mathcal L(\ell_2^n)\to\mathbb C
        \text{ is linear}\}.
\)
Indeed, if
\[
        \sum_{i,j=1}^n c_{ij}u_{ij}(U)=0
        \qquad
        \text{for all }U\in\mathcal U_n,
\]
and $C=\big(\,\overline{c_{ij}}\,\big)$, then
\[
        \operatorname{tr}(C^*U)
        =
        \sum_{i,j=1}^n c_{ij}u_{ij}(U)
        =
        0,
        \qquad U\in\mathcal U_n.
\]
Since the linear span of \(\mathcal U_n\) is all of
\(\mathcal L(\ell_2^n)\), this forces \(C=0\). Hence
\[
        \mathcal P_1(\mathcal U_n)
        =
        \mathcal E_{\pi_{\mathrm{std}}^*}(\mathcal U_n).
\]
By trace duality and the restriction isometry established above, the map
\[
        \mathcal S_1(n)
        \longrightarrow
        \mathcal P_1(\mathcal U_n),
        \qquad
        A\longmapsto
        \big[U\mapsto \operatorname{tr}(AU)\big],
\]
is an onto isometry. Thus, with the uniform norm on
\(\mathcal E_{\pi_{\mathrm{std}}^*}(\mathcal U_n)\), we obtain the
isometric identification
\[
        \mathcal S_1(n)
        \cong
        \mathcal E_{\pi_{\mathrm{std}}^*}(\mathcal U_n).
\]
Since
\[
        d_{\pi_{\mathrm{std}}^*}=n
        \qquad\text{and}\qquad
        \chi_{\pi_{\mathrm{std}}^*}(U)
        =
        \overline{\operatorname{tr}(U)},
\]
Corollary~\ref{cor:biregular-machine} provides a new proof of
\eqref{trace-one}. Theorem~\ref{thm:unitary-projection-asymptotic}
substantially extends this result and, as a consequence, yields a
corresponding strengthening of the asymptotic formula
in~\eqref{trace-two}.

\subsubsection{Group actions on symmetric tensor products}
\label{sec:unitary-symmetric-tensors}
To uncover the representation-theoretic structure of
\(\mathcal P_d(\mathcal U_n)\), we now pass to the tensorial model of
homogeneous polynomials and make explicit how it interacts with the
group action. For background on symmetric tensor products and spaces
of homogeneous polynomials, we refer to
\cite{defant2019libro,dineen2012complex}.

Let \(X\) be a finite-dimensional complex vector space and let
\(d\in\mathbb N\). We denote the \(d\)-fold symmetric tensor product of
\(X\) by
\[
        \bigotimes\nolimits^{s,d}X.
\]
It is spanned by the symmetrized tensors
\[
    \otimes^{s,d}(x_1,\ldots,x_d)
    :=
    \frac{1}{d!}
    \sum_{\sigma\in S_d}
    x_{\sigma(1)}\otimes\cdots\otimes x_{\sigma(d)},
\]
where \(S_d\) denotes the symmetric group of all permutations of
\(\{1,\ldots,d\}\).
For \(x\in X\), we write
\[
        \otimes^d x
        :=
        \otimes^{s,d}(x,\ldots,x).
\]
By polarization over \(\mathbb C\),
\[
        \bigotimes\nolimits^{s,d}X
        =
        \operatorname{span}
        \bigl\{\otimes^d x:x\in X\bigr\}.
\]
Suppose that a group \(G\) acts linearly on \(X\). The induced action on
\(\bigotimes^{s,d}X\) is given by
\[
        g\hbullet
        \otimes^{s,d}(x_1,\ldots,x_d)
        :=
        \otimes^{s,d}
        (g\hbullet x_1,\ldots,g\hbullet x_d).
\]
The corresponding contragredient action on \(X^*\) is defined by
\[
        (g\hbullet\varphi)(x)
        :=
        \varphi(g^{-1}\hbullet x),
\]
while the induced action on \(\mathcal P_d(X)\) is
\[
        (g\hbullet p)(x)
        :=
        p(g^{-1}\hbullet x).
\]
With these conventions, the canonical map
\[
        \Psi:
        \bigotimes\nolimits^{s,d}X^*
        \longrightarrow
        \mathcal P_d(X),
        \qquad
        \Psi(\otimes^d\varphi)(x)
        :=
        \varphi(x)^d,
\]
extended linearly, is a \(G\)-equivariant isomorphism.

The same observation applies to product-group actions. Suppose that
\(G\times H\) acts linearly on \(X\). Then the induced action on
\(\mathcal P_d(X)\) is given by
\[
        ((g,h)\hbullet p)(x)
        :=
        p\bigl((g,h)^{-1}\hbullet x\bigr),
\]
and the canonical identification
\[
        \mathcal P_d(X)
        \cong
        \bigotimes\nolimits^{s,d}X^*
\]
is \(G\times H\)-equivariant.
We finally apply this tensorial description to
\[
    X=\mathcal L(\mathbb C^n).
\]
Consider the canonical linear isomorphism
\[
    (\mathbb C^n)^*\otimes\mathbb C^n
    \longrightarrow
    \mathcal L(\mathbb C^n),
    \qquad
    \varphi\otimes x
    \longmapsto
    \bigl(y\mapsto\varphi(y)x\bigr).
\]
Under the biregular action, this identification is equivariant when
\[
    (g,h)\hbullet(\varphi\otimes x)
    :=
    h^*\varphi\otimes gx,
\]
where
\(
    (h^*\varphi)(y):=\varphi(h^{-1}y).
\)
Indeed, the tensor \(h^*\varphi\otimes gx\) corresponds to the
rank-one operator
\[
    y\longmapsto \varphi(h^{-1}y)\,gx
    =
    g\bigl(\varphi(h^{-1}y)x\bigr),
\]
which is precisely \(gAh^{-1}\) when
\(A(y)=\varphi(y)x\). Thus
\[
    \mathcal L(\mathbb C^n)
    \cong
    (\mathbb C^n)^*\otimes\mathbb C^n
\]
equivariantly.
Passing to duals gives
\[
    \mathcal L(\mathbb C^n)^*
    \cong
    \mathbb C^n\otimes(\mathbb C^n)^*.
\]
Under this identification, the induced action is
\[
    (g,h)\hbullet(x\otimes\varphi)
    :=
    hx\otimes g^*\varphi.
\]
Consequently,
\begin{equation}
\label{symmetry}
    \mathcal P_d\bigl(\mathcal L(\mathbb C^n)\bigr)
    \cong
    \bigotimes\nolimits^{s,d}
    \bigl(\mathbb C^n\otimes(\mathbb C^n)^*\bigr)
\end{equation}
equivariantly, where the action on the symmetric tensor power is the
diagonal action induced by the preceding formula. This realization
will serve as the starting point for the representation-theoretic
decomposition in Theorem~\ref{thm:unitary-pw-polynomials}.

\subsubsection{Representation-theoretic setup}
\label{sec:unitary-polynomial-projection}
We begin by recalling the notation for the polynomial irreducible
representations of \(\mathcal U_n\). Since some readers from functional analysis may be less familiar with
the representation theory of the unitary group, we develop the
necessary material here in slightly greater detail.

A partition \(\lambda\vdash d\) is a
non-increasing sequence
\[
        \lambda=(\lambda_1,\lambda_2,\ldots)
\]
of non-negative integers satisfying
\[
        |\lambda|
        :=
        \sum_j\lambda_j
        =
        d.
\]
We write
\[
        (d)=(d,0,\ldots,0),
        \qquad
        (1^d)=(1,\ldots,1).
\]
The length of \(\lambda\) is
\[
        \ell(\lambda)
        :=
        \#\{j:\lambda_j\ge 1\}.
\]
For each \(r\ge 1\), let
\[
        m_r(\lambda)
        :=
        \#\{j:\lambda_j=r\}
\]
denote the multiplicity of the part \(r\). Thus \(m_r(\lambda)\) counts
the number of entries of \(\lambda\) equal to \(r\), whereas
\(\lambda_r\) denotes the \(r\)-th entry of the partition. We set
\[
        z_\lambda
        :=
        \prod_{r\ge 1}
        r^{m_r(\lambda)}m_r(\lambda)!.
\]
Since all but finitely many \(m_r(\lambda)\) vanish, this product is
finite.
For an alphabet \(X=(x_i)_i\), define the power sums
\[
        p_r(X)
        :=
        \sum_i x_i^r,
        \qquad
        p_\lambda(X)
        :=
        \prod_{j=1}^{\ell(\lambda)}
        p_{\lambda_j}(X).
\]
Later, \(X\) will be the eigenvalue alphabet of a unitary matrix. Thus, if
\(g\in\mathcal U_n\) has eigenvalues
\[
        \operatorname{eigs}(g)
        =
        (\alpha_1,\ldots,\alpha_n),
\]
then
\[
        p_r(\operatorname{eigs}(g))
        =
        \sum_{i=1}^n\alpha_i^r
        =
        \operatorname{tr}(g^r).
\]
For example, if \(\lambda=(3,2,2,1)\), then
\[
        \ell(\lambda)=4,
        \qquad
        m_1(\lambda)=1,
        \quad
        m_2(\lambda)=2,
        \quad
        m_3(\lambda)=1,
\]
and hence
\[
        z_\lambda
        =
        1^1 1!\,2^2 2!\,3^1 1!
        =
        24.
\]
Now let \(n\ge 1\), and let \(\lambda\) be a partition with
\(\ell(\lambda)\le n\). The Schur polynomial \(s_\lambda\) in \(n\)
variables is defined by the Weyl, or bialternant, formula
\[
        s_\lambda(x_1,\ldots,x_n)
        :=
        \frac{
        \det\bigl(x_i^{\lambda_j+n-j}\bigr)_{1\le i,j\le n}}
        {\det\bigl(x_i^{n-j}\bigr)_{1\le i,j\le n}}.
\]
The denominator is the Vandermonde determinant, and the quotient,
initially defined for pairwise distinct \(x_1,\ldots,x_n\), extends
uniquely to a symmetric polynomial in \(x_1,\ldots,x_n\). We set
\(s_\lambda=0\) when \(\ell(\lambda)>n\). For orientation,
\[
        s_{(1)}(x_1,\ldots,x_n)
        =
        x_1+\cdots+x_n.
\]
For \(n=2\),
\[
        s_{(2)}(x_1,x_2)
        =
        x_1^2+x_1x_2+x_2^2,
        \qquad
        s_{(1,1)}(x_1,x_2)
        =
        x_1x_2.
\]
For every partition \(\lambda\) with \(\ell(\lambda)\leq n\), let
\[
    \pi_\lambda:
    \mathcal U_n
    \longrightarrow
    \mathcal U\bigl(S_\lambda(\mathbb C^n)\bigr)
\]
denote the  irreducible unitary representation, unique up to equivalence, whose character
is the corresponding Schur polynomial:
\[
    \chi_{\pi_\lambda}(g)
    =
    s_\lambda(\operatorname{eigs}(g)),
    \qquad g\in\mathcal U_n.
\]
An irreducible unitary representation of \(\mathcal U_n\) is called
polynomial if it is equivalent to \(\pi_\lambda\) for some partition
\(\lambda\) with \(\ell(\lambda)\leq n\). Equivalently, these are
precisely the irreducible representations that extend to representations
of \(GL_n(\mathbb C)\) whose matrix coefficients are polynomials in the
matrix entries. Merely requiring an algebraic extension to
\(GL_n(\mathbb C)\) would also allow negative powers of the determinant.

Thus \(\lambda\) labels the representation, while
\(S_\lambda(\mathbb C^n)\) is its carrier space. When the action is
understood, we also refer to \(S_\lambda(\mathbb C^n)\) as the
corresponding polynomial representation. Its dimension is obtained by
evaluating the character at the identity:
\[
        \dim S_\lambda(\mathbb C^n)
        =
        \chi_{\pi_\lambda}(I_n)
        =
        s_\lambda(1^n),
        \qquad
        1^n=(1,\ldots,1).
\]
The basic examples are
\[
        S_{(1)}(\mathbb C^n)
        =
        \mathbb C^n,
        \qquad
        \chi_{\pi_{(1)}}(g)
        =
        \operatorname{tr}(g),
\]
and, more generally,
\[
        S_{(d)}(\mathbb C^n)
        =
        \bigotimes^{s,d} \mathbb C^n,
        \qquad
        S_{(1^d)}(\mathbb C^n)
        =
        \bigwedge^d(\mathbb C^n).
\]
Here
\(
    \bigotimes\nolimits^{s,d}\mathbb C^n
    =
    \operatorname{Sym}^d(\mathbb C^n)
\)
denotes the \(d\)-th symmetric tensor power of \(\mathbb C^n\)
(explained in the preceding section), while
\(    \bigwedge\nolimits^d\mathbb C^n
\)
denotes its \(d\)-th exterior power. The latter is non-zero, and gives
an irreducible representation, only for \(d\leq n\). The former is
spanned by symmetrized tensors, whereas the latter is spanned by
alternating tensors.

For every partition \(\lambda\) with \(\ell(\lambda)\leq n\), we equip \(S_\lambda(\mathbb C^n)\) with a
\(\mathcal U_n\)-invariant Hermitian inner product. Such an inner
product may be obtained by averaging any Hermitian inner product over
\(\mathcal U_n\).
More precisely, since \(S_\lambda(\mathbb C^n)\) is finite-dimensional, it may first
be equipped with an arbitrary Hermitian inner product
\(\langle\cdot,\cdot\rangle_0\). Averaging this inner product over the
compact group \(\mathcal U_n\), we define
\[
    \langle u,v\rangle_{\mathcal U_n}
    :=
    \int_{\mathcal U_n}
    \bigl\langle
        \pi_\lambda(U)u,\pi_\lambda(U)v
    \bigr\rangle_0
    \,d\mu_{\mathcal U_n}(U),
    \qquad
    u,v\in S_\lambda(\mathbb C^n).
\]
This is
again a positive-definite Hermitian inner product, and the invariance
of Haar measure gives
\[
    \bigl\langle
        \pi_\lambda(V)u,\pi_\lambda(V)v
    \bigr\rangle_{\mathcal U_n}
    =
    \langle u,v\rangle_{\mathcal U_n},
    \qquad V\in\mathcal U_n.
\]
Thus \(\pi_\lambda\) becomes a unitary representation. We shall tacitly
equip \(S_\lambda(\mathbb C^n)\) with such a
\(\mathcal U_n\)-invariant Hermitian inner product.
By irreducibility, the invariant Hermitian inner product is unique
up to multiplication by a positive scalar. We fix one such choice
throughout; its normalization will not play a role.

In the standard tensor constructions, we use the natural induced
inner products. Thus \(\mathbb C^n\) carries its usual Euclidean
Hermitian inner product, and on the $d$-th tensor power $\bigotimes\nolimits^{d}\mathbb C^n$ we set
\[
    \left\langle
        x_1\otimes\cdots\otimes x_d,\,
        y_1\otimes\cdots\otimes y_d
    \right\rangle
    :=
    \prod_{j=1}^d\langle x_j,y_j\rangle_2.
\]
Its restrictions to the symmetric and exterior powers
    $\operatorname{Sym}^d(\mathbb C^n)$
   and
    $\bigwedge\nolimits^d\mathbb C^n$
are already \(\mathcal U_n\)-invariant.

For completeness, we also recall how the polynomial representations fit
into the full unitary dual of \(\mathcal U_n\). All irreducible unitary
representations of \(\mathcal U_n\) are parametrized by dominant integral
highest weights
\[
        \mu=(\mu_1\ge\cdots\ge\mu_n)\in\mathbb Z^n.
\]
Partitions \(\lambda\) with \(\ell(\lambda)\le n\) parametrize precisely
the polynomial irreducible representations. More generally, if
\(\mu\in\mathbb Z^n\) is a dominant integral highest weight and
\[
        m:=\mu_n,
        \qquad
        \lambda_i:=\mu_i-\mu_n
        \quad(1\le i\le n),
\]
then \(\lambda\) is a partition  with $\lambda_n=0$, and the irreducible representation
corresponding to \(\mu\) can be written uniquely as
\[
        S_\lambda(\mathbb C^n)\otimes(\det)^m.
\]
Here
\[
        (\det)^m:
        \mathcal U_n\longrightarrow\mathbb C,
        \qquad
        g\longmapsto(\det g)^m,
\]
is the one-dimensional determinant character. In the present
application, the biregular decomposition of
\(\mathcal P_d(\mathcal U_n)\) is indexed by the partitions
\(\lambda\vdash d\) with \(\ell(\lambda)\leq n\); the corresponding
irreducible representation types are
\(\pi_\lambda^*\sboxtimes\pi_\lambda\).
For further background on Schur functors, Schur polynomials, and the
parametrization of polynomial irreducible representations of
\(\mathcal U_n\), we refer to \cite{fulton2013representation}.

\subsubsection{The Peter--Weyl decomposition of \(\mathcal P_d(\mathcal U_n)\)}
\label{subsub:rep-identifications}

We now identify which biregular Peter--Weyl blocks occur in
\(\mathcal P_d(\mathcal U_n)\). The answer is exactly governed by the
partitions of \(d\) introduced above.

\begin{theorem}
\label{thm:unitary-pw-polynomials}
Let \(d,n\in\mathbb N\). For every partition
\(\lambda\vdash d\) with \(\ell(\lambda)\le n\), let
\(
    \pi_\lambda:
    \mathcal U_n
    \longrightarrow
    \mathcal U\bigl(S_\lambda(\mathbb C^n)\bigr)
\)
denote the corresponding irreducible unitary representation. Then
\[
    \mathcal P_d(\mathcal U_n)
    =
    \bigoplus_{\substack{\lambda\vdash d\\ \ell(\lambda)\le n}}
    \mathcal E_{\pi_\lambda^*}(\mathcal U_n)
\]
orthogonally in \(L_2(\mathcal U_n)\). Equivalently, under the
biregular action of
\(\mathcal U_n\times\mathcal U_n\), the space
\(\mathcal P_d(\mathcal U_n)\) is the multiplicity-free direct sum of
the irreducible representation types
\(
    \pi_\lambda^*\sboxtimes\pi_\lambda
    \), where 
    \(\lambda\vdash d\) and \(\ell(\lambda)\le n\).
\end{theorem}

\begin{proof}
By the discussion in
Subsection~\ref{subsub:operators-to-unitary-group} and the equivariant
identification in \eqref{symmetry}, restriction to \(\mathcal U_n\)
gives
\[
    \mathcal P_d(\mathcal U_n)
    \cong
    \mathcal P_d\bigl(\mathcal L(\mathbb C^n)\bigr)
    \cong
    \bigotimes\nolimits^{s,d}
    \bigl(\mathbb C^n\otimes(\mathbb C^n)^*\bigr)
\]
as \(\mathcal U_n\times\mathcal U_n\)-representations. Here the
action on \(\mathbb C^n\otimes(\mathbb C^n)^*\) is given by
\[
    (g,h)\cdot(x\otimes\varphi)
    =
    hx\otimes g^*\varphi.
\]
Applying the Cauchy formula for symmetric powers (see, for example,
\cite[Exercise~6.11]{fulton2013representation}) to the present
\(\mathcal U_n\times\mathcal U_n\)-module yields
\[
    \mathcal P_d(\mathcal U_n)
    \cong
    \bigoplus_{\substack{\lambda\vdash d\\ \ell(\lambda)\le n}}
    S_\lambda(\mathbb C^n)
    \otimes
    S_\lambda(\mathbb C^n)^*
\]
as \(\mathcal U_n\times\mathcal U_n\)-representations.
On the
summand corresponding to \(\lambda\), the action is
\[
    (g,h)\cdot(u\otimes\varphi)
    =
    \pi_\lambda(h)u
    \otimes
    \pi_\lambda^*(g)\varphi,
\]
where
\(
    \pi_\lambda:
    \mathcal U_n
    \longrightarrow
    \mathcal U\bigl(S_\lambda(\mathbb C^n)\bigr)
\)
denotes the irreducible representation associated with \(\lambda\).
The first group factor acts on \(S_\lambda(\mathbb C^n)^*\), while
the second acts on \(S_\lambda(\mathbb C^n)\). Interchanging the tensor
factors therefore puts the action in the form used in
Lemma~\ref{lem:biregular-matrix-coefficients}.

We now identify this summand with the corresponding biregular
Peter--Weyl block. Fix
\(\lambda\vdash d\) with \(\ell(\lambda)\le n\), and apply
Lemma~\ref{lem:biregular-matrix-coefficients} to the irreducible
representation
\(
    \pi_\lambda^*
\)
on \(S_\lambda(\mathbb C^n)^*\). Since
\[
    \bigl(S_\lambda(\mathbb C^n)^*\bigr)^*
    \cong
    S_\lambda(\mathbb C^n),
\]
the lemma gives
\[
    \mathcal E_{\pi_\lambda^*}(\mathcal U_n)
    \cong
    S_\lambda(\mathbb C^n)^*
    \otimes
    S_\lambda(\mathbb C^n),
\]
where the \(\mathcal U_n\times\mathcal U_n\)-action is
\[
    (g,h)\cdot(\varphi\otimes u)
    =
    \pi_\lambda^*(g)\varphi
    \otimes
    \pi_\lambda(h)u.
\]
Consider the flip isomorphism
\[
    F_\lambda:
    S_\lambda(\mathbb C^n)^*
    \otimes
    S_\lambda(\mathbb C^n)
    \longrightarrow
    S_\lambda(\mathbb C^n)
    \otimes
    S_\lambda(\mathbb C^n)^*,
    \qquad
    F_\lambda(\varphi\otimes u)
    :=
    u\otimes\varphi.
\]
For \(g,h\in\mathcal U_n\),
\[
\begin{aligned}
    F_\lambda\bigl(
        (g,h)\cdot(\varphi\otimes u)
    \bigr)
    &=
    F_\lambda\bigl(
        \pi_\lambda^*(g)\varphi
        \otimes
        \pi_\lambda(h)u
    \bigr) =
    \pi_\lambda(h)u
    \otimes
    \pi_\lambda^*(g)\varphi 
    =
    (g,h)\cdot F_\lambda(\varphi\otimes u).
\end{aligned}
\]
Hence \(F_\lambda\) is
\(\mathcal U_n\times\mathcal U_n\)-equivariant. Therefore
\[
    S_\lambda(\mathbb C^n)
    \otimes
    S_\lambda(\mathbb C^n)^*
    \cong
    \mathcal E_{\pi_\lambda^*}(\mathcal U_n)
\]
as \(\mathcal U_n\times\mathcal U_n\)-representations.
Distinct partitions \(\lambda\vdash d\) yield pairwise inequivalent
irreducible representations \(\pi_\lambda^*\). Since the biregular
Peter--Weyl decomposition of \(L_2(\mathcal U_n)\) is
multiplicity-free, the corresponding blocks are mutually orthogonal
and each occurs with multiplicity one. Consequently,
\[
    \mathcal P_d(\mathcal U_n)
    =
    \bigoplus_{\substack{\lambda\vdash d\\ \ell(\lambda)\le n}}
    \mathcal E_{\pi_\lambda^*}(\mathcal U_n)
\]
orthogonally in \(L_2(\mathcal U_n)\).
\end{proof}

\subsubsection{The reproducing kernel of
\(\mathcal P_d(\mathcal U_n)\)}
\label{sec:unitary-reproducing-kernel}

We now use the preceding Peter--Weyl decomposition to obtain an
explicit formula for the reproducing kernel of
\(\mathcal P_d(\mathcal U_n)\). Let
\[
    \mathbf k_{\mathcal P_d(\mathcal U_n)}
\]
denote this kernel in \(L_2(\mathcal U_n)\), and let
\(I\in\mathcal U_n\) be the identity matrix.
By Theorem~\ref{thm:unitary-pw-polynomials} and
Theorem~\ref{thm:homogeneous-machine}, for every
\(g\in\mathcal U_n\),
\[
    \mathbf k_{\mathcal P_d(\mathcal U_n)}(I,g)
    =
    \sum_{\substack{\lambda\vdash d\\ \ell(\lambda)\le n}}
    d_{\pi_\lambda^*}\,
    \overline{\chi_{\pi_\lambda^*}(g)}
    =   \sum_{\substack{\lambda\vdash d\\ \ell(\lambda)\le n}}
    d_{\pi_\lambda}\,
    \chi_{\pi_\lambda}(g)
    =
    \sum_{\substack{\lambda\vdash d\\ \ell(\lambda)\le n}}
    \dim S_\lambda(\mathbb C^n)\,
    \chi_{\pi_\lambda}(g).
\]
Indeed,
\(
    d_{\pi_\lambda^*}=d_{\pi_\lambda}
    =\dim S_\lambda(\mathbb C^n)
\)
and
\(
    \chi_{\pi_\lambda^*}
    =
    \overline{\chi_{\pi_\lambda}}.
\)
Using
\[
        \dim S_\lambda(\mathbb C^n)=s_\lambda(1^n)
        \qquad\text{and}\qquad
        \chi_{S_\lambda(\mathbb C^n)}(g)
        =
        s_\lambda(\operatorname{eigs}(g)),
\]
this becomes
\[
        \mathbf k_{\mathcal P_d(\mathcal U_n)}(I,g)
        =
        \sum_{\substack{\lambda\vdash d\\ \ell(\lambda)\le n}}
        s_\lambda(1^n)\,
        s_\lambda(\operatorname{eigs}(g)).
\]
The next elementary symmetric-function identity rewrites this Schur sum
in the trace variables.

\begin{lemma}
\label{lem:schur-power-sum-specialization}
For every \(d,n\in\mathbb N\),
\[
        \sum_{\substack{\lambda\vdash d\\ \ell(\lambda)\le n}}
        s_\lambda(1^n)s_\lambda(x)
        =
        \sum_{\mu\vdash d}
        \frac{n^{\ell(\mu)}}{z_\mu}\,p_\mu(x).
\]
\end{lemma}

\begin{proof}
We use the two standard expansions of the Cauchy kernel,
\[
        \prod_{i,j}(1-x_i y_j)^{-1}
        =
        \sum_{\lambda}s_\lambda(x)s_\lambda(y),
\]
and
\[
        \prod_{i,j}(1-x_i y_j)^{-1}
        =
        \sum_{\mu}\frac1{z_\mu}p_\mu(x)p_\mu(y),
\]
see \cite[Ch.~I, \S4, Eqs.~(4.2), (4.4)]{Macdonald1995}. Taking the
homogeneous part of degree \(d\) and then specializing \(y=1^n\), we get
\[
        p_r(1^n)=n
        \qquad\text{and hence}\qquad
        p_\mu(1^n)=n^{\ell(\mu)}.
\]
Moreover, \(s_\lambda(1^n)=0\) whenever \(\ell(\lambda)>n\). This gives
the claimed identity.
\end{proof}

We therefore obtain the following explicit kernel formula.

\begin{theorem}
\label{thm:unitary-reproducing-kernel}
For every \(n,d\in\mathbb N\) and every \(g\in\mathcal U_n\),
\[
        \mathbf k_{\mathcal P_d(\mathcal U_n)}(I,g)
        =
        \sum_{\mu\vdash d}
        \frac{n^{\ell(\mu)}}{z_\mu}
        \prod_{j=1}^{\ell(\mu)}
        \operatorname{tr}\bigl(g^{\mu_j}\bigr).
\]
Equivalently,
\[
        \mathbf k_{\mathcal P_d(\mathcal U_n)}(I,g)
        =
        \sum_{\substack{\lambda\vdash d\\ \ell(\lambda)\le n}}
        \dim S_\lambda(\mathbb C^n)\,
        \chi_{S_\lambda(\mathbb C^n)}(g)
        =
        \sum_{\substack{\lambda\vdash d\\ \ell(\lambda)\le n}}
        s_\lambda(1^n)\,
        s_\lambda(\operatorname{eigs}(g)).
\]
\end{theorem}

\begin{proof}
The second formula was obtained above from
Theorems~\ref{thm:unitary-pw-polynomials} and ~\ref{thm:homogeneous-machine}. Then Lemma~\ref{lem:schur-power-sum-specialization}
 gives
\[
        \mathbf k_{\mathcal P_d(\mathcal U_n)}(I,g)
        =
        \sum_{\mu\vdash d}
        \frac{n^{\ell(\mu)}}{z_\mu}
        p_\mu(\operatorname{eigs}(g)).
\]
If \(x=(x_1,\ldots,x_n)=\operatorname{eigs}(g)\), then
\(
        p_r(x)=\sum_{i=1}^n x_i^r=\operatorname{tr}(g^r).
\)
Thus
\[
        p_\mu(\operatorname{eigs}(g))
        =
        \prod_{j=1}^{\ell(\mu)}
        \operatorname{tr}\bigl(g^{\mu_j}\bigr),
\]
which proves the trace formula.
\end{proof}

For orientation, the first cases are
\[
        \mathbf k_{\mathcal P_1(\mathcal U_n)}(I,g)
        =
        n\,\operatorname{tr}(g),\quad
        \mathbf k_{\mathcal P_2(\mathcal U_n)}(I,g)
        =
        \frac{n^2}{2}\operatorname{tr}(g)^2
        +
        \frac{n}{2}\operatorname{tr}(g^2),
\]
and
\[
        \mathbf k_{\mathcal P_3(\mathcal U_n)}(I,g)
        =
        \frac{n^3}{6}\operatorname{tr}(g)^3
        +
        \frac{n^2}{2}\operatorname{tr}(g)\operatorname{tr}(g^2)
        +
        \frac{n}{3}\operatorname{tr}(g^3).
\]

\subsubsection{The projection constant of
\(\mathcal P_d(\mathcal U_n)\) and its high-dimensional asymptotics}
\label{sec:unitary-high-dimensional-asymptotics}

We now combine the kernel formula with the biregular projection constant
machine. This gives an exact integral formula and, after normalization by
the square root of the dimension, a simple high-dimensional limit.

\begin{theorem}
\label{thm:unitary-projection-asymptotic}
For every \(n,d\in\mathbb N\),
\[
        \boldsymbol\lambda\bigl(\mathcal P_d(\mathcal L(\ell_2^n))\bigr)
        =
        \int_{\mathcal U_n}
        \Bigg|
        \sum_{\mu\vdash d}
        \frac{n^{\ell(\mu)}}{z_\mu}
        \prod_{j=1}^{\ell(\mu)}
        \operatorname{tr}\bigl(g^{\mu_j}\bigr)
        \Bigg|
        \,d\mu_{\mathcal U_n}(g).
\]
Moreover, for fixed \(d\),
\[
        \lim_{n\to\infty}
        \frac{
        \boldsymbol\lambda\bigl(\mathcal P_d(\mathcal L(\ell_2^n))\bigr)}
        {\sqrt{\binom{n^2+d-1}{d}}}
        =
        \frac{\Gamma\left(1+\frac d2\right)}{\sqrt{d!}}.
\]
Equivalently,
\[
        \lim_{n\to\infty}
        \frac{
        \boldsymbol\lambda\bigl(\mathcal P_d(\mathcal L(\ell_2^n))\bigr)}
        {n^d}
        =
        \frac{\Gamma\left(1+\frac d2\right)}{d!}.
        \]
\end{theorem}

\begin{proof}
The restriction map identifies
\[
        \mathcal P_d(\mathcal L(\ell_2^n))
        \quad\text{isometrically with}\quad
        \mathcal P_d(\mathcal U_n).
\]
By Theorem~\ref{thm:unitary-pw-polynomials}, the space
\(\mathcal P_d(\mathcal U_n)\) is a finite sum of biregular
Peter--Weyl blocks. Hence it is accessible, and the biregular machine
gives
\[
        \boldsymbol\lambda\bigl(\mathcal P_d(\mathcal L(\ell_2^n))\bigr)
        =
        \int_{\mathcal U_n}
        \left|
        \mathbf k_{\mathcal P_d(\mathcal U_n)}(I,g)
        \right|
        \,d\mu_{\mathcal U_n}(g).
\]
Inserting the kernel formula from
Theorem~\ref{thm:unitary-reproducing-kernel} gives the asserted integral
formula.

We now prove the asymptotic formula. Let \(U_n\) be Haar distributed on
\(\mathcal U_n\). Dividing the preceding identity by \(n^d\), we get
\[
        \frac{
        \boldsymbol\lambda\bigl(\mathcal P_d(\mathcal L(\ell_2^n))\bigr)}
        {n^d}
        =
        \mathbb E
        \Bigg|
        \sum_{\mu\vdash d}
        \frac{n^{\ell(\mu)-d}}{z_\mu}
        \prod_{j=1}^{\ell(\mu)}
        \operatorname{tr}\bigl(U_n^{\mu_j}\bigr)
        \Bigg|.
\]
The leading contribution comes from the partition
$1^d=(1,\ldots,1)$.
For this partition,
$\ell(1^d)=d$ and $z_{1^d}=d!$,
                hence its contribution is
\[
        \frac1{d!}\bigl(\operatorname{tr}U_n\bigr)^d.
\]
Thus
\[
        \sum_{\mu\vdash d}
        \frac{n^{\ell(\mu)-d}}{z_\mu}
        \prod_{j=1}^{\ell(\mu)}
        \operatorname{tr}\bigl(U_n^{\mu_j}\bigr)
        =
        \frac1{d!}\bigl(\operatorname{tr}U_n\bigr)^d
        +
        R_n,
\]
where
\[
        R_n
        :=
        \sum_{\substack{\mu\vdash d\\ \mu\neq 1^d}}
        \frac{n^{\ell(\mu)-d}}{z_\mu}
        \prod_{j=1}^{\ell(\mu)}
        \operatorname{tr}\bigl(U_n^{\mu_j}\bigr).
\]
We claim that \(R_n\to0\) in \(L^1\). For every fixed partition
\(\mu\vdash d\), the mixed-moment formula of Diaconis--Shahshahani
\cite[Theorem~2]{diaconis1994eigenvalues} gives, for all sufficiently
large \(n\),
\[
        \mathbb E
        \left|
        \prod_{j=1}^{\ell(\mu)}
        \operatorname{tr}\bigl(U_n^{\mu_j}\bigr)
        \right|^2
        =
        \prod_{r\geq1}r^{m_r(\mu)}m_r(\mu)!
        =
        z_\mu.
\]
Since \(d\) is fixed, the same threshold for \(n\) can be used for all
partitions of \(d\). By the Cauchy--Schwarz inequality and
\(\ell(\mu)\leq d-1\) for \(\mu\ne1^d\), we therefore obtain
\[
        \|R_n\|_1
        \leq
        \sum_{\substack{\mu\vdash d\\\mu\ne1^d}}
        \frac{n^{\ell(\mu)-d}}{\sqrt{z_\mu}}
        =
        O_d\left(\frac1n\right).
\]
In particular, \(\|R_n\|_1\to0\). Using
\(
        \bigl||a+b|-|a|\bigr|\leq |b|,
\)
we obtain
\[
        \left|
        \frac{
        \boldsymbol\lambda\bigl(\mathcal P_d(\mathcal L(\ell_2^n))\bigr)}
        {n^d}
        -
        \frac1{d!}
        \mathbb E
        \left|
        \operatorname{tr}(U_n)
        \right|^d
        \right|
        \longrightarrow0.
\]
By \cite[Theorem~1]{diaconis1994eigenvalues},
\[
        \operatorname{tr}(U_n)\Longrightarrow W,
\]
where \(W\) is a standard complex Gaussian random variable, normalized
by \(\mathbb E|W|^2=1\). The mixed-moment formula also gives
\[
        \mathbb E|\operatorname{tr}(U_n)|^{2d}=d!
\]
for all sufficiently large \(n\). The remaining finitely many moments
are finite because \(|\operatorname{tr}(U_n)|\leq n\). Thus
\[
        \sup_n\mathbb E|\operatorname{tr}(U_n)|^{2d}<\infty,
\]
so \(\bigl(|\operatorname{tr}(U_n)|^d\bigr)_n\) is uniformly
integrable. Therefore
\[
        \lim_{n\to\infty}
        \mathbb E
        \left|
        \operatorname{tr}(U_n)
        \right|^d
        =
        \mathbb E|W|^d.
\]
As mentioned in \eqref{eq:complex-gaussian-moments},
\[
        \mathbb E|W|^d
        =
        \Gamma\left(1+\frac d2\right).
\]
Hence
\[
        \lim_{n\to\infty}
        \frac{
        \boldsymbol\lambda\bigl(\mathcal P_d(\mathcal L(\ell_2^n))\bigr)}
        {n^d}
        =
        \frac{\Gamma\left(1+\frac d2\right)}{d!},
\]
which is the equivalent formulation. Since the operator space has
complex dimension \(n^2\), its degree-\(d\) polynomial space has
dimension
\[
        \binom{n^2+d-1}{d}.
\]
Applying \eqref{DIMO1} with \(n^2\) in place of \(n\) gives
\[
        \sqrt{\binom{n^2+d-1}{d}}
        =
        \frac{n^d}{\sqrt{d!}}
        \left(1+O_d\left(\frac1{n^2}\right)\right),
\]
which proves the normalized limit.
\end{proof}

\begin{remark} For \(d=1,2,3\), the limiting constants $\Gamma\big(1+\frac d2\big)/\sqrt{d!}$ are, respectively, \(\sqrt{\pi}/2\), \(1/\sqrt{2}\), and \(\sqrt{6\pi}/8\). \end{remark}

\section*{Acknowledgments}

The second author would like to thank Alain Valette, who approached him after
his talk at the BIRS--CMO Workshop ``Geometric Nonlinear Functional Analysis''
in Oaxaca and subsequently shared with him helpful notes on accessibility and
minimal projections. At that time, the authors were already studying
projection constants of function spaces on compact homogeneous spaces through
representation-theoretic methods, and Valette's notes helped clarify the role
of isotypic components and their relation to accessibility.

The authors used ChatGPT-5.6 Sol for language editing, stylistic revision,
and general consultation during the preparation of this manuscript.

\newcommand{\etalchar}[1]{$^{#1}$}

\end{document}